\documentclass[11pt,reqno]{amsart}
\usepackage{amssymb,amsmath,amsfonts,amsthm,enumerate,stmaryrd,tensor,mathtools,dsfont,upgreek,bbm,mathrsfs,nameref,hyperref,cleveref,microtype}
\usepackage[left=0.75 in, right=0.75 in,top=0.75 in, bottom=0.75 in]{geometry}
\numberwithin{equation}{section}
\renewcommand{\Bar}{\overline}

\newcommand{\R}{\mathbb{R}}
\newcommand{\N}{\mathbb{N}}

\newcommand{\Q}{\mathbb{Q}}
\newcommand{\C}{\mathbb{C}}

\newcommand{\imp}{\;\Rightarrow\;}
\newcommand{\m}{\mathrm}

\newcommand{\lv}{\lVert}
\newcommand{\rv}{\rVert}

\newcommand{\J}{\boldsymbol{J}}

\newcommand{\al}{\alpha}
\newcommand{\be}{\beta}

\newcommand{\ep}{\varepsilon}
\newcommand{\f}{\frac}
\newcommand{\sig}{\sigma}
\newcommand{\gam}{\gamma}
\newcommand{\del}{\delta}

\newcommand{\oo}{^\circ}
\newcommand{\pd}{\partial}

\newcommand{\grad}{\nabla}
\newcommand{\bpm}{\begin{pmatrix}}
\newcommand{\epm}{\end{pmatrix}}

\newcommand{\loc}{\m{loc}}
\renewcommand{\bar}{\overline}
\newcommand{\Le}{\mathfrak{L}}

\newcommand{\emb}{\hookrightarrow}
\newcommand{\res}{\restriction}
\renewcommand{\le}{\leqslant}
\renewcommand{\ge}{\geqslant}
\newcommand{\jump}[1]{\left\llbracket#1\right\rrbracket}
\newcommand{\tjump}[1]{\llbracket#1\rrbracket}

\newcommand{\bjump}[1]{\Big\llbracket#1\Big\rrbracket}
\newcommand{\norm}[1]{\left\lv#1\right\rv}
\newcommand{\bnorm}[1]{\Big\lv#1\Big\rv}
\newcommand{\snorm}[1]{\big\lv#1\big\rv}
\newcommand{\tnorm}[1]{\lv#1\rv}
\newcommand{\p}[1]{\left(#1\right)}
\newcommand{\bp}[1]{\Big(#1\Big)}
\renewcommand{\sp}[1]{\big(#1\big)}
\newcommand{\tp}[1]{(#1)}
\newcommand{\abs}[1]{\left|#1\right|}
\newcommand{\babs}[1]{\Big|#1\Big|}

\newcommand{\tabs}[1]{|#1|}
\renewcommand{\sb}[1]{\left[{#1}\right]}
\newcommand{\bsb}[1]{\Big[{#1}\Big]}
\newcommand{\ssb}[1]{\big[{#1}\big]}
\newcommand{\tsb}[1]{[{#1}]}
\newcommand{\cb}[1]{\left\{{#1}\right\}}

\newcommand{\bcb}[1]{\Big\{{#1}\Big\}}
\newcommand{\tcb}[1]{\{{#1}\}}
\newcommand{\br}[1]{\left\langle #1 \right\rangle}

\providecommand{\tbr}[1]{\langle #1 \rangle}
\newcommand{\ssum}[2]{{\textstyle\sum\limits_{#1}^{#2}}}

\renewcommand{\bf}[1]{\mathbf{#1}}
\newcommand{\ii}{\m{i}}

\newtheorem{prop}{Proposition}[section]
\newtheorem{thm}[prop]{Theorem}
\newtheorem{defn}[prop]{Definition}
\newtheorem{lem}[prop]{Lemma}

\newtheorem{rmk}[prop]{Remark}

\newtheorem{innercustomthm}{Theorem}
\newenvironment{customthm}[1]
  {\renewcommand\theinnercustomthm{#1}\innercustomthm}
  {\endinnercustomthm}
  
\def\XXint#1#2#3{{\setbox0=\hbox{$#1{#2#3}{\int}$ }
\vcenter{\hbox{$#2#3$ }}\kern-.6\wd0}}

\title[Traveling waves: incompressible, multilayer case]{
Traveling wave solutions to the multilayer free boundary incompressible Navier-Stokes equations
}
\author{Noah Stevenson}
\address{
Department of Mathematical Sciences\\
Carnegie Mellon University\\
Pittsburgh, PA 15213, USA
}
\email[N. Stevenson]{nwsteven@andrew.cmu.edu}
\author{Ian Tice}
\address{
Department of Mathematical Sciences\\
Carnegie Mellon University\\
Pittsburgh, PA 15213, USA
}
\email[I. Tice]{iantice@andrew.cmu.edu}
\thanks{I. Tice was supported by an NSF CAREER Grant (DMS \#1653161). N. Stevenson was supported by the summer research support provided by this grant. }

\begin{document}
\begin{abstract}
For a natural number $m \ge 2$, we study $m$
layers of finite depth, horizontally infinite, viscous, and incompressible fluid bounded below by a flat rigid bottom. Adjacent layers meet at free interface regions, and the top layer is bounded above by a free boundary as well. A uniform gravitational field, normal to the rigid bottom, acts on the fluid. We assume that the fluid mass densities are strictly decreasing from bottom to top and consider the cases with and without surface tension acting on the free surfaces. In addition to these gravity-capillary effects, we allow a force to act on the bulk and external stress tensors to act on the free interface regions.  Both of these additional forces are posited to be in traveling wave form: time-independent when  viewed in a coordinate system moving at a constant, nontrivial velocity parallel to the lower rigid boundary. Without surface tension in the case of two dimensional fluids and with all positive surface tensions in the higher dimensional cases, we prove that for each sufficiently small force and stress tuple there exists a traveling wave solution. The existence of traveling wave solutions to the one layer configuration ($m=1$) was recently established and, to the best of our knowledge, this paper is the first construction of traveling wave solutions to the incompressible Navier-Stokes equations in the $m$-layer arrangement.
\end{abstract}
\maketitle
\section{Introduction}
 
\subsection{Eulerian coordinate formulation}\label{section on Eulerian coordinate formulation}

In this paper we study traveling wave solutions to the viscous surface-internal wave problem, which describes the evolution of a finite number of layers of incompressible and viscous fluid.  We posit that the fluid layers contiguously occupy horizontally-infinite, finite-depth, and time-evolving slabs sitting atop a rigid hyperplane in ambient Euclidean space of dimension $n\in\N\setminus\cb{0,1}$ (the physically relevant dimensions are $2$ and $3$, but our analysis works more generally).  Within each layer the fluid dynamics are described by the incompressible Navier-Stokes equations, and jump conditions couple the dynamics between layers.  The multiple layers serve as a model of stratified fluids.  These occur, for example, when salinity or temperature change rapidly with respect to depth.

In order to properly state the PDEs considered in this analysis, we first set the necessary notation.  Fix the number of layers of fluid $m\in\N\setminus\cb{0,1}$. Let $\upalpha=\cb{a_\ell}_{\ell=1}^m$ be a strictly increasing sequence of positive real numbers, i.e. $0<a_1<\dots <a_m$. We refer to $\upalpha$ as the depth parameters. We associate to $\upalpha$ the collection of admissible graph interfaces, which is the subset of $m$-tuples of continuous and bounded functions
\begin{equation}
    \mathscr{A}\p{\upalpha}=\cb{\tp{\eta_\ell}_{\ell=1}^m\subset C^0_b\p{\R^{n-1}}\;:\;0<a_1+\eta_1<\cdots<a_m+\eta_m\text{ on }\R^{n-1}}.
\end{equation}
If $\upeta=\tp{\eta_\ell}_{\ell=1}^m\in\mathscr{A}\p{\upalpha}$, then for $\ell\in\cb{1,\dots,m}$ we define the slab-like domain
\begin{equation}
    \Omega_\ell[\upeta]=\begin{cases}
    \cb{\p{x,y}\in\R^{n-1}\times\R\;:\;\eta_{\ell-1}\p{x}+a_{\ell-1}<y<\eta_{\ell}\p{x}+a_\ell}&\text{when }2\le\ell\le m\\
    \cb{\p{x,y}\in\R^{n-1}\times\R\;:\;0<y<\eta_1\p{x}+a_1}&\text{when }\ell=1
    \end{cases}
\end{equation}
and the free boundaries 
\begin{equation}
    \Sigma_\ell[\upeta]=\cb{\p{x,y}\in\R^{n-1}\times\R\;:\;y=a_\ell+\eta_\ell\p{x}}.
\end{equation}
We also define the union of the slabs, entire domain, and rigid lower boundary, respectively, as
\begin{multline}\label{omega_eta notation}
    \Omega[\upeta]=\Omega_1[\upeta]\cup\cdots\cup\Omega_m[\upeta],\;\Omega^{\m{e}}[\upeta]=\cb{\p{x,y}\in\R^{n-1}\times\R\;:\;0<y<a_m+\eta_m\p{x}},\\\text{and}\;\Sigma_0 = \cb{\p{x,y}\in\R^{n-1}\times\R\;:\;y=0}.
\end{multline}
Observe that $(\Bar{\Omega[\upeta]})\oo=\Omega^{\m{e}}[\upeta]$.  We will often need to distinguish between derivatives that are parallel to $\Sigma_0$ and vertical derivatives, so we write $\nabla = (\nabla_\|, \partial_n)$.  Note that the operator $\nabla_\|$ is the full spatial gradient for the free surface functions, which have the spatial domain $\R^{n-1}$. 

Suppose that $\upeta=\tp{\eta_\ell}_{\ell=1}^m\in\mathscr{A}\p{\upalpha}$ is given and each $\eta_\ell$ is Lipschitz.  Then for $X \in H^1(\Omega[\upeta];\R^d)$, for some $d \in \N^+$, the restriction of $X$ to each $\Omega_\ell[\upeta]$ belongs to $H^1(\Omega_\ell[\upeta];\R^d)$, and so from standard trace theory we have trace operators onto the upper and lower boundaries, $\Sigma_{\ell}[\upeta]$ and $\Sigma_{\ell-1}[\upeta]$, which we denote by $\m{Tr}_{\Sigma_\ell[\upeta]}^{\uparrow}X$ and $\m{Tr}_{\Sigma_{\ell-1}[\upeta]}^{\downarrow}X$, respectively.  In turn,  for $\ell\in\cb{1,\dots,m}$ this allows us to define the interfacial jumps via 
\begin{equation}
    \jump{X}_\ell=
    \begin{cases}
    \m{Tr}_{\Sigma_\ell[\upeta]}^{\downarrow}X-\m{Tr}_{\Sigma_\ell[\upeta]}^{\uparrow}X&\text{when }1\le\ell\le m-1\\
    -\m{Tr}^{\uparrow}_{\Sigma_m[\upeta]}X&\text{when }\ell=m.
    \end{cases}
\end{equation}
Note that $\jump{X}_m$ is not really a jump, but we will employ this notation for brevity in writing PDEs throughout the paper.  If $u$ is a weakly differentiable vector field we define its symmetrized gradient as the matrix field $\mathbb{D}u=\grad u+\grad u^{\m{t}}$, where the superscript `$\m{t}$' denotes the matrix transpose. If $\upmu=\cb{\mu_\ell}_{\ell=1}^m\subset\R^+$ is a sequence of positive fluid viscosity parameters, we define the associated stress tensor as the mapping
\begin{equation}
\textstyle S^{\upmu}:L^2\tp{\Omega[\upeta]}\times H^1\tp{\Omega[\upeta];\R^n}\to L^2\tp{\Omega[\upeta];\R^{n\times n}_{\m{sym}}}\text{ via }S^{\upmu}\p{p,u}=\sum_{\ell=1}^m\mathbbm{1}_{\Omega_\ell[\upeta]}\p{pI_{n\times n}-\mu_\ell\mathbb{D}u},
\end{equation}
where $\R^{n\times n}_{\m{sym}}$ denotes the set of $n\times n$ real symmetric matrices. 

With the notation established, we are now equipped to describe the model and the equations of motion in time. At equilibrium, we posit that the layers of fluid occupy the domain $\Omega[0]$ with the $\ell^{\m{th}}$ layer occupying the region $\Omega_\ell[0]$; furthermore, when perturbed from equilibrium there are free surface functions $\upzeta\p{t,\cdot}=\tp{\zeta_\ell\p{t,\cdot}}_{\ell=1}^m\in\mathscr{A}\p{\upalpha}$ for time $t\ge0$ describing the evolution of the fluid layer domains in such a way that $\Omega[\upzeta\p{t,\cdot}]$ is the region occupied by the union of all the layers, and the $\ell^{\m{th}}$ layer occupies $\Omega_\ell[\upzeta\p{t,\cdot}]$. The fluid velocity and pressure are described by the functions $w\p{t,\cdot}:\Omega[\upzeta(t,\cdot)]\to\R^n$, $r\p{t,\cdot}:\Omega[\upzeta(t,\cdot)]\to\R$.  The density of the fluid occupying the region $\Omega_\ell[\upzeta(t,\cdot)]$ is the constant $\rho_\ell\in\R^+$, and the viscosity of this portion of the fluid is the constant $\mu_\ell \in \R^+$.  

We assume that the fluid is acted upon by the following forces. The bulk (the region of fluid occupying $\Omega[\upzeta\p{t,\cdot}]$) is acted on by: a uniform external gravitational field $-\mathfrak{g}e_n\in\R^n$, for acceleration of gravity $\mathfrak{g} \in\R^+$; and by a generic force $F\p{t,\cdot}:\Omega[\upzeta(t,\cdot)]\to\R^n$. The $\ell^{\m{th}}$ free surface is acted upon by: a force generated by an externally applied stress tensor $T_\ell:\Sigma_\ell[\upzeta(t,\cdot)]\to\R^{n\times n}_{\m{sym}}$; and a force generated by the surface itself, which is modeled in the standard way by $-\sig_\ell\m{H}\p{\zeta_\ell\p{t,\cdot}}$ for $\sig_\ell\ge0$ a surface tension coefficient and 
\begin{equation}
\m{H}(\eta_\ell) = \nabla_\| \cdot \sp{  \grad_\|\eta_\ell (1+|\grad_\|\eta_\ell|^2)^{-1/2}}    \end{equation}
the mean curvature operator. In addition, the  upper surface (the $m^{\m{th}}$ one)  is acted on by a constant external pressure $P_{\m{ext}}\in\R$.

The equations of motion are:
\begin{equation}\label{equations of motion written in Eulerian coordinates}
    \begin{cases}
        \rho_\ell\p{\pd_t+w\cdot\grad}w+\grad\cdot S^{\upmu}\p{r,w}=-\mathfrak{g}\rho_\ell e_n+F&\text{in }\Omega_\ell[\upzeta(t,\cdot)],\;\ell\in\cb{1,\dots,m}\\
        \grad\cdot u=0&\text{in }\Omega[\upzeta\p{t,\cdot}]\\
        P_\m{ext}\nu_m-S^{\upmu}\p{r,w}\nu_m-\sig_m\m{H}\p{\zeta_m}\nu_m=T_m\nu_m&\text{on }\Sigma_m[\upzeta\p{t,\cdot}]\\
        \jump{S^{\upmu}\p{r,w}}_\ell\nu_\ell-\sig_\ell\m{H}\p{\zeta_\ell}\nu_\ell=T_\ell\nu_\ell&\text{on }\Sigma_\ell[\upzeta\p{t,\cdot}],\;\ell\in\cb{1,\dots,m-1}\\
        \pd_t\zeta_\ell+w\cdot\p{\grad_\|\zeta_\ell,0}=w\cdot e_n&\text{on }\Sigma_\ell[\upzeta(t,\cdot)],\;\ell\in\cb{1,\dots,m}\\
        \jump{w}_\ell=0&\text{on }\Sigma_\ell[\upzeta(t,\cdot)],\;\ell\in\cb{1,\dots,m-1}\\
        w=0&\text{on }\Sigma_0.
    \end{cases}
\end{equation}
Here the upward pointing unit normal to the surface $\Sigma_\ell[\upzeta(t,\cdot)]$ is
\begin{equation}
    \nu_\ell=(1+|\grad_\|\zeta_\ell|^2)^{-1/2}(-\grad_\|\zeta_\ell,1),
\end{equation}
and we write $\grad\cdot S^{\upmu}(r,w)$ to mean the $n-$vector with $i^{\m{th}}$ component equal to the divergence of the $i^{\m{th}}$ row of $S^{\upmu}(r,w)$.

We briefly comment on the physics of the above system of PDEs. The first two equations of~\eqref{equations of motion written in Eulerian coordinates} are the incompressible Navier-Stokes equations. The first asserts a Newtonian balance of forces, while the second enforces that the associated flow is locally volume preserving and hence, because the density is constant in the slab domains, mass is conserved. The third and fourth equations are the dynamic boundary conditions, which are understood as force balance on the interfaces, and the fifth equation is the kinematic boundary condition, which dictates the surfaces' motion with the fluid.  The final two equations are the no-slip conditions: the Eulerian velocity vanishes on the lower rigid boundary and is continuous across the free interface regions. For a more physical description of these equations and boundary conditions we refer to Wehausen-Laitone~\cite{MR0119656}.

In this paper we construct traveling wave solutions to the system~\eqref{equations of motion written in Eulerian coordinates}. These are solutions that are time-independent when viewed in an inertial coordinate system obtained from the above Eulerian coordinates through a Galilean transformation. In order for the stationary condition to hold, the moving coordinate system must be traveling at a constant velocity parallel to $\Sigma_0$.  Up to a rigid rotation fixing the vector $e_n$, we may assume that the traveling coordinate system is moving at a constant velocity $\gam e_1$ for a signed wave speed $\gam\in\R\setminus\cb{0}$. 

In the new coordinates the stationary free surface functions are described by the unknowns $\upeta=\tp{\eta_\ell}_{\ell=1}^m\in\mathscr{A}\p{\al}$; these are related to $\upzeta$ via $\upeta\p{x-\gam te_1}=\upzeta\p{t,x}$. Next we posit that $v\p{x-t\gam e_1,y}=w\p{t,x,y}$,
\begin{equation}
    \textstyle q\p{x-t\gam e_1,y}=r\p{t,x,y}-P_{\m{ext}}-\mathfrak{g}\sum_{\ell=1}^m\mathbbm{1}_{\p{a_{\ell-1},a_\ell}}\p{y}\sb{\rho_\ell\p{a_\ell-y}+\sum_{k=\ell+1}^m\rho_k\p{a_k-a_{k-1}}},
\end{equation}
$\mathcal{F}\p{x-t\gam e_1,y}=F\p{t,x,y}$, and $\mathcal{T}_\ell\p{x-t\gam e_1}=T_\ell\p{t,x,a_\ell+\zeta_\ell\p{t,x}}$, for $t\ge0$ and  $\p{x,y}\in\R^{n-1}\times\p{0,a_m}$, where $v:\Omega[\upeta]\to\R^n$, $q:\Omega[\upeta]\to\R$, $\mathcal{F}:\Omega[\upeta]\to\R^n$, and $\mathcal{T}_\ell:\R^{n-1}\to\R^{n\times n}_{\m{sym}}$ are the stationary velocity field, renormalized pressure, external force, and external stresses, respectively. In the traveling coordinate system the PDE satisfied by the unknowns $\p{q,v,\tp{\eta_\ell}_{\ell=1}^m}$ with forcing $\p{\mathcal{F},\tp{\mathcal{T}_\ell}_{\ell=1}^m}$ is the following system:
\begin{equation}\label{equations of motion written in traveling Eulerian coordinates}
    \begin{cases}
        \rho_\ell\tsb{\p{v-\gam e_1}\cdot\grad} v+\grad\cdot S^{\upmu}\p{q,v}=\mathcal{F}&\text{in }\Omega_\ell[\upeta],\;\ell\in\cb{1,\dots,m}\\
        \grad\cdot v=0&\text{in }\Omega[\upeta]\\
        \jump{S^{\upmu}\p{q,v}}_\ell\mathcal{N}_\ell=\p{\mathfrak{g}\jump{\uprho}_\ell\eta_\ell+\sig_\ell\m{H}\p{\eta_\ell}}\mathcal{N}_\ell+\mathcal{T}_\ell\mathcal{N}_\ell&\text{on }\Sigma_\ell[\upeta],\;\ell\in\cb{1,\dots,m}\\
        -\gam\pd_1\eta_\ell+v\cdot\p{\grad_\|\eta_\ell,0}=v\cdot e_n&\text{on }\Sigma_\ell[\upeta],\;\ell\in\cb{1,\dots,m}\\
        \jump{v}_\ell=0&\text{on }\Sigma_\ell[\upeta],\;\ell\in\cb{1,\dots,m-1}\\
        u=0&\text{on }\Sigma_0.
    \end{cases}
\end{equation}
In the above we write $\mathcal{N}_\ell=\tp{-\grad_\|\eta_\ell,1}$ and $\uprho=\sum_{\ell=1}^m\mathbbm{1}_{\Omega_\ell[0]}\rho_\ell$.  Note that renormalizing the pressure in this way has the effect of shifting the gravitational term from the bulk to the interfaces.

We conclude our discussion of the model with a comment about the role of the forcing and interfacial stresses, $\p{\mathcal{F},\tp{\mathcal{T}_\ell}_{\ell=1}^m}$, appearing in \eqref{equations of motion written in traveling Eulerian coordinates}.  The simplest configuration occurs when $\mathcal{F} =0$ and $\mathcal{T}_\ell =0$ for $1\le \ell \le m-1$, but $\mathcal{T}_m = -\varphi I_{n \times n}$ for a given scalar function $\varphi: \R^{n-1} \to \R$.  In this configuration, $\varphi$ can be viewed as a spatially localized source of pressure moving with velocity $\gamma e_1$ above the fluid.  We have chosen to study the more general framework with $\p{\mathcal{F},\tp{\mathcal{T}_\ell}_{\ell=1}^m}$ in order to allow for more sources of external force and stress.

\subsection{Remarks on previous work}\label{section on previous work}

Traveling wave solutions to the equations of fluid dynamics have been a subject of intense mathematical study for more than a century, so a complete review of the literature is well beyond the scope of this paper.  The vast majority of this work has focused on inviscid models, in which the Navier-Stokes equations in \eqref{equations of motion written in traveling Eulerian coordinates} are replaced by the Euler equations.  For a thorough review of the inviscid literature, we refer to the works of Toland \cite{Toland_1996}, Groves \cite{Groves_2004}, and Strauss \cite{Strauss_2010}.

In the viscous literature there are various results on stationary solutions to the free boundary problems, which correspond to traveling waves with vanishing velocity, $\gamma =0$. For works on stationary solutions in layer geometries, we refer to Jean \cite{Jean_1980}, Pileckas \cite{Pileckas_1983,Pileckas_1984,Pileckas_2002}, Gellrich \cite{Gellrich_1993}, Nazarov-Pileckas \cite{NP_1999,NP_1999_2},  Pileckas-Zaleskis \cite{PL_2003}, and Bae-Cho \cite{BC_2000}.  Traveling wave solutions without a free boundary were constructed by Chae-Dubovski\u{\i} \cite{CD_1996} in full space and Kagei-Nishida \cite{KN_2019} as bifurcations from Poiseuille flow in rigid channels.  

To the best of our knowledge, the first construction of traveling wave solutions to the free boundary incompressible Navier-Stokes equations (system~\eqref{equations of motion written in Eulerian coordinates} for $m=1$) was only accomplished recently in the work of Leoni-Tice in~\cite{leoni2019traveling}, and there are no known results involving multiple layers.  The multilayer problem is an important variant that appears in the study of internal waves in stratified fluids.  This stratification can occur, for instance, due to changes in salinity or temperature.

As mentioned above, the system \eqref{equations of motion written in traveling Eulerian coordinates} can be used to model a source of spatially localized pressure translating above the fluid.  This configuration has been studied in recent experiments with a tube of air, translating uniformly above a wave tank, blowing onto a single layer of viscous fluid and resulting in traveling waves.  For details,  we refer to the works of Akylas-Cho-Diorio-Duncan \cite{DCDA_2011,CDAD_2011}, Masnadi-Duncan \cite{MD_2017}, and Park-Cho \cite{PC_2016,PC_2018}.

\subsection{Reformulation in an independent domain}\label{section on reformulation in a fixed domain}

The domain itself is one of the unknowns in the system~\eqref{equations of motion written in traveling Eulerian coordinates}, which presents a fundamental difficulty in producing solutions.  Following the strategy of the single-layer case from \cite{leoni2019traveling}, we overcome this obstacle with another change of coordinates and unknowns. We flatten to a domain that is independent of both time and the free surface functions, which comes at the expense of worsening the nonlinearities of the system.

We begin by defining the following family of flattening maps. Set $a_0=0$ and $\eta_0=0$. For $\ell\in\cb{1,\dots,m}$ we define the mapping $\mathfrak{F}_\ell:\Bar{\Omega_\ell[0]}\to\Bar{\Omega_\ell[\upeta]}$ with the assignment
\begin{equation}\label{calcium}
    \textstyle\mathfrak{F}_\ell\p{x,y}=\bp{x,\f{a_\ell-y}{a_\ell-a_{\ell-1}}(a_{\ell-1}+\eta_{\ell-1}\p{x})+\f{y-a_{\ell-1}}{a_{\ell}-a_{\ell-1}}(a_\ell+\eta_\ell\p{x})}
\end{equation}
for $(x,y)\in\R^{n-1}\times[a_{\ell-1},a_\ell]=\Bar{\Omega_\ell[0]}$.  First we observe that each $\mathfrak{F}_\ell$ is bijective with inverse given via
\begin{equation}\label{sodium}
    \textstyle{\mathfrak{F}_\ell}^{-1}\p{x,y}=\p{x,\f{a_\ell+\eta_\ell\p{x}-y}{a_\ell+\eta_\ell\p{x}-a_{\ell-1}-\eta_{\ell-1}\p{x}}a_{\ell-1}+\f{y-a_{\ell-1}-\eta_{\ell-1}\p{x}}{a_\ell+\eta_\ell\p{x}-a_{\ell-1}-\eta_{\ell-1}\p{x}}a_\ell}  
\end{equation}
for $\p{x,y}\in\Bar{\Omega_\ell[\upeta]}$, whenever $a_\ell-a_{\ell-1}\neq\eta_{\ell-1}-\eta_{\ell}$ pointwise. If this inequality holds, then $\mathfrak{F}_\ell$ is a homeomorphism inheriting the regularity of the tuple $\upeta$. We propose to paste these functions together to build our sought-after flattening map. That is, we define $\mathfrak{F}:\Bar{\Omega^{\m{e}}[\upeta]}\to\Bar{\Omega^{\m{e}}[0]}$ via $\mathfrak{F}=\mathfrak{F}_\ell$ on $\Bar{\Omega_\ell[0]}$. This assignment defines a homeomorphism because $\mathfrak{F}_\ell=\mathfrak{F}_{\ell-1}$ on $\Sigma_{\ell-1}[0]$ for $\ell\in\cb{2,\dots,m}$.

Provided that $\upeta$ is differentiable, for $\p{x,y}\in\Omega_\ell[0]$ we can compute the gradient 
\begin{equation}
    \grad\mathfrak{F}_\ell\p{x,y}=\bpm
    I_{(n-1)\times(n-1)}&0_{(n-1)\times 1}\\
    \f{a_\ell-y}{a_\ell-a_{\ell-1}}\grad_\|\eta_{\ell-1}\p{x}+\f{y-a_{\ell-1}}{a_\ell-a_{\ell-1}}\grad_\|\eta_\ell\p{x}&\f{a_\ell+\eta_\ell\p{x}-a_{\ell-1}-\eta_{\ell-1}\p{x}}{a_\ell-a_{\ell-1}}
    \epm,
\end{equation}
the Jacobian
\begin{equation}
    J_\ell\p{x,y}=\det\grad\mathfrak{F}_\ell(x,y)=\f{a_\ell+\eta_\ell\p{x}-a_{\ell-1}-\eta_{\ell-1}\p{x}}{a_{\ell}-a_{\ell-1}},
\end{equation}
and the geometry matrices
\begin{equation}
    \mathcal{A}_\ell\p{x,y}=\grad\mathfrak{F}\p{x,y}^{-\m{t}}=\bpm I_{(n-1)\times(n-1)}&-\f{\p{a_\ell-y}\grad_\|\eta_{\ell-1}\p{x}+\p{y-a_{\ell-1}}\grad_\|\eta_\ell\p{x}}{a_\ell+\eta_\ell\p{x}-a_{\ell-1}-\eta_{\ell-1}\p{x}}\\0_{1\times\p{n-1}}&\f{a_\ell-a_{\ell-1}}{a_\ell+\eta_\ell\p{x}-a_{\ell-1}-\eta_{\ell-1}\p{x}}\epm.
\end{equation}
We then set $\mathcal{A}:\Omega[0]\to\R^{n\times n}$ via $\mathcal{A}=\mathcal{A}_\ell$ in $\Omega_\ell[0]$, and $J:\Omega[0]\to\R$ via $J=J_\ell$ in $\Omega_\ell[0]$:
We may now reformulate~\eqref{equations of motion written in traveling Eulerian coordinates} as a quasilinear system in the fixed domain $\Omega[0]$.
\begin{equation}\label{flattened problem at the nonlinear level}
    \begin{cases}
        \rho_\ell[\p{u-\gam e_1}\cdot\mathcal{A}\grad] u+(\mathcal{A}\grad)\cdot S^{\upmu}_{\mathcal{A}}\p{p,u}=f&\text{in }\Omega_\ell[0],\;\ell\in\cb{1,\dots,m}\\
        (\mathcal{A}\grad)\cdot u=0&\text{in }\Omega[0]\\
        \jump{S_{\mathcal{A}}^{\upmu}\p{p,u}}_\ell\mathcal{N}_\ell=\p{\mathfrak{g}\jump{\uprho}_\ell\eta_\ell+\sig_\ell\m{H}\p{\eta_\ell}}\mathcal{N}_\ell+\mathcal{T}_\ell\mathcal{N}_\ell&\text{on }\Sigma_\ell[0],\;\ell\in\cb{1,\dots,m}\\
        \gam\pd_1\eta_\ell+u\cdot\mathcal{N}_\ell=0&\text{on }\Sigma_\ell[0],\;\ell\in\cb{1,\dots,m}\\
        \jump{u}_\ell=0&\text{on }\Sigma_\ell[0],\;\ell\in\cb{1,\dots,m-1}\\
        u=0&\text{on }\Sigma_0,
    \end{cases}
\end{equation}
for the flattened velocity field and pressure $u=v\circ\mathfrak{F}$ and $p=q\circ\mathfrak{F}$. In the above we have also set $f=\mathcal{F}\circ\mathfrak{F}$, allowed $\mathcal{A}$ to act on the `vector' $\grad$ by standard matrix multiplication, and introduced the operator
\begin{equation}
    \textstyle S^{\upmu}_{\mathcal{A}}\p{p,u}=\sum_{\ell=1}^m\mathbbm{1}_{\Omega_\ell[0]}\sb{pI_{n\times n}-\mu_\ell(\grad u)\mathcal{A}^{\m{t}}-\mu_\ell\mathcal{A}({\grad u}^{\m{t}})}.
\end{equation}
The $n$-vector $(\mathcal{A}\grad)\cdot S^{\upmu}_{\mathcal{A}}(p,u)$ has $i^{\m{th}}$ component equal to the $\mathcal{A}$-divergence of the $i^{\m{th}}$ row of $S^{\upmu}_{\mathcal{A}}(p,u)$.
\subsection{Statement of main results and discussion}\label{section on statement of main results and discussion}

We now give the two main results obtained from the analysis in this paper. We provide somewhat informal and abbreviated statements in order to avoid the need to introduce here some nonstandard function spaces we employ in our analysis. The proper statements are found later in the paper at the indicated theorems. The definitions of the function sets $C^k$, $C^k_b$, and $C^k_0$. can be found in Section~\ref{section on conventions of notation}.

Our first result regards the solvability of the flattened problem in~\eqref{flattened problem at the nonlinear level}: it tells us that if the strict Rayleigh-Taylor condition, $0<\rho_m<\cdots<\rho_1$,  is satisfied along with certain conditions on the dimension $n$ and the surface tensions $\tcb{\sigma_\ell}_{\ell=1}^m$, then the multilayer flattened free boundary problem~\eqref{flattened problem at the nonlinear level} is well-posed for all nontrivial wave speeds and small forcing and applied stresses.

\begin{customthm}{1}[Proved in Theorem~\ref{theorem on the solvability of the flattened problem}]\label{custom theorem 1}
Suppose that either $n=2$ and $\tcb{\sig_\ell}_{\ell=1}^m=0$ or else $n\ge 2$ and $\tcb{\sig_\ell}_{\ell=1}^m\subset\R^+$.  Let $\R\ni s>n/2$, $0<\rho_m<\cdots<\rho_1$, and $\N\ni r< s-n/2$.  Then there exist Banach spaces
\begin{multline}
        \textstyle\mathcal{X}^s\emb C^{1+r}_b(\Omega[0])\times \left[ C_b^0(\Omega^{\m{e}}[0];\R^n)\cap C^{2+r}_b(\Omega[0];\R^n) \right] \times(C_0^{3+r}(\R^{n-1}))^m\\
        \textstyle\text{and}\quad\mathcal{Z}^s\emb\R\times(C^{1+r}_0(\R^{n-1};\R^{n\times n}_{\m{sym}}))^m\times C^r_b(\Omega[0];\R^n)
\end{multline}
and open sets $\mathcal{V}_s\subset\mathcal{X}^s$ and $\mathcal{U}_s\subset\mathcal{Z}^s$ such that the following hold.
\begin{enumerate}
    \item $\tp{0,0,\tp{0}_{\ell=1}^m}\in\mathcal{V}_s$ and $(\R\setminus\cb{0})\times\cb{\tp{0}_{\ell=1}^m}\times\cb{0}\subset\mathcal{U}_s$.
    \item For each $\p{\gam,\tp{\mathcal{T}_{\ell}}_{\ell=1}^m,f}\in\mathcal{U}_s$ there exists a unique $\p{p,u,\tp{\eta_{\ell}}_{\ell=1}^m}\in\mathcal{V}_s$ that is a classical solution to~\eqref{flattened problem at the nonlinear level} with the former tuple as data. Moreover, the free surface functions obey the bound
    \begin{equation}
    \textstyle\max\{\norm{\eta_1}_{C^0_b},\dots,\norm{\eta_m}_{C^0_b}\}\le\f14\min\cb{a_1,a_2-a_1,\dots,a_m-a_{m-1}}.
    \end{equation}
    \item The mapping $\mathcal{U}_s\ni(\gam,\p{\mathcal{T}_\ell}_{\ell=1}^m,f)\mapsto\p{p,u,\p{\eta_\ell}_{\ell=1}^m}\in\mathcal{V}_s$ is smooth.
\end{enumerate}
\end{customthm}
Next, we take the solutions constructed by the previous theorem and build their associated inverse flattening maps. This process results in traveling wave solutions to the Eulerian formulation of the free boundary problem~\eqref{equations of motion written in traveling Eulerian coordinates}. 

\begin{customthm}{2}[Proved in Proposition~\ref{properties of the flatteneing map and its inverse} and Theorem~\ref{solvability of the free boundary and free interface problem}]\label{custom theorem 2}
Let $\N\ni k>n/2$, $0<\rho_m<\cdots<\rho_1$, and $\N\ni r< k-n/2$. Suppose the dimension $n\in\N \backslash \{0,1\}$ and the surface tension coefficients $\cb{\sig_\ell}_{\ell=1}^m$ be related as in Theorem~\ref{custom theorem 1}, and let  $\mathcal{U}_k$ be as in the theorem.  Then for each $(\gam,\tp{\mathcal{T}_\ell}_{\ell=1}^m,f)\in\mathcal{U}_k$ the solution $(p,u,\tp{\eta_\ell}_{\ell=1}^m)\in\mathcal{V}_k$ to~\eqref{flattened problem at the nonlinear level} provided by Theorem~\ref{custom theorem 1} satisfies the following.
\begin{enumerate}
    \item When defining the flattening map $\mathfrak{F}$ from the tuple $\upeta=\tp{\eta_\ell}_{\ell=1}^m$  as in Section~\ref{section on reformulation in a fixed domain}, the result is a bi-Lipschitz homeomorphism $\mathfrak{F}:\Bar{\Omega^{\m{e}}[0]}\to\Bar{\Omega^{\m{e}}[\upeta]}$ that is a $C^{3+r}$-diffeomorphism on the $m$ slab domains. In other words, $\mathfrak{F}$ and $\mathfrak{F}^{-1}$ are Lipschitz and satisfy the inclusions
    \begin{equation}
        \textstyle\mathfrak{F}\in C^{3+r}({\Omega[0]};{\Omega[\upeta]})\text{ and }\mathfrak{F}^{-1}\in C^{3+r}({\Omega[\upeta]};{\Omega[0]}).
    \end{equation}
    \item Setting 
    \begin{multline}
        \textstyle(q,v,\tp{\eta_\ell}_{\ell=1}^{m})=(p\circ\mathfrak{F}^{-1},u\circ\mathfrak{F}^{-1},\tp{\eta_\ell}_{\ell=1}^m) \\
        \in \textstyle C^{1+r}_b(\Omega^{\m{e}}[\upeta]) \times [C^0_b(\Omega[\upeta];\R^n)\cap C^{2+r}_b(\Omega[\upeta];\R^n)] \times(C^{3+r}_0(\R^{n-1}))^m
    \end{multline}
    gives a classical solution to the free boundary problem~\eqref{equations of motion written in traveling Eulerian coordinates} with signed wave speed $\gam\in\R\setminus\cb{0}$, applied surface stresses $\tp{\mathcal{T}_\ell}_{\ell=1}^m\subset C^{1+r}_0(\R^{n-1};\R^{n\times n}_{\m{sym}})$, and external forcing $\mathcal{F}=f\circ\mathfrak{F}^{-1}\in C^r_b(\Omega[\upeta];\R^n)$.
\end{enumerate}
\end{customthm}

Following the lead of the single layer analysis in~\cite{leoni2019traveling}, our strategy for proving Theorems~\ref{custom theorem 1} and~\ref{custom theorem 2} can be succinctly described as follows: we find appropriate Banach spaces such that the locally defined mapping associated to the flattened problem in system~\eqref{flattened problem at the nonlinear level} is well-defined, smooth, and satisfies the hypotheses of the implicit function theorem around the zero solution. This grants us the small data solution operator described in the first theorem. From these solutions to the flattened problem, we use the free surface functions to build the flattening map and its inverse to undo the reformulation described in Section~\ref{section on reformulation in a fixed domain}. This then yields the second theorem.

The only serious difficulties in progressing from Theorem~\ref{custom theorem 1} to Theorem~\ref{custom theorem 2} lie in verifying that the flattening map $\mathfrak{F}$ and its inverse preserve not only the standard Sobolev spaces, but the specialized ones we employ in our analysis.  Fortunately, these difficulties were already overcome in the single layer analysis of~\cite{leoni2019traveling}, and the solution is readily ported to the multilayer context of the present paper.  As such, the main thrust of this paper is proving Theorem~\ref{custom theorem 1}, which presents a number of nontrivial difficulties not encountered in the single layer analysis.  The remainder of this discussion describes the path to this theorem in greater detail.

To invoke the implicit function theorem, we are led to study the linearization of system~\eqref{flattened problem at the nonlinear level}, which is recorded in~\eqref{multilayer traveling Stokes with gravity-capillary boundary and jump conditions}. Even though this is a linear PDE, there are several obstacles that make solving it both an interesting and nontrivial endeavor. The first of these is the selection of appropriate Banach spaces for data and solutions for the linearized flattened problem. These spaces need be chosen so that: 1) the nonlinear operator associated to the flattened system is locally well-defined near the zero solution and is at least continuously differentiable, 2) they embed within subspaces of the classical scales measuring differentiability, and 3) the linearized problem induces a Banach isomorphism. The first and last point ensure that the hypothesis of the implicit function theorem are satisfied, and the second point guarantees that our notion of solution to the nonlinear flattened problem will be the classical one.

Unfortunately, for data belonging to subspaces of standard $L^2$-based Sobolev spaces, the natural a priori estimates associated to the linearized PDE~\eqref{multilayer traveling Stokes with gravity-capillary boundary and jump conditions} are too weak to force the solution tuple $\p{p,u,\p{\eta_\ell}_{\ell=1}^m}$ to belong to standard Sobolev spaces. The same problem was encountered in the single layer problem; to circumvent the issue, in Section 5 of~\cite{leoni2019traveling} novel scales of specialized Sobolev spaces were introduced, which satisfy the three requirements mentioned above.  Fortunately for us, we find that the appropriate Banach spaces for the multilayer problem are natural modifications of the single-layer problem's spaces: see Definitions \ref{space for the free interface and free surface functions} and \ref{container space for the pressure}.  It is worth pointing out that, while these spaces arise naturally as the spaces that contain the solutions to~\eqref{multilayer traveling Stokes with gravity-capillary boundary and jump conditions}, they have rather odd properties.  For instance, in the most physically important case of $n=3$ the space for the free surface functions is strongly anisotropic in the sense that it is not closed under composition with rotations (see Remark 5.4 in \cite{leoni2019traveling}).

We now turn to the question of how to solve the problem~\eqref{multilayer traveling Stokes with gravity-capillary boundary and jump conditions}, which is not a standard elliptic boundary value system (i.e. not in the form studied in the classic paper of Agmon, Douglis, and Nirenberg \cite{ADN_1964}) due to the fact that some of the unknowns, namely $\p{\eta_\ell}_{\ell=1}^m$, appear only on the boundary.  Building on the strategy of \cite{leoni2019traveling}, we attack this problem with the help of the normal stress to normal Dirichlet map (see Definition~\ref{definition of the normal stress to normal dirichlet psido}), which is $\p{\psi_\ell}_{\ell=1}^m \mapsto \upnu_\gamma\tp{\psi_\ell}_{\ell=1}^m=\tp{\m{Tr}_{\Sigma_\ell}v \cdot e_n}_{\ell=1}^m$, where $(q,v)$ solve~\eqref{normal_stress pde}.  Then a solution $(p,u,\p{\eta_\ell}_{\ell=1}^m)$ will take the form $p =- \mathfrak{g}\sum_{\ell=1}^m\jump{\uprho}_\ell\eta_\ell\mathbbm{1}_{\p{0,a_\ell}} + q + r$ and $u = v + w$ for $(q,v)$ solving~\eqref{normal_stress pde} with data  $\p{\psi_\ell}_{\ell=1}^m = \tp{\sigma_\ell \Delta_\| \eta_\ell}_{\ell=1}^m$ and $(r,w)$ solving 
\begin{equation}\label{intro wr reformulate}
    \begin{cases}
        \grad\cdot S^{\upmu}\p{r,w}-\gam\rho_\ell\pd_1 w= f+\mathfrak{g}\sum_{\ell=1}^m\jump{\uprho}_\ell\grad\eta_\ell\mathbbm{1}_{\p{0,a_\ell}} &\text{in }\Omega_\ell,\;\ell\in\cb{1,\dots,m}\\
        \grad\cdot w=g&\text{in }\Omega\\
        \jump{S^{\upmu}\p{r,w}e_n}_\ell=k_\ell &\text{on }\Sigma_\ell,\;\ell\in\cb{1,\dots,m}\\
        w\cdot e_n=h_\ell-\gam\pd_1\eta_\ell - [\upnu_{-\gamma}\tp{\sig_k\Delta_\|\eta_k }_{k=1}^m]_\ell  &\text{on }\Sigma_\ell,\;\ell\in\cb{1,\dots,m}\\
        \jump{w}_{\ell}=0&\text{on }\Sigma_\ell,\;\ell\in\cb{1,\dots,m-1}\\
        w=0&\text{on }\Sigma_0.
    \end{cases}
\end{equation}
At first glance this seems no better than~\eqref{multilayer traveling Stokes with gravity-capillary boundary and jump conditions}, but the advantage of this form is that even for given  $\upeta = \p{\eta_\ell}_{\ell=1}^m$ belonging to the specialized Sobolev spaces, the right sides of this system belong to standard Sobolev spaces (see Proposition~\ref{linear topological properties of container for free surface and free interface functions}).  However, if we think of $\upeta$ as given, then this system is overdetermined in the sense that $n+1$ scalar boundary conditions are specified at each $\Sigma_\ell$ rather than the $n$ needed to uniquely determine solutions.  This leads us to study the overdetermined problem~\eqref{overdetermined multilayer traveling stokes}.

The problem~\eqref{overdetermined multilayer traveling stokes} cannot be solved for arbitrary data tuples $\p{g,f,\tp{k_\ell}_{\ell=1}^m,\tp{h_\ell}_{\ell=1}^m}$.  Indeed, data for which a solution exists must satisfy certain compatibility conditions, which we identify in Section~\ref{section on data compatibility and associated isomorphism}.  Remarkably, this then yields a mechanism for solving~\eqref{intro wr reformulate} for general data $\p{g,f,\tp{k_\ell}_{\ell=1}^m,\tp{h_\ell}_{\ell=1}^m}$: we solve for $\upeta$ so that the modified data tuple 
\begin{equation}\label{modified data tuple}
    \bp{g,f+\mathfrak{g}\ssum{\ell=1}{m}\jump{\uprho}_\ell\grad\eta_\ell\mathbbm{1}_{\p{0,a_\ell}},\p{k_\ell}_{\ell=1}^m, \tp{h_\ell-\gam\pd_1\eta_\ell}_{\ell=1}^m -\upnu_{-\gamma}\tp{\sig_\ell\Delta_\|\eta_\ell }_{\ell=1}^m}
\end{equation}
satisfies the compatibility conditions, and then we solve for $(r,w)$ using the solvability theory for the overdetermined problem~\eqref{overdetermined multilayer traveling stokes}, which we also develop in Section~\ref{section on data compatibility and associated isomorphism}. 

In following this strategy for determining $\upeta$ in terms of the data, we uncover another remarkable fact: after horizontal Fourier transformation, the bulk term $\mathfrak{g}\sum_{\ell=1}^{m}\jump{\uprho}_\ell\grad\eta_\ell\mathbbm{1}_{\p{0,a_\ell}}$ in the compatibility condition shifts back to a boundary term involving the symbol of the pseudodifferential operator $(\Psi$DO) $\upnu_{-\gamma}$, which allows us to show that the compatibility condition for the modified data tuple~\eqref{modified data tuple} is equivalent to a system of pseudodifferential equations ($\Psi$DEs) on $\R^{n-1}$.  More precisely, we show in Proposition~\ref{diagonalization of normal stress to normal dirichlet} that $\upnu_\gamma$ has associated symbol $\bf{n}_\gam:\R^{n-1}\to\C^{m\times m}$, and we prove that the compatibility conditions are equivalent to the  $\Psi$DEs 
\begin{equation}\label{intro psiDE}
\bf{p}_\gam(\xi) \mathscr{F}[\upeta](\xi)= \mathscr{F}[\upvarphi](\xi) \text{ for } \xi \in \R^{n-1},   
\end{equation}
where $\mathscr{F}$ denotes the Fourier transform on $\R^{n-1}$ (acting on each component of the tuple $\upeta$ in the obvious way), 
\begin{equation}
\bf{p}_\gam\tp{\xi}=\bf{n}_{-\gam}\p{\xi} \m{diag}\tp{-\mathfrak{g}\jump{\uprho}_1  + 4\pi^2\abs{\xi}^2 \sigma_1,\dots, -\mathfrak{g}\jump{\uprho}_m  + 4\pi^2\abs{\xi}^2 \sigma_m }
  - 2\pi\ii\gam\xi_1I_{m\times m} \in \mathbb{C}^{m \times m},
\end{equation}
and  $\upvarphi:\R^{n-1}\to\mathbb{C}^m$ is a particular tuple depending on the data $\p{g,f,\tp{k_\ell}_{\ell=1}^m,\tp{h_\ell}_{\ell=1}^m}$.  Note that the symbol $\bf{p}_\gam$ is a synthesis of the symbols for the differential operator $\gamma \partial_1$, the normal stress to normal Dirichlet operator $\upnu_{-\gamma}$, and the elliptic capillary operators $\mathfrak{g}\jump{\uprho}_\ell+\sig_\ell\Delta_\|$.

Provided that $\bf{p}_\gam$ is almost everywhere invertible, we then have the determination $\upeta = \mathscr{F}^{-1}[{\bf{p}_\gam}^{-1}\mathscr{F}[\upvarphi]]$.  However, given the complicated form of $\bf{p}_\gamma$, it is far from obvious that this holds or that, if it is true, the resulting formula for $\upeta$ produces free surfaces that are both physically sensible and mathematically useful in our implicit function theorem scheme.  In order to prove these, we need to know two crucial pieces of information: detailed facts about the regularity of $\upvarphi$, and precise asymptotic developments of $\bf{n}_\gam(\xi)$ as $\abs{\xi}\to0$ and $\abs{\xi}\to\infty$. 

It is here where the present paper seriously deviates from the strategy employed for a single layer in \cite{leoni2019traveling}, which involved brute forcing the asymptotics of the symbol from an explicit expression given by the solution to the ODE system resulting from applying $\mathscr{F}$ to~\eqref{normal_stress pde}.  Due to essential singularities in the symbol at $\abs{\xi}= \infty$, this approach is rather delicate and involves numerous tedious calculations for which computer algebra systems are of little assistance.  If we were to attempt to port this brute force  approach to the $m$-layer problem, the number of these tedious asymptotic developments that we would need to compute by hand would be on the order of $m^2$, which is already disagreeable when $m=2$ and is outright impossible in the general case $m \ge 2$.

In the present paper we thus abandon the brute force strategy and develop a more elegant and flexible method for deriving the asymptotic developments of the symbol $\bf{n}_\gamma$.  Our technique is based on a synthesis of novel energy estimates for solutions of the multilayer traveling Stokes system~\eqref{multilayer traveling stokes with stress boundary conditions}, a duality-based formulation of the compatibility conditions for~\eqref{overdetermined multilayer traveling stokes}, and estimates for solutions to certain prescribed divergence equations.  The key observation is the energy equivalence of Theorem~\ref{energy estimates for the normal stress problem} for solutions to the applied normal stress problem~\eqref{normal_stress pde}, which characterizes the data space for solenoidal weak formulations (only employing solenoidal test functions to avoid introducing the pressure), and may thus be of independent interest in the study of the Stokes system.

The symbol ${\bf{p}_\gamma}^{-1}$, together with the properties of $\upvarphi$ (see Section~\ref{section on measuring data compatibility}), ultimately determine the nonstandard Sobolev spaces employed in our analysis.  Thus, by employing this strategy, we can indeed solve for the free surface functions and then solve the linearized problem~\eqref{multilayer traveling Stokes with gravity-capillary boundary and jump conditions}.  This leads us to the isomorphism theorems Theorems~\ref{isomorphism associated to multilayer traveling Stokes with gravity-capillary boundary and jump conditions in the surface tension case} and~\ref{isomorphism associated to multilayer traveling Stokes with gravity-capillary boundary and jump conditions in the zero surface tension case}, which then form the backbone of the implicit function scheme discussed above.

\subsection{Outline of paper}

We begin our linear analysis in Section~\ref{section on multilayer traveling Stokes with stress boundary and jump conditions}, where we study the multilayer traveling Stokes equations subject to stress boundary conditions, as well as the specified divergence and normal trace problem. These are systems~\eqref{multilayer traveling stokes with stress boundary conditions} and~\eqref{multi-normal trace-divergence problem}, respectively. The analysis of the latter PDE in Section~\ref{section on divergence problem} explores a necessary and sufficient compatibility condition for the data. The result is a solution operator and important technical estimates.

Section~\ref{subsection on isomorphism associated to multilayer traveling Stokes} dives into the analysis of the system~\eqref{multilayer traveling stokes with stress boundary conditions}.  This system is elliptic, and the well-posedness theory is straightforward.  However, the solution operator for this system plays a foundational role in our subsequent analysis, as they allow us to define the normal stress to normal Dirichlet $\Psi$DO, $\upnu_{\gamma}$, as well as build more complicated solution operators to other PDE systems.

Section~\ref{section on analysis of the normal stress to normal Dirichlet pseudodifferential operator} next studies the normal stress to normal Dirichlet operator. The asymptotic developments of its symbol are computed using the energy structure of the multilayer traveling Stokes system and estimates from the specified divergence and normal trace problem.

In Section~\ref{section on overdetermined multilayer traveling Stokes} we analyze the overdetermined  variant of the multilayer traveling Stokes system.  In Section~\ref{section on data compatibility and associated isomorphism} we characterize spaces of compatible data for which this PDE admits solutions. Then, in Section~\ref{section on measuring data compatibility}, we examine more closely what it means for data to be compatible and develop a particular measurement of compatibility for general data, which leads us to the tuple $\upvarphi$ appearing in the $\Psi$DEs~\eqref{intro psiDE}.  We prove estimates for $\upvarphi$ in frequency space that aid in the solving of these $\Psi$DEs.

Section~\ref{section on linearized flattened problem} synthesizes the previous two sections and draws from the specialized Sobolev space analysis of~\cite{leoni2019traveling} to build the Banach isomorphism solution operator associated with linearized flattened problem. Section~\ref{section on specialized Sobolev space interlude, well-definedness, and injectivity} proves that the proposed solution operator is well-defined and injective. Sections~\ref{section on isomorphism in the case with surface tension} and~\ref{section on isomorphism in the case without surface tension} prove surjectivity in the cases $n\ge 2$ and strictly positive surface tensions and $n=2$ and vanishing surface tensions, respectively. 

Section~\ref{section on nonlinear analysis} contains the  nonlinear analysis and the proofs of the main theorems. We combine the linear analysis from Section~\ref{section on linearized flattened problem} with more results on specialized Sobolev spaces in order to satisfy the hypothesis of the implicit function theorem. Theorems~\ref{custom theorem 1} and~\ref{custom theorem 2} then follow.

 Finally, in Appendix~\ref{section on tools from analysis}, we record some useful facts from analysis used throughout the paper. These include notions of real valued tempered distributions, (anti-)duality and the Lax-Milgram lemma, tangential Fourier multipliers, and Korn's inequality.

\subsection{Conventions of notation}\label{section on conventions of notation}
The standard Lebesgue measure on the Euclidean space $\R^d$ is $\Le^d$.  The symbol $\mathbb{K}$ will be used in situations in which both $\mathbb{K} = \R$ and $\mathbb{K} = \C$ are valid.

Whenever the expression $a \lesssim b$ appears in a proof of a result, it means that there is a constant $C\in\R^+$, depending only on the parameters implicitly and explicitly quantified in the statement of the result, such that $a\le Cb$. We also write $a \asymp b$ to mean $a \lesssim b$ and $b \lesssim a$.

Given complex vector spaces $X$, $Y$, and $Z$ we say that a mapping $B:X\times Y\to Z$ is sesquilinear if it is linear in the left argument and antilinear in the right argument. The dot product $\cdot$ denotes the standard sesquilinear Euclidean inner product on $\C^d$, and we write $:$ for the sesquilinear Frobenius inner product on $\C^{d\times d}$.  We denote the divergence and tangential divergence operators with
\begin{equation}
    \grad\cdot f=\ssum{j=1}{n}\pd_j(f\cdot e_j)\text{ and }(\grad_\|,0)\cdot f=\ssum{k=1}{n-1}\pd_k(f\cdot e_k)
\end{equation}
for appropriate $\C^n$-valued functions $f$.  Note that this does not violate our sesquilinearity rule because the arguments of $\grad\cdot f$ and $(\grad_\|,0)\cdot f$ are outside of the domain of the dot product.

If $\mathcal{H}$ is a complex Hilbert space, then $\mathcal{H}^{\Bar{\ast}}$ denotes the set of continuous and antilinear functionals on $\mathcal{H}$, i.e. the antidual. Sometimes we will need to simultaneously refer to complex and real Hilbert spaces. When doing so we will use $\mathcal{H}^{\Bar{\ast}}$ to refer to the usual dual space when the base field is $\R$ and the antidual when it's $\C$. Given a complex Hilbert $\mathcal{H}$, the antidual pairing is the sesquilinear form $\br{\cdot,\cdot}_{\mathcal{H}^{\Bar{\ast}},\mathcal{H}}:\mathcal{H}^{\Bar{\ast}}\times\mathcal{H}\to\C$ defined via $\br{F,v}_{\mathcal{H}^{\Bar{\ast}},\mathcal{H}}=F\p{v}$. The Fourier transform is denoted $\mathscr{F}[\cdot]$.

Finally, we set the following function space notation. If $U\subset\R^{d_1}$ and $V\subset\R^{d_2}$ are open subsets of Euclidean space and $r\in\N$ we define
\begin{equation}
    \begin{cases}
    C^r(U;V)=\tcb{f:U\to V\;|\;f\text{ is continuous along with its derivatives of order $k$,}\;\forall\;k\in\N^+,\;k\le r}\\
    C^r_b(U;V)=\tcb{f\in C^r(U;V)\;|\;\max_{0\le k\le r}\sup_{z\in U}|D^kf(z)|<\infty}\\
    C^r_0(\R^{d_1};\R^{d_2})=\tcb{f\in C^r_b(\R^{d_1};\R^{d_2})\;|\;\lim_{\abs{z}\to\infty}\max_{0\le k\le r}\abs{D^kf(z)}=0}.
    \end{cases}
\end{equation}
Let $\upeta\in\mathscr{A}(\upalpha)$, $s\ge0$, and $\mathbb{K}\in\tcb{\R,\C}$. If $t,a\in\R$ we identify
\begin{equation}
    H^t(\R^{n-1}\times\tcb{a};\mathbb{K}^d)\simeq H^t(\R^{n-1};\mathbb{K}^d)
\end{equation}
in the obvious way. We say that a vector field is solenoidal if its distributional divergence vanishes. The closed subspace of $H^1$ consisting of solenoidal fields on $\Omega^{\m{e}}[\upeta]$ vanishing identically on the lower boundary is denoted
\begin{equation}
    {_0}H^1_{\sig}(\Omega[\upeta];\mathbb{K}^n)=\tcb{f\in H^1(\Omega^{\m{e}}[\upeta];\mathbb{K}^n)\;:\;\grad\cdot f=0\text{ and }\m{Tr}_{\Sigma_0}f=0}.
\end{equation}
Note that functions in this space are required to be in $H^1$ on the entire domain $\Omega^{\m{e}}[\upeta]$ (see \eqref{omega_eta notation}).  For $s \in \R^+ \cup \{0\}$ we also define
\begin{equation}
    {_0}H^{1+s}(\Omega[\upeta];\mathbb{K}^d)=\tcb{f\in H^1(\Omega^{\m{e}}[\upeta];\mathbb{K}^d)\;:\;f\res\Omega[\upeta]\in H^s(\Omega[\upeta];\mathbb{K}^d)\text{ and }\m{Tr}_{\Sigma_0}f=0}.
\end{equation}
Note that functions in this space are required to be in $H^1$ of the entire domain but only  $H^{1+s}$ on each subdomain $\Omega_\ell[\upeta]$.  A norm that makes the above vector space Banach is given by
\begin{equation}
    \norm{f}_{{_0}H^{1+s}(\Omega[\upeta])}^2=\ssum{\ell=1}{m}\norm{f\res\Omega_\ell[\upeta]}_{H^{1+s}(\Omega_\ell[\upeta])}^2.
\end{equation}
Observe in particular that taking $s=0$ implies that we will also denote ${_0}H^1(\Omega^{\m{e}}[\upeta];\mathbb{K}^d)$ with ${_0}H^1(\Omega[\upeta];\mathbb{K}^d)$.

\section{Multilayer traveling Stokes with stress boundary and jump conditions}\label{section on multilayer traveling Stokes with stress boundary and jump conditions}
In this section and the two succeeding we analyze linear systems of PDEs in the fixed domains $\Omega_\ell[0]$ with boundary conditions prescribed on $\Sigma_j[0]$ for $\ell\in\cb{1,\dots,m}$ and $j\in\cb{0,1,\dots,m}$. In the interest of concision we make the following change of notation: $\Omega_\ell[0]\mapsto\Omega_\ell$, $\Sigma_j[0]\mapsto\Sigma_j$, and $\Omega[0]\mapsto\Omega$.

Specific to this section of the paper is analysis of the following system of PDEs:
\begin{equation}\label{multilayer traveling stokes with stress boundary conditions}
    \begin{cases}
    \grad\cdot S^{\upmu}\p{p,u}-\gamma\rho_\ell\pd_1u=f&\text{in }\Omega_\ell,\;\ell\in\cb{1,\dots,m}\\
    \grad\cdot u=g&\text{in }\Omega\\
    \jump{S^{\upmu}\p{p,u}e_n}_\ell=k_\ell&\text{on }\Sigma_\ell,\;\ell\in\cb{1,\dots,m}\\
    \jump{u}_{\ell}=0&\text{on }\Sigma_\ell,\;\ell\in\cb{1,\dots,m-1}\\
    u=0&\text{on }\Sigma_0,
    \end{cases}
\end{equation}
with unknown velocity $u$ and pressure $p$, and with prescribed data $f$, $g$, and $\tp{k_\ell}_{\ell=1}^m$. The viscosity parameters are $\upmu=\cb{\mu_\ell}_{\ell=1}^m\subset\R^+$, $\cb{\rho_\ell}_{\ell=1}^m\subset\R^+$ are the density parameters, and $\gam\in\R$ is the signed wave speed.  Out of necessity, in this section we will work with real and complex valued solutions.  We recall from Section \ref{section on conventions of notation} that $\mathbb{K} \in \{\R,\C\}$ and
in the complex case the symbols $\cdot$ and $:$ are sesquilinear, which allows us to suppress writing complex conjugates in many expressions.

\subsection{Specified divergence and multi-normal trace problem}\label{section on divergence problem}

Before we dive into the analysis of~\eqref{multilayer traveling stokes with stress boundary conditions}, we first develop a few auxiliary results concerning the following multi-normal trace-divergence problem. That is, given a collection of normal traces $\tp{g_\ell}_{\ell=1}^m$, with $g_\ell$ defined on $\Sigma_\ell$, and given $f:\Omega\to\mathbb{K}$ one asks for $u:\Omega\to\mathbb{K}^n$ satisfying:
\begin{equation}\label{multi-normal trace-divergence problem}
    \begin{cases}
    \grad\cdot u=f&\text{in }\Omega\\
    u\cdot e_n=g_\ell&\text{on }\Sigma_\ell,\;\ell\in\cb{1,\dots,m}\\
    u=0&\text{on }\Sigma_0.
    \end{cases}
\end{equation}
In general this problem is overdetermined in the sense that if $u$ is a solution belonging to a appropriate function space then a nontrivial compatibility condition must hold among the data $f$ and $\tp{g_\ell}_{\ell=1}^m$. We codify this precisely in the following result.
\begin{prop}[Divergence compatibility estimate]\label{divergence_divergence}
    Let $u\in{_0}H^1\p{\Omega;\mathbb{K}^n}$ and set $f=\grad\cdot u\in L^2\p{\Omega;\mathbb{K}}$ and $g_\ell=\m{Tr}_{\Sigma_\ell}u\cdot e_n\in H^{1/2}\p{\Sigma_\ell;\mathbb{K}}$. Then for each $\ell\in\cb{1,\dots,m}$ we have the inclusion
\begin{equation}
    g_\ell\p{\cdot,a_\ell}-\int_{\p{0,a_\ell}}f\p{\cdot,y}\;\m{d}y\in\dot{H}^{-1}\p{\R^{n-1};\mathbb{K}}.
\end{equation}
Moreover, we have the bound
\begin{equation}
    \ssum{\ell=1}{m}\bsb{g_\ell-\int_{\p{0,a_\ell}}f}_{\dot{H}^{-1}}\le2\pi\bp{\ssum{\ell=1}{m}\sqrt{a_\ell}}\norm{u}_{L^2}.
\end{equation}
\end{prop}
\begin{proof}
As justified by the absolute continuity on lines characterization of ${_0}H^1\p{\Omega;\mathbb{K}^n}$, we may integrate the equation $\grad\cdot u=f$ in the vertical variable over $\p{0,a_\ell}$ and employ the second fundamental theorem of calculus. This results in the identity
\begin{equation}
    \int_{\p{0,a_\ell}}f=g_\ell+(\grad_{\|},0)\cdot \int_{\p{0,a_\ell}}u.
\end{equation}
Therefore, by H\"older's inequality and Tonelli's theorem,
\begin{equation}
\bsb{g_\ell-\int_{\p{0,a_\ell}}f}_{\dot{H}^{-1}}^2\le4\pi^2\int_{\R^{n-1}}\babs{\int_{\p{0,a_\ell}}u}^2\le 4\pi^2a_\ell\norm{u}^2_{L^2}.
\end{equation}
The stated estimate follows.
\end{proof}
The remainder of this subsection is devoted to the converse of the previous lemma: the satisfaction of this compatibility condition is also sufficient in guaranteeing the solvability of the PDE~\eqref{multi-normal trace-divergence problem}. The first ingredient we require is some right inverse to the divergence operator that enforces the vanishing trace on the lower boundary $\Sigma_0$.
\begin{lem}[A right inverse to the divergence]\label{simple right inverse to the divergence}
Let $a,b\in\R$ with $a<b$ and set $U=\R^{n-1}\times\p{a,b}$. There exists a linear and continuous  mapping $\Pi_U:L^2\p{U;\mathbb{K}}\to{_0}H^1\p{U;\mathbb{K}^n}$ such that $\grad\cdot \Pi_Uf=f$ for all $f\in L^2\p{U;\mathbb{K}}$. 
\end{lem}
\begin{proof}
The existence of such an operator in the case $\mathbb{K}=\R$ is well-known. See, for instance, Proposition 2.1 in~\cite{leoni2019traveling}. In the instance that $\mathbb{K}=\C$ one may simply take the real valued operator to act on real and imaginary parts of the data separately in the obvious way.
\end{proof}
Next, we may explicitly construct a solution operator to~\eqref{multi-normal trace-divergence problem} in the special case of $f=0$ and $m=1$.
\begin{lem}[Solenoidal extension operator]\label{simple solution operator to theone layer solenoidal problem}
Let $W=\R^{n-1}\times\p{a,b}$ and $\Sigma=\R^{n-1}\times\cb{b}$ for $a,b\in\R$, $a<b$. There exists a bounded linear operator $P_W:H^{1/2}\p{\Sigma;\mathbb{K}}\cap\dot{H}^{-1}\p{\Sigma}\to{_0}H^1\p{W;\mathbb{K}^n}$ such that $\grad\cdot P_Wg=0$ and $\m{Tr}_{\Sigma}P_Wg\cdot e_n=g$ for all $g\in H^{1/2}\p{\Sigma;\mathbb{K}}\cap\dot{H}^{-1}\p{\Sigma;\mathbb{K}}$.
\end{lem}
\begin{proof}
It is sufficient to consider the case that $a=0$ and $b\in\R^+$. We explicitly construct the solution operator with the horizontal Fourier transform. Given $g\in\dot{H}^{-1}\p{\Sigma;\mathbb{K}}\cap H^{1/2}\p{\Sigma;\mathbb{K}}$ we define the auxiliary functions $v:\R^{n-1}\times\p{0,b}\to\C^{n-1}$ and $w:\R^{n-1}\times\p{0,b}\to\C$ via 
\begin{equation}
    v\p{\xi,t}=\hat{g}\p{\xi}\f{\ii\xi\sinh\p{t\abs{\xi}}}{2\pi\abs{\xi}\p{\cosh\p{b\abs{\xi}}-1}}\text{ and }w\p{\xi,t}=\hat{g}\p{\xi}\f{\cosh\p{t\abs{\xi}}-1}{\cosh\p{b\abs{\xi}}-1}.
\end{equation}
We propose that setting $P_Wg=\mathscr{F}^{-1}\p{v,w}$ gives the desired solution operator. In order to check that this is well-defined and continuous, it is sufficient to use Parseval's and Tonelli's theorems and observe the following four computations.  First:
\begin{multline}
\norm{\mathscr{F}^{-1}v}_{L^2\dot{H}^1}^2=\int_{\p{0,b}}\int_{\R^{n-1}}\abs{\xi}^2\abs{\hat{g}\p{\xi}}^2\f{\sinh\p{t\abs{\xi}}^2}{\p{\cosh\p{b\abs{\xi}}-1}^2}\;\m{d}\xi\;\m{d}t\\=\f{1}{4}\int_{\R^{n-1}}\max\{\abs{\xi},\abs{\xi}^{-2}\}\abs{\hat{g}\p{\xi}}^2\min\{1,\abs{\xi}^{3}\}\f{-2b\abs{\xi}+\sinh\p{2b\abs{\xi}}}{\p{\cosh\p{b\abs{\xi}}-1}^2}\;\m{d}\xi\le c_0\p{b}\norm{g}^2_{\dot{H}^{-1}\cap H^{1/2}}.
\end{multline}
Second:
\begin{multline}
    \norm{\mathscr{F}^{-1}v}^2_{{_0}H^1L^2}=\f{1}{4\pi^2}\int_{\p{0,b}}\int_{\R^{n-1}}\abs{\xi}^2\abs{\hat{g}\p{\xi}}^2\f{\cosh\p{t\abs{\xi}}^2}{\p{\cosh\p{b\abs{\xi}}-1}^2}\;\m{d}\xi\;\m{d}t\\=\f{1}{16\pi^2}\int_{\R^{n-1}}\max\{\abs{\xi},\abs{\xi}^{-2}\}\abs{\hat{g}\p{\xi}}^2\min\{1,\abs{\xi}^3\}\f{2b\abs{\xi}+\sinh\p{2b\abs{\xi}}}{\p{\cosh\p{b\abs{\xi}}-1}^2}\;\m{d}\xi\le c_1\p{b}\norm{g}^2_{\dot{H}^{-1}\cap H^{1/2}}.
\end{multline}
Third:
\begin{multline}
    \norm{\mathscr{F}^{-1}w}^2_{L^2\dot{H}^1}=4\pi^2\int_{\p{0,b}}\int_{\R^{n-1}}\abs{\xi}^2\abs{\hat{g}\p{\xi}}^2\f{\p{\cosh\p{t\abs{\xi}}-1}^2}{\p{\cosh\p{b\abs{\xi}}-1}^2}\;\m{d}\xi\;\m{d}t\\
    =\pi^2\int_{\R^{n-1}}\max\{\abs{\xi},\abs{\xi}^{-2}\}\abs{\hat{g}\p{\xi}}^2\min\{1,\abs{\xi}^3\}\f{6b\abs{\xi}-8\sinh\p{b\abs{\xi}}+\sinh\p{b\abs{\xi}}}{\p{\cosh\p{b\abs{\xi}}-1}^2}\;\m{d}\xi\le c_2\p{b}\norm{g}_{\dot{H}^{-1}\cap H^{1/2}}^2.
\end{multline}
Fourth:
\begin{equation}
    \norm{\mathscr{F}^{-1}w}^2_{{_0}H^1L^2}=\int_{\p{0,b}}\int_{\R^{n-1}}\abs{\xi}^2\abs{\hat{g}\p{\xi}}^2\f{\sinh\p{t\abs{\xi}}^2}{\p{\cosh\p{b\abs{\xi}}-1}^2}\;\m{d}\xi\;\m{d}t\le c_0\p{b}\norm{g}_{\dot{H}^{-1}\cap H^{1/2}}^2.
\end{equation}
It is straightforward to check that $\grad\cdot P_Wg=0$, $\m{Tr}_{\Sigma}P_Wg\cdot e_n=g$, and $\m{Tr}_{\R^{n-1}\times\tcb{0}}P_Wg=0$. By Proposition~\ref{characterizations of real-valued tempered distruibutions} and Remark~\ref{real valued functions and real valued tempered distributions} it is also ensured that $P_Wg$ is real-valued whenever $g$ is real-valued. This completes the proof.
\end{proof}
We may piece together the operators from Lemmas~\ref{simple solution operator to theone layer solenoidal problem} and~\ref{simple right inverse to the divergence} to solve problem~\eqref{multi-normal trace-divergence problem} in the single prescribed normal trace case, $m=1$.
\begin{prop}[Solution operator to~\eqref{multi-normal trace-divergence problem}: single layer case]\label{solution operator for the base case}
Let $a,b\in\R$ with $a<b$. Define the Hilbert space \begin{equation}
    \mathfrak{W}\p{a,b}=\big\{\tp{f,g}\in L^2\tp{\R^{n-1}\times\p{a,b};\mathbb{K}}\times H^{1/2}\tp{\R^{n-1}\times\cb{b};\mathbb{K}}\;:\;\norm{\p{f,g}}_{\mathfrak{W}}<\infty\big\}
\end{equation}
for the norm
\begin{equation}
    \norm{\p{f,g}}_{\mathfrak{W}}^2=\norm{f}_{L^2}^2+\norm{g}_{H^{1/2}}^2+\bsb{g-\int_{\p{a,b}}f}_{\dot{H}^{-1}}^2.
\end{equation}
There exists a bounded linear $Q^{a,b}:\mathfrak{W}\p{a,b}\to{_0}H^1\tp{\R^{n-1}\times\p{a,b};\mathbb{K}^n}$ such that $\grad\cdot Q^{a,b}\p{f,g}=f$ and $\m{Tr}_{\R^{n-1}\times\cb{b}}Q^{a,b}\p{f,g}\cdot e_n=g$.
\end{prop}
\begin{proof}
Set $W=\R^{n-1}\times\p{a,b}$ and $\Sigma=\R^{n-1}\times\cb{b}$. We propose that the assignment
    \begin{equation}
        Q^{a,b}\p{f,g}=\Pi_{W}f+P_{W}\sb{g-\m{Tr}_{\Sigma}\Pi_{W}f\cdot e_n} \text{ for } \p{f,g}\in\mathfrak{W}\p{a,b}
    \end{equation}
has the desired properties. Well-definedness and continuity of $Q^{a,b}$ are assured as soon as one observes the bound
\begin{multline}
    \sb{g-e_n\cdot\m{Tr}_{\Sigma}\Pi_Wf}_{\dot{H}^{-1}}\le\bsb{g-\int_{\p{a,b}}f}_{\dot{H}^{-1}}+\bsb{e_n\cdot\m{Tr}_{\Sigma}\Pi_Wf-\int_{\p{a,b}}f}_{\dot{H}^{-1}}\\\le\norm{\p{f,g}}_{\mathfrak{W}}+2\pi\sqrt{b-a}\norm{\Pi_Wf}_{L^2}\lesssim\norm{\p{f,g}}_{\mathfrak{W}}.
\end{multline}
Note that in the second to last inequality above we have employed the divergence compatibility estimate from Proposition~\ref{divergence_divergence}.
\end{proof}
We need one final lemma before we solve the general case of problem~\eqref{multi-normal trace-divergence problem}.
\begin{lem}\label{simple extension lemma}
    Let $a,b\in\R$ with $0<a<b$. There exists a bounded linear extension operator
    \begin{equation}
        E^{a,b}:{_0}H^1\tp{\R^{n-1}\times\p{0,a};\mathbb{K}^n}\to H^1_0\tp{\R^{n-1}\times\p{0,b};\mathbb{K}^{n}}.
    \end{equation}
    That is, $E^{a,b}f=f$ on $\R^{n-1}\times\p{0,a}$ for all $f\in{_0}H^1\p{\R^{n-1};\mathbb{K}^n}$.
\end{lem}
\begin{proof}
We construct $E^{a,b}$ via a simple reflection. Given $f\in{_0}H^1\tp{\R^{n-1}\times\p{0,a};\mathbb{K}^n}$ we define
\begin{equation}
    E^{a,b}f\p{x,y}=
    \begin{cases}
    f\p{x,y}&\text{when }\p{x,y}\in\R^{n-1}\times\p{0,a}\\
    f\p{x,y}=f\p{x,a-a\p{y-a}/(b-a)}&\text{when }\p{x,y}\in\R^{n-1}\times\p{a,b}.
    \end{cases}
\end{equation}
Thanks to the absolute continuity on lines characterization of $H^1$, we observe that $E^{a,b}$ takes values within the claimed target. This extension is also continuous since bi-Lipschitz change of coordinates boundedly preserve $H^1$ inclusion.
\end{proof}

We now have tools that are sufficient in solving the general case of problem~\eqref{multi-normal trace-divergence problem}.

\begin{thm}[Solution operator to~\eqref{multi-normal trace-divergence problem}: multilayer case]\label{solution operator to the multi-normal trace-divergence problem}
We define the appropriate Hilbert space for data in problem~\eqref{multi-normal trace-divergence problem}. For $\upalpha=\cb{a_\ell}_{\ell=1}^m\subset\R^+$ with $0<a_1<\cdots<a_m$ we define
\begin{equation}
    \textstyle\mathfrak{X}^m\p{\upalpha}=\{\p{f,\cb{g_\ell}_{\ell=1}^m}\in L^2\p{\Omega;\mathbb{K}}\times\prod_{\ell=1}^mH^{1/2}\p{\Sigma_\ell;\mathbb{K}}\;:\;\norm{\p{f,\tp{g_\ell}_{\ell=1}^m}}_{\mathfrak{X}^m\p{\upalpha}}<\infty\}
\end{equation}
for the norm
\begin{equation}
    \tnorm{\p{f,\tp{g_\ell}_{\ell=1}^m}}^2_{\mathfrak{X}^m\p{\upalpha}}=\norm{f}_{L^2}^2+\ssum{\ell=1}{m}\bp{\norm{g_\ell}_{H^{1/2}}^2+\bsb{g_\ell-\int_{\p{0,a_\ell}}f}_{\dot{H}^{-1}}^2}.
\end{equation}
There exists a linear and continuous mapping $Q_m:\mathfrak{X}^m\p{\upalpha}\to{_0}H^1\p{\Omega;\mathbb{K}^n}$ such that $\grad\cdot Q_m\p{f,\tp{g_\ell}_{\ell=1}^m}=f$ and for each $\ell\in\cb{1,\dots,m}$ one has $\m{Tr}_{\Sigma_\ell}Q_m\p{f,\tp{g_\ell}_{\ell=1}^m}\cdot e_n=g_\ell$, for all data $\p{f,\tp{g_\ell}_{\ell=1}^m}\in\mathfrak{X}^m\p{\upalpha}$.
\end{thm}
\begin{proof}
We construct the desired solution operator by way of mathematical induction on the number of specified normal traces. The precise statement to be proved, which we denote by  \emph{statement}$(m)$ for $m\in\N^+$, is as follows: for all strictly increasing sequences $\upalpha=\cb{a_\ell}_{\ell=1}^m\subset\R^+$ there exists a bounded linear mapping $Q_m:\mathfrak{X}^m\p{\upalpha}\to{_0}H^1\p{\Omega;\mathbb{K}^n}$ that is a solution operator to problem~\eqref{multi-normal trace-divergence problem}.

The case $m=1$ is handled by Proposition~\ref{solution operator for the base case}. Now suppose that $m\in\N^+$ is such that \emph{statement}$(m)$ holds true. We will prove that  \emph{statement}$(m+1)$ is true.

Let $\upbeta=\cb{b_\ell}_{\ell=1}^{m+1}\subset\R^+$ be any sequence such that $0<b_1<\dots<b_m<b_{m+1}$ and set $\upalpha=\cb{b_\ell}_{\ell=1}^m$. By hypothesis, there is a solution operator to problem~\eqref{multi-normal trace-divergence problem}, $Q_m$, for the domain $\Omega_m=\R^{n-1}\times\p{0,b_m}$ and boundary regions $\Sigma_\ell=\R^{n-1}\times\cb{b_\ell}$, for $\ell\in\cb{1,\dots,m}$. Set $\Omega_{m+1}=\R^{n-1}\times\p{0,b_{m+1}}$ and $\Sigma_{m+1}=\R^{n-1}\times\cb{b_{m+1}}$.

We propose to define for $(f,\tp{g_\ell}_{\ell=1}^{m+1})\in\mathfrak{X}^{m+1}\p{\upbeta}$
\begin{multline}\label{what's my line}
    Q_{m+1}(f,\tp{g_\ell}_{\ell=1}^{m+1})=E^{b_m,b_{m+1}}Q_m(f\mathbbm{1}_{\Omega_m},\tp{g_{\ell}}_{\ell=1}^m)\\+Q^{b_m,b_{m+1}}(f\mathbbm{1}_{\Omega_{m+1}\setminus\Omega_m}-\grad\cdot E^{b_m,b_{m+1}}Q_m(f\mathbbm{1}_{\Omega_m},\tp{g_{\ell}}_{\ell=1}^m),g_{m+1}),
\end{multline}
where $E^{b_m,b_{m+1}}$ is the extension operator from Lemma~\ref{simple extension lemma} and $Q^{b_m,b_{m+1}}$ is the solution operator from the single layer problem from Lemma~\eqref{solution operator for the base case}. First, we check that this assignment is well-defined. It is clear that $\p{f\mathbbm{1}_{\Omega_m},\tp{g_\ell}_{\ell=1}^m}\in\mathfrak{X}^m\p{\upalpha}$ with $\norm{\p{f\mathbbm{1}_{\Omega_m},\tp{g_\ell}_{\ell=1}^m}}_{\mathfrak{X}^m}\le\snorm{(f,\tp{g_\ell}_{\ell=1}^{m+1})}_{\mathfrak{X}^{m+1}}$. Hence the first term appearing on the right hand side of the equality in~\eqref{what's my line} is a well-defined and continuous assignment. To check that these same properties hold for the second term too, we observe the following compatibility estimate:
\begin{multline}
    \bsb{g_{m+1}-\int_{\tp{b_m,b_{m+1}}}f+\int_{\tp{b_m,b_{m+1}}}\grad\cdot E^{b_{m},b_{m+1}}Q_m(f\mathbbm{1}_{\Omega_m},\tp{g_\ell}_{\ell=1}^m)}_{\dot{H}^{-1}}\le\bsb{g_{m+1}-g_m-\int_{\tp{b_m,b_{m+1}}}f}_{\dot{H}^{-1}}\\
    +\bsb{g_m+\int_{\tp{b_m,b_{m+1}}}\grad\cdot E^{b_{m},b_{m+1}}Q_m(f\mathbbm{1}_{\Omega_m},\tp{g_\ell}_{\ell=1}^m)}_{\dot{H}^{-1}}\le2\tnorm{(f,\tp{g_\ell}_{\ell=1}^{m+1})}_{\mathfrak{X}^{m+1}}\\+2\pi\sqrt{b_{m+1}-b_m}\tnorm{E^{b_m,b_{m+1}}Q_m\tp{f\mathbbm{1}_{\Omega_m},\tp{g_\ell}_{\ell=1}^m}}_{L^2}.
\end{multline}
In the above we have employed the divergence compatibility estimate from Proposition~\ref{divergence_divergence} and the boundedness of the extension operator $E^{b_m,b_{m+1}}$ from Lemma~\ref{simple extension lemma}. 
Hence $Q_{m+1}$ is well-defined and continuous. In the set $\Omega_m$ we have the equality $Q_{m}=Q_{m+1}$. The second term in the definition of $Q_{m+1}$, equation~\eqref{what's my line}, vanishes on $\Sigma_m$ and the first term in the definition vanishes on $\Sigma_{m+1}$; therefore, $Q_{m+1}$ is a solution operator to the problem~\eqref{multi-normal trace-divergence problem} in the $m+1$-prescribed normal trace case. This completes the induction.
\end{proof}

\subsection{Isomorphism associated to multilayer traveling Stokes}\label{subsection on isomorphism associated to multilayer traveling Stokes}

In this subsection we construct a solution operator to the multilayer traveling Stokes problem with stress boundary conditions in equation~\eqref{multilayer traveling stokes with stress boundary conditions}. The validity of this section's results over the fields $\R$ and $\C$ is integral to the proof of Theorem~\ref{finer asymptotic development of the multiplier of the normal stress to normal Dirichlet pseudodifferential operator} in the next subsection. We remind the reader that the Euclidean inner product is sesquilinear with the left argument the linear one and that essential information regarding anti-duality can be found in Appendix~\ref{appendix on (anti-)duality and Lax-Milgram}.

We begin by studying the weak formulation and showing the existence of weak solutions. First we focus on existence of a pressure with the following result.
\begin{lem}[Image of the gradient is the annihilator of solenoidal fields]\label{the pressure lemma}
Suppose that $F\in({_0}H^1\p{\Omega;\mathbb{K}^n})^{\Bar{\ast}}$ vanishes on solenoidal fields. Then, there exists $p\in L^2\p{\Omega;\mathbb{K}}$ such that for all $u\in{_0}H^1\p{\Omega;\mathbb{K}^n}$
\begin{equation}
    \br{F,u}_{({_0}H^1)^{\bar{\ast}},{_0}H^1}=\int_{\Omega}p\cdot(\grad\cdot u).
\end{equation}
\end{lem}
\begin{proof}
The case $\mathbb{K}=\R$ is handled by Corollary 2.3 in~\cite{leoni2019traveling}. We will show this is sufficient to justify the case $\mathbb{K}=\C$ as well. Given an antilinear functional $F\in({_0}H^1\p{\Omega;\C^n})^{\Bar{\ast}}$ we define the $\R$-linear functionals $F_{\m{Re}},F_{\m{Im}}\in({_0}H^1\p{\Omega;\R^n})^\ast$ via 
\begin{equation}
    \br{F_{\m{Re}},v}=\m{Re}\sb{\br{F,v}},\;\br{F_{\m{Im}},v}=\m{Re}\sb{\br{F,\ii v}}\text{ for }v\in {_0}H^1\p{\Omega;\R^n}.
\end{equation}
Observe that if $F$ annihilates solenoidal fields then $F_{\m{Re}}$ and $F_{\m{Im}}$ satisfy the hypothesis of the lemma for the $\R$-valued case. Therefore there are $q,r\in L^2\p{\Omega;\R}$ such that for all $v,w\in{_0}H^1\p{\Omega;\R^n}$
\begin{equation}
    \m{Re}\sb{\br{F,v+\ii w}}=\br{F_{\m{Re}},v}+\br{F_{\m{Im}},w}=\int_{\Omega}q\grad\cdot v+r\grad\cdot w=\m{Re}\bsb{\int_{\Omega}\p{q+\ii r}\cdot(\grad\cdot \p{v+\ii w})}.
\end{equation}
This suggests that we set $p\in L^2\p{\Omega;\C}$ via $p=q+\ii r$. It remains to check that $G\in({_0}H^1\p{\Omega;\C^n})^{\Bar{\ast}}$ defined via $\br{G,u}= \br{F,u}-\int_{\Omega}p\cdot(\grad\cdot u)\in\C$ vanishes identically. The above computation shows that $\m{Re}\sb{\br{G,u}}=0$ for all $u$. By antilinearity, the real part of $G$ determines $G$ entirely; i.e. for all $u\in{_0}H^1\p{\Omega;\C^n}$ it holds $\br{G,u}=\m{Re}\sb{\br{G,u}}+\ii\m{Re}\sb{\br{G,\ii u}}$. Thus $G=0$ and the proof is complete.
\end{proof}

The truth of the two subsequent results in the $\R$-valued case is a consequence of Theorem 2.4 in~\cite{leoni2019traveling}. We include a proof here in the $\mathbb{K}$-valued case for $\mathbb{K}\in\cb{\R,\C}$.

\begin{lem}\label{lemma on the coercivity of the bilinear form}
For $\gam\in\R^+$ We define the sesquilinear (bilinear if $\mathbb{K}=\R$) mapping $B_\gam:{_0}H^1(\Omega;\mathbb{K}^n)\times{_0}H^1(\Omega;\mathbb{K}^n)\to\C$ via
\begin{equation}
    B_\gam(w,v)=\ssum{\ell=1}{m}\int_{\Omega_\ell}\f{\mu_\ell}{2}\mathbb{D}w:\mathbb{D}v-\gam\rho_\ell\pd_1w\cdot v.
\end{equation}
Then we have the identity
\begin{equation}\label{nowhere man}
    \m{Re}\sb{B_\gam(w,w)}=\ssum{\ell=1}{m}\int_{\Omega_\ell}\f{\mu_\ell}{2}\abs{\mathbb{D}w}^2,\text{ for all }w\in{_0}H^1(\Omega;\mathbb{K}^n).
\end{equation}
In particular, $B_\gam$ is coercive over the $H^1$-norm.
\end{lem}
\begin{proof}
    We observe that
    \begin{equation}
    -\gam\ssum{\ell=1}{m}\int_{\Omega_\ell}\rho_\ell\pd_1w\cdot w=\gam\ssum{\ell=1}{m}\int_{\Omega_\ell}\rho_\ell w\cdot\pd_1w=\Bar{\gam\ssum{\ell=1}{m}\int_{\Omega_\ell}\rho_\ell\pd_1w\cdot w}\imp\gam\ssum{\ell=1}{m}\int_{\Omega_\ell}\rho_\ell\pd_1w\cdot w\in\ii\R.
    \end{equation}
Equation~\eqref{nowhere man} follows. We now deduce that $B_\gam$ is $H^1$-coercive by Korn's inequality (see Appendix~\ref{appendix on Korn's inequality}).
\end{proof}
We now use the preceding lemmas to construct weak solutions to~\eqref{multilayer traveling stokes with stress boundary conditions}.
\begin{prop}[Existence and uniqueness of weak solutions to~\eqref{multilayer traveling stokes with stress boundary conditions}]\label{weak solutions to the traveling stokes with stress boundary problem}
For $\gam\in\R$ we define the mapping
\begin{equation}
\upchi_\gam:L^2\tp{\Omega;\mathbb{K}}\times{_0}H^1\tp{\Omega;\mathbb{K}^n}\to L^2\tp{\Omega;\mathbb{K}}\times({_0}H^1\tp{\Omega;\mathbb{K}^n})^{\Bar{\ast}}\text{ via }\upchi_\gam\p{p,u}=\p{\grad\cdot u,\mathfrak{E}_{\gamma}(p,u)},
\end{equation}
where the antilinear functional $\mathfrak{E}_{\gamma}(p,u) \in({_0}H^1\tp{\Omega;\mathbb{K}^n})^{\Bar{\ast}}$ is defined via
\begin{multline}\label{weak formulation of the traveling stokes problem}
\br{\mathfrak{E}_{\gamma}(p,u),v}_{({_0}H^1)^{\Bar{\ast}},{_0}H^1}=-\int_{\Omega}p\cdot (\grad\cdot v)+\ssum{\ell=1}{m}\int_{\Omega_\ell}\f{\mu_\ell}{2}\mathbb{D}u:\mathbb{D}v-\gam\rho_\ell\pd_1u\cdot v\\=-\int_{\Omega}S^{\upmu}\p{p,u}:\grad v-\gam\ssum{\ell=1}{m}\int_{\Omega_\ell}\rho_\ell\pd_1u\cdot v
\end{multline}
for $v\in{_0}H^1\p{\Omega;\mathbb{K}^n}$. Then $\upchi_\gam$ is a Hilbert isomorphism for all $\gam\in\R$.
\end{prop}
\begin{proof}
We begin by showing that $\upchi_\gam$ is a surjection. Thanks to the observation that the sesquilinear (bilinear when $\mathbb{K}=\R$) form $B_\gam$ is bounded paired with the coercive estimate of Lemma~\ref{lemma on the coercivity of the bilinear form}, we are free to invoke the Lax-Milgram lemma (see Proposition~\ref{lax-milgram lemma}) for $B_\gam$ on any closed subspace of ${_0}H^1\p{\Omega;\mathbb{K}^n}$.

Let $\p{g,F}\in L^2\p{\Omega;\mathbb{K}}\times({_0}H^1\p{\Omega;\mathbb{K}^n})^{\Bar{\ast}}$ be any data pair.  Then Lax-Milgram implies that there exists a unique $w\in{_0}H^1_{\sigma}\p{\Omega;\mathbb{K}^n}$ such that for all $v\in{_0}H^1_{\sigma}\p{\Omega;\mathbb{K}^n}$ 
\begin{equation}\label{decomposed}
    B_\gam\p{w,v}=-B_\gam\p{\Pi_\Omega g,v}+\br{F,v}_{({_0}H^1)^{\Bar{\ast}},{_0}H^1},
\end{equation}
where $\Pi_\Omega$ is the bounded right inverse to the divergence granted by Lemma~\ref{simple right inverse to the divergence}. Next, we apply Lemma~\ref{the pressure lemma} to find that there is a pressure $p\in L^2\p{\Omega;\mathbb{K}}$ such that for all $v\in{_0}H^1\p{\Omega;\mathbb{K}^n}$
\begin{equation}
     B_\gam\p{w,v}=-B_\gam\p{\Pi_\Omega g,v}+\br{F,v}_{({_0}H^1)^{\Bar{\ast}},{_0}H^1}+\int_{\Omega}p\cdot(\grad\cdot v).
\end{equation}
We may now conclude that $\upchi_\gam\p{p,w+\Pi_\Omega g}=\p{g,F}$, showing this mapping to be a surjection.

On the other hand, suppose $\p{p,u}\in L^2\p{\Omega;\mathbb{K}}\times{_0}H^1\p{\Omega;\mathbb{K}^n}$ satisfy $\upchi_\gam\p{p,u} = \p{g,F}.$  Then we decompose $u=w+\Pi_\Omega g$ and take $v=w\in{_0}H^1_{\sigma}\p{\Omega;\mathbb{K}^n}$ in identity~\eqref{weak formulation of the traveling stokes problem} to see that 
\begin{equation}
\br{F,w}_{({_0}H^1)^{\Bar{\ast}},{_0}H^1} = \br{\mathfrak{E}_{\gamma}(p,u),w}_{({_0}H^1)^{\Bar{\ast}},{_0}H^1} = B_\gam(u,w) = B_\gam(u,u) - B_\gam(u,\Pi_\Omega g).    
\end{equation}
Again employing Lemma~\ref{simple right inverse to the divergence},  we deduce from this that
\begin{equation}\label{velocity field estimate}
  \norm{u}_{{_0}H^1}^2  \lesssim B_\gam(u,u) \lesssim \tp{ \norm{g}_{L^2} +\norm{F}_{({_0}H^1)^{\Bar{\ast}}} }  \norm{u}_{{_0}H^1},
 \text{ and so }
  \norm{u}_{{_0}H^1} \lesssim  \norm{g}_{L^2} +\norm{F}_{({_0}H^1)^{\Bar{\ast}}}. 
\end{equation}
We then take $v=\Pi_\Omega p$ in~\eqref{weak formulation of the traveling stokes problem} and use~\eqref{velocity field estimate} to deduce that
\begin{equation}\label{pressure estimate}
\norm{p}_{L^2}\lesssim\norm{u}_{{_0}H^1}+\norm{F}_{({_0}H^1)^{\Bar{\ast}}}\lesssim\norm{g}_{L^2}+\norm{F}_{({_0}H^1)^{\Bar{\ast}}}.
\end{equation}
Estimates~\eqref{velocity field estimate} and~\eqref{pressure estimate} show that $\upchi_\gam$ is also an injection.
\end{proof}

The following lemma will be useful in the next sections where it will be important to know that ${\upchi_\gam}^{-1}$ commutes with tangential multipliers (see Appendix~\ref{appendix on (tangential) multipliers}).

\begin{lem}\label{commutes with tangential multipliers 1}
Let $\gam\in\R$ and $\omega\in L^\infty\tp{\R^{n-1};\C}$, and consider the tangential multiplier $M_\omega$ is defined in Definitions~\ref{definition of tangential multipliers} and~\ref{definition of tangential fourier multipliers 2}.  If $\p{g,F}\in L^2\p{\Omega;\C}\times({_0}H^1\p{\Omega;\C^n})^{\Bar{\ast}}$ and $\p{p,u}={\upchi_\gam}^{-1}\p{g,F}$, then $\p{M_\omega p,M_\omega u}={\upchi_\gam}^{-1}\p{M_\omega g,M_\omega F}$. 
\end{lem}
\begin{proof}
We simply check that $M_\omega g=M_{\omega}\grad\cdot u=\grad\cdot M_{\omega}u$ and note that if $v\in({_0}H^1\p{\Omega;\C^n})^\ast$, then
\begin{equation}
    \tbr{M_{\omega}F,v}_{({_0}H^1)^{\Bar{\ast}},{_0}H^1}=\tbr{F,M_{\Bar{\omega}}v}_{({_0}H^1)^{\Bar{\ast}},{_0}H^1}=\int_{\Omega}S^{\upmu}\p{M_\omega p,M_\omega u}:\grad v-\gam\ssum{\ell=1}{m}\int_{\Omega_\ell}\rho_\ell\pd_1M_\omega u\cdot v.
\end{equation}
Therefore, $\upchi_\gam\p{M_\omega p,M_\omega u}=\p{M_\omega g,M_\omega F}$.
\end{proof}

Next, we examine the regularity of weak solutions. We make the following notation.

\begin{defn}\label{strong to weak conversion notation}
For $s\in\R^+\cup\cb{0}$ we define the continuous and linear maps 
\begin{equation}
 \textstyle\mathscr{O}:\prod_{\ell=1}^mH^{-1/2}(\Sigma_\ell;\mathbb{K}^n)\to({_0}H^1(\Omega;\mathbb{K}^n))^{\Bar{\ast}},\;\mathscr{P}:H^s(\Omega;\mathbb{K}^n)\times\prod_{\ell=1}^mH^{1/2+s}(\Sigma_\ell;\mathbb{K}^n)\to({_0}H^1(\Omega;\mathbb{K}^n)^{\Bar{\ast}}
\end{equation}
with actions on $v\in{_0}H^1(\R^n;\mathbb{K}^n)$ given by
\begin{equation}
    \tbr{\mathscr{O}\tp{\phi_{\ell}}_{\ell=1}^m,v}_{{{_0}H^1}^{\Bar{\ast}},{_0}H^1}=\ssum{\ell=1}{m}\tbr{\phi_\ell,\m{Tr}_{\Sigma_\ell}v}_{H^{-1/2},H^{1/2}}  
\end{equation}
for $\tp{\phi_\ell}_{\ell=1}^m\in\prod_{\ell=1}^mH^{-1/2}(\Sigma_\ell;\mathbb{K}^n)$, and
\begin{equation}
    \tbr{\mathscr{P}(f,\tp{k_\ell}_{\ell=1}^m),v}_{({_0}H^1)^{\Bar{\ast}},{_0}H^1}=\int_{\Omega}f\cdot v+\ssum{\ell=1}{m}\int_{\Sigma_\ell}k_\ell\cdot v,
\end{equation}
for $f\in H^{s}(\Omega;\mathbb{K}^n)$ and $\tp{k_\ell}_{\ell=1}^m\in\prod_{\ell=1}^mH^{1/2+s}(\Sigma_\ell;\mathbb{K}^n)$. 
\end{defn}

We can now state our regularity result.

\begin{prop}[Regularity of weak solutions to~\eqref{multilayer traveling stokes with stress boundary conditions}]\label{regularity of the multilayer traveling stokes with stress boundary conditions}
Let $s\in\R^+\cup\cb{0}$ and
\begin{equation}
    \textstyle(g,f,\tp{k_\ell}_{\ell=1}^m)\in H^{1+s}\p{\Omega;\mathbb{K}}\times H^s\p{\Omega;\mathbb{K}^n}\times\prod_{\ell=1}^mH^{1/2+s}\p{\Sigma_\ell;\mathbb{K}^n}.
\end{equation}
Suppose that $\p{p,u}\in L^2\p{\Omega;\mathbb{K}}\times{_0}H^1\p{\Omega;\mathbb{K}^n}$ are such that $\upchi_\gam\p{p,u}=\p{g,\mathscr{P}(f,\tp{k_\ell}_{\ell=1}^m)}$. Then, in fact, we also have the inclusions $p\in H^{1+s}\p{\Omega;\mathbb{K}}$ and $u\in {_0}H^{2+s}\p{\Omega;\mathbb{K}^n}$, as well as the universal estimate
\begin{equation}
    \ssum{\ell=1}{m}\ssb{\norm{p}_{H^{1+s}\p{\Omega_\ell}}+\norm{u}_{H^{2+s}\p{\Omega_\ell}}}\lesssim\ssum{\ell=1}{m}\ssb{\norm{g}_{H^{1+s}\p{\Omega_\ell}}+\norm{f}_{H^s\p{\Omega_\ell}}+\norm{k_\ell}_{H^{1/2+s}}}.
\end{equation}
Finally the pair $\p{p,u}$ are a strong solution to the multilayer traveling stokes problem with stress boundary conditions in equation~\eqref{multilayer traveling stokes with stress boundary conditions} with data tuple $\p{g,f,\tp{k_\ell}_{\ell=1}^m}$.
\end{prop}
\begin{proof}
This is a standard induction and interpolation argument based on applying horizontal difference quotients to derive control of horizontal derivatives and then exploiting the elliptic structure of the Stokes operator to control the vertical derivatives. For a sketch we refer the reader to the real valued one layer case in Theorem 2.5 of~\cite{leoni2019traveling}.
\end{proof}

We may now synthesize the previous two propositions to state the isomorphism of Hilbert spaces associated to problem~\eqref{multilayer traveling stokes with stress boundary conditions}.

\begin{thm}[Existence and uniqueness of strong and classical solutions to~\eqref{multilayer traveling stokes with stress boundary conditions}]\label{isomorphism associated to multilayer traveling stokes with stress boundary conditions}
Let $s\in\R^+\cup\cb{0}$ and $\gam\in\R$. Define the bounded linear operator
\begin{equation}
    \textstyle
    \Upphi_\gam:H^{1+s}\p{\Omega;\mathbb{K}}\times{_0}H^{2+s}\p{\Omega;\mathbb{K}^n}\to H^{1+s}\p{\Omega;\mathbb{K}}\times H^{s}\p{\Omega;\mathbb{K}^n}\times\prod_{\ell=1}^mH^{1/2+s}\p{\Sigma_\ell;\mathbb{K}^n}
\end{equation}
with the assignment
\begin{equation}
    \textstyle\Upphi_\gam\p{p,u}=\p{\grad\cdot u,\sum_{\ell=1}^m\mathbbm{1}_{\Omega_\ell}\sb{\grad\cdot S^{\upmu}\p{p,u}-\gam\rho_\ell\pd_1u},\tp{\jump{S^{\upmu}\p{p,u}e_n}_\ell}_{\ell=1}^m}.
\end{equation}
Then $\Upphi_\gam$ is a Hilbert isomorphism.
\end{thm}
\begin{proof}
Proposition~\ref{weak solutions to the traveling stokes with stress boundary problem} shows $\Upphi_\gam$ to be injective. Propositions~\ref{weak solutions to the traveling stokes with stress boundary problem} and~\ref{regularity of the multilayer traveling stokes with stress boundary conditions} show that $\Upphi_\gam$ is a surjection.
\end{proof}

\subsection{Analysis of the normal stress to normal Dirichlet pseudodifferential operator}\label{section on analysis of the normal stress to normal Dirichlet pseudodifferential operator}

In this subsection we study a $\Psi$DO built from $\Upphi_\gam$. We make the following definition.

\begin{defn}[Normal stress to normal Dirichlet $\Psi$DO]\label{definition of the normal stress to normal dirichlet psido}Let $\gam\in\R$ and $s\in\R^+\cup\cb{-1,0}$. We define the normal stress to normal Dirichlet pseudodifferential operator to be  the bounded linear mapping $\upnu_\gamma:\prod_{\ell=1}^mH^{1/2+s}\p{\Sigma_\ell;\mathbb{K}}\to\prod_{\ell=1}^mH^{3/2+s}\p{\Sigma_\ell;\mathbb{K}}$ given by
\begin{equation}
    \textstyle \upnu_\gamma\tp{\psi_\ell}_{\ell=1}^m=\tp{\m{Tr}_{\Sigma_\ell}u\cdot e_n}_{\ell=1}^m,
\end{equation}
for $\p{p,u}\in H^{1+s}\p{\Omega;\mathbb{K}}\times{_0}H^{2+s}\p{\Omega;\mathbb{K}^n}$ the unique solution to the normal stress problem:
\begin{equation}\label{normal_stress pde}
    \begin{cases}
        \grad\cdot S^{\upmu}\p{p,u}+\gam\rho_\ell\pd_1u=0&\text{in }\Omega_\ell,\;\ell\in\cb{1,\dots,m}\\
        \grad\cdot u=0&\text{in }\Omega\\
        \jump{S^{\upmu}\p{p,u}e_n}_\ell=\psi_\ell e_n&\text{on }\Sigma_\ell,\;\ell\in\cb{1,\dots,m}\\
        \jump{u}_{\ell}=0&\text{on }\Sigma_\ell,\;\ell\in\cb{1,\dots,m-1}\\
        u=0&\text{on }\Sigma_0.
    \end{cases}
\end{equation}
In other words, $\p{p,u}={\upchi_{-\gam}}^{-1}(0,\mathscr{O}\tp{\psi_{\ell}e_n}_{\ell=1}^m)$ for the operators $\mathscr{O}$ and ${\upchi_{-\gam}}^{-1}$ from Definition~\ref{strong to weak conversion notation} and Proposition~\ref{weak solutions to the traveling stokes with stress boundary problem}, respectively (note the minus sign preceding $\gam$).
\end{defn}
\begin{rmk}\label{remark about the normal stress to normal dirchlet mapping}
The boundedness of $\upnu_\gam$ is a consequence of the boundedness of $\mathscr{O}$, Proposition~\ref{weak solutions to the traveling stokes with stress boundary problem}, and Theorem~\ref{isomorphism associated to multilayer traveling stokes with stress boundary conditions}. The restriction $s\in\R^{+}\cup\cb{-1,0}$ is not important. By interpolation theory we are free to take any $s\in[-1,\infty)$ in the previous definition statement.
\end{rmk}
We begin by proving that the Fourier transform diagonalizes $\upnu_\gam$. This gives us a representation of this linear mapping as a frequency space multiplication operator.
\begin{prop}[Diagonalization of $\upnu_\gamma$]\label{diagonalization of normal stress to normal dirichlet}
There exists a bounded and measurable matrix field $\bf{n}_\gam:\R^{n-1}\to\C^{m\times m}$ such that $\Bar{\bf{n}_\gam\p{\xi}}=\bf{n}_\gam\p{-\xi}$ for a.e. $\xi\in\R^{n-1}$ and for all $s\in[-1,\infty)$ and all $\tp{\psi_\ell}_{\ell=1}^m\in\prod_{\ell=1}^mH^{1/2+s}\p{\Sigma_\ell;\mathbb{K}}$ we have the equality
\begin{equation}
    \mathscr{F}\sb{\upnu_\gam\tp{\psi_\ell}_{\ell=1}^m}\p{\xi}=\bf{n}_\gam\p{\xi}\mathscr{F}\sb{\tp{\psi_\ell}_{\ell=1}^m}\p{\xi}\text{ for a.e. }\xi\in\R^{n-1}.
\end{equation}
Moreover, there exists a constant $c\in\R^+$, depending only on the physical parameters, such that for a.e. $\xi\in\R^{n-1}$ we have $\abs{\bf{n}_\gam\p{\xi}}\le c(1+\abs{\xi}^2)^{-1/2}$.
\end{prop}
\begin{proof}
    Let $j,k\in\cb{1,\dots,m}$. Define $\upnu_\gam^{j,k}:H^{-1/2}\p{\Sigma_k;\mathbb{K}}\to H^{1/2}\p{\Sigma_j;\mathbb{K}}$ via $\upnu_\gam^{j,k}\psi=\m{Tr}_{\Sigma_j}u\cdot e_n$ for $\p{p,u}$ a solution pair to the normal stress PDE~\eqref{normal_stress pde} with normal stress $\psi$ on the surface $\Sigma_k$. In other words,
    \begin{equation}
        (p,u)={\upchi_{-\gam}}^{-1}(0,\mathscr{O}\tp{\psi\del_{k,\ell}e_n}_{\ell=1}^m),\text{ for }\del_{\cdot,\cdot}\text{ the Kronecker delta}.
    \end{equation}
    
    It is clear that this assignment is bounded and, by Lemma~\ref{commutes with tangential multipliers 1}, translation invariant. We are in a position to apply Proposition~\ref{tilo-fourier multiplier Sobolev space characterization} to obtain a measurable function $\bf{n}_\gam^{j,k}:\R^{n-1}\to\C$ such that $\Bar{\bf{n}_\gam^{j,k}\p{\xi}}=\bf{n}_\gam^{j,k}\p{-\xi}$ for a.e. $\xi\in\R^{n-1}$, that obeys the estimate
    \begin{equation}
        \m{esssup}\tcb{(1+\abs{\xi}^2)^{-1/2}\tabs{\bf{n}^{j,k}_\gam(\xi)}\;:\;\xi\in\R^{n-1}}\le 2\tnorm{\upnu_\gam^{j,k}}_{\mathcal{L}(H^{-1/2};H^{1/2})},
    \end{equation}
    and satisfies the identity $\mathscr{F}\tsb{\upnu_{\gam}^{j,k}\psi}=\bf{n}_\gam^{j,k}\mathscr{F}\tsb{\psi}$ for all $\psi\in H^{-1/2}(\Sigma_k;\mathbb{K})$.
    
    We set $\bf{n}_\gam:\R^{n-1}\to\C^{m\times m}$ via ${\bf{n}_\gam(\xi)}_{j,k}=\bf{n}_\gam^{j,k}(\xi)$ for $\xi\in\R^{n-1}$ and $j,k\in\tcb{1,\dots,m}$. The following computation verifies that this matrix field is the sought after spectral representation:
    \begin{multline}
        \mathscr{F}\tsb{\upnu_\gam\tp{\psi_\ell}_{\ell=1}^m}=\ssum{j,k=1}{m}(\mathscr{F}\tsb{\upnu_\gam(\psi_\ell\del_{\ell,k})_{\ell=1}^m}\cdot e_j)e_j=\ssum{j,k=1}{m}(\mathscr{F}\tsb{\upnu^{j,k}_\gam\psi_k}\cdot e_j)e_j\\=\ssum{j,k=1}{m}(\bf{n}_\gam^{j,k}\mathscr{F}\tsb{\psi_k}\cdot e_j)e_j=\ssum{j,k=1}{m}({\bf{n}_\gam}_{j,k}\mathscr{F}\tsb{\psi_k}\cdot e_j)e_j=\bf{n}_\gam\mathscr{F}\tsb{\tp{\psi_\ell}_{\ell=1}^m}.
    \end{multline}
\end{proof}
We next observe how the multiplier $\bf{n}_\gam$ changes under the map $\gamma\mapsto-\gamma$.
\begin{prop}[Adjoint of the normal stress to normal Dirichlet multiplier]\label{adjoint of the normal stress to normal Dirichlet multiplier}
If $\gam\in\R^+$ then for a.e. $\xi\in\R^{n-1}$ we have the adjoint identity ${\bf{n}_\gam(\xi)}^\ast=\bf{n}_{-\gam}(\xi)$.
\end{prop}
\begin{proof}
Let $(\psi_\ell)_{\ell=1}^m,(\phi_\ell)_{\ell=1}^m\in\prod_{\ell=1}^mH^{-1/2}(\Sigma_\ell;\C)$ and denote $(p,u)={\upchi_{-\gam}}^{-1}(0,\mathscr{O}(\psi_\ell e_n)_{\ell=1}^m)$, $(q,v)={\upchi_\gam}^{-1}(0,\mathscr{O}(\phi_\ell e_n)_{\ell=1}^m)$, where we recall that $\upchi_{\gam}$ and $\upchi_{-\gam}$ are defined in Proposition~\ref{weak solutions to the traveling stokes with stress boundary problem} and $\mathscr{O}$ is from Definition~\ref{strong to weak conversion notation}. By testing $v$ in the weak formulation for $(p,u)$, recalling that $\grad\cdot v=\grad\cdot u=0$, and integrating by parts, we obtain the identities
\begin{multline}
    \tbr{(\psi_\ell)_{\ell=1}^m,(\m{Tr}_{\Sigma_\ell}v\cdot e_n)_{\ell=1}^m}_{H^{-1/2},H^{1/2}}=\ssum{\ell=1}{m}\int_{\Omega_\ell}\f{\mu_\ell}{2}\mathbb{D}u:\mathbb{D}v+\gam\rho_\ell\pd_1u\cdot v=\Bar{\ssum{\ell=1}{m}\int_{\Omega_\ell}\f{\mu_\ell}{2}\mathbb{D}v:\mathbb{D}u-\gam\rho_\ell\pd_1v\cdot u}\\=\Bar{\tbr{(\phi_\ell)_{\ell=1}^m,(\m{Tr}_{\Sigma_\ell}u\cdot e_n)_{\ell=1}^m}_{H^{-1/2},H^{1/2}}}=\tbr{(\m{Tr}_{\Sigma_\ell}u\cdot e_n)_{\ell=1}^m,(\phi_\ell)_{\ell=1}^m}_{H^{1/2},H^{-1/2}}.
\end{multline}
Hence, we may apply Propositions~\ref{(anti-)dual representation of Sobolev spaces} and~\ref{diagonalization of normal stress to normal dirichlet} along with Definition~\ref{definition of the normal stress to normal dirichlet psido} to deduce  that
\begin{equation}\label{floor still needs sweeping}
    \int_{\R^{n-1}}\mathscr{F}\tsb{(\psi_\ell)_{\ell=1}^m}\cdot\bf{n}_{-\gam}\mathscr{F}\tsb{(\phi_\ell)_{\ell=1}^m}=\int_{\R^{n-1}}\bf{n}_\gam\mathscr{F}[(\psi_\ell)_{\ell=1}^m]\cdot\mathscr{F}\tsb{(\phi_\ell)_{\ell=1}^m}.
\end{equation}
Let $a,b\in\C^{m}$ and $\psi,\phi\in L^2(\R^{n-1};\C)$. In~\eqref{floor still needs sweeping} we are free to take $(\psi_\ell)_{\ell=1}^m=(a_\ell\psi)_{\ell=1}^m$ and $(\phi_\ell)_{\ell=1}^m=(b_\ell\phi)_{\ell=1}^m$ to see that
\begin{equation}
  \int_{\R^{n-1}}\tp{a\cdot \bf{n}_{-\gam}b-\bf{n}_{\gam}a\cdot b}\mathscr{F}[\psi]\cdot\mathscr{F}[\phi]=0.
\end{equation}
As $\phi$ and $\psi$ are arbitrary, we deduce that, up to a null set depending on $a$ and $b$, we may equate $a\cdot\bf{n}_{-\gam}b=\bf{n}_\gam a\cdot b$. By letting $a$ and $b$ range over the members of the standard basis of $\C^m$ and recalling that a countable union of null sets is, again, a null set we conclude that ${\bf{n}_\gam(\xi)}^\ast=\bf{n}_{-\gam}(\xi)$ for a.e. $\xi\in\R^{n-1}$.
\end{proof}

The remainder of this subsection is devoted to a more precise asymptotic development of the matrix field $\bf{n}_\gam$ as $\xi\to0$ and $\xi\to\infty$, which we achieve with energy estimates.  Recall that from  Proposition~\ref{weak solutions to the traveling stokes with stress boundary problem} and trace theory we have the equivalence
\begin{equation}\label{do do de da}
    \tnorm{(\psi_\ell)_{\ell=1}^m}_{H^{-1/2}}\asymp\norm{u}_{{_0}H^1}+\norm{p}_{L^2}
\end{equation}
for $(\psi_\ell)_{\ell=1}^m\in\prod_{\ell=1}^mH^{-1/2}(\Sigma_\ell;\mathbb{K})$ and $(p,u)={\upchi_{-\gamma}}^{-1}(0,\mathscr{O}(\psi_\ell e_n)_{\ell=1}^m)$.  Our next result shows that if we weaken the control of $(\psi_\ell)_{\ell=1}^m$ at low frequencies on the Fourier side, then we can remove $p$ from the right.  The resulting equivalence will play a key role in our asymptotic developments of $\bf{n}_\gam$.  

\begin{thm}[Normal stress and velocity energy equivalence]\label{energy estimates for the normal stress problem}
For $\gam\in\R$ there exists a constant $c\in\R^+$, depending also on the remaining physical parameters, for which we have the equivalence
\begin{equation}\label{blackbird fly}
    c^{-1}\norm{u}_{{_0}H^1}\le\bp{\int_{\R^{n-1}}\min\tcb{\abs{\xi}^2,\abs{\xi}^{-1}}\tabs{\mathscr{F}\tsb{\tp{\psi_\ell}_{\ell=1}^m}\p{\xi}}^2\;\m{d}\xi}^{1/2}\le c\norm{u}_{{_0}H^1}
\end{equation}
for all data $\tp{\psi_{\ell}}_{\ell=1}^m\in\prod_{\ell=1}^mH^{-1/2}\p{\Sigma_\ell;\mathbb{K}}$, where $\p{p,u}\in L^2\p{\Omega;\mathbb{K}}\times{_0}H^1\p{\Omega;\mathbb{K}^n}$ is the uniquely determined weak solution to the normal stress problem in equation~\eqref{normal_stress pde}, i.e. $\p{p,u}={\upchi_{-\gam}}^{-1}(0,\mathscr{O}\tp{\psi_{\ell}e_n}_{\ell=1}^m)$ for the operators $\mathscr{O}$ and ${\upchi_{-\gam}}^{-1}$from Definition~\ref{strong to weak conversion notation} and Proposition~\ref{weak solutions to the traveling stokes with stress boundary problem}.
\end{thm}
\begin{proof}
We begin by proving the left inequality of~\eqref{blackbird fly}. By the definition of $\p{p,u}={\upchi_{-\gam}}^{-1}(0,\mathscr{O}\tp{\psi_{\ell}e_n}_{\ell=1}^m)$ we have that for all $v\in{_0}H^1(\Omega;\mathbb{K}^n)$
\begin{equation}\label{revolution 1}
    \ssum{\ell=1}{m}\int_{\Omega_\ell}\f{\mu_\ell}{2}\mathbb{D}u:\mathbb{D}v+\gam\rho_\ell\pd_1u\cdot v=\int_{\Omega}p\cdot(\grad\cdot v)+\ssum{\ell=1}{m}\tbr{\psi_\ell,\m{Tr}_{\Sigma_\ell}v\cdot e_n}_{H^{-1/2},H^{1/2}}.
\end{equation}
Taking $v=u$ in this equation and taking the real part (see Lemma~\ref{lemma on the coercivity of the bilinear form}) implies that
\begin{multline}\label{helter skelter}
    \norm{u}_{{_0}H^1}^2\lesssim\m{Re}\bsb{\ssum{\ell=1}{m}\tbr{\psi_\ell,\m{Tr}_{\Sigma_\ell}u\cdot e_n}_{H^{-1/2},H^{1/2}}}=\m{Re}\bsb{\int_{\R^{n-1}}\mathscr{F}\tsb{\tp{\psi_\ell}_{\ell=1}^m}\cdot\mathscr{F}\tsb{\tp{\m{Tr}_{\Sigma_\ell}u\cdot e_n}_{\ell=1}^m}}\\
    \le\bp{\int_{\R^{n-1}}\min\tcb{\abs{\xi}^2,\abs{\xi}^{-1}}\tabs{\mathscr{F}\tsb{\tp{\psi_\ell}_{\ell=1}^m}\p{\xi}}\;\m{d}\xi}^{1/2}\bp{\int_{\R^{n-1}}\max\tcb{\abs{\xi}^{-2},\abs{\xi}}\tabs{\mathscr{F}\tsb{\tp{\m{Tr}_{\Sigma_\ell}u\cdot e_n}_{\ell=1}^m}\p{\xi}}\;\m{d}\xi}^{1/2},
\end{multline}
where we have used that $u$ is solenoidal, Korn's inequality  (see Appendix~\ref{appendix on Korn's inequality}), and the (anti-)duality of Sobolev spaces (see Appendix~\ref{appendix on (anti-)duality and Lax-Milgram}). Next we use the boundedness of traces and the divergence compatibility estimate of Proposition~\ref{divergence_divergence} to further bound
\begin{equation}\label{long, long, long}
    \bp{\int_{\R^{n-1}}\max\tcb{\abs{\xi}^{-2},\abs{\xi}}\tabs{\mathscr{F}\tsb{\tp{\m{Tr}_{\Sigma_\ell}u\cdot e_n}_{\ell=1}^m}\p{\xi}}\;\m{d}\xi}^{1/2}\le\norm{\tp{\m{Tr}_{\Sigma_\ell}u\cdot e_n}_{\ell=1}^m}_{H^{1/2}\cap \dot{H}^{-1}}\lesssim\norm{u}_{{_0}H^1}.
\end{equation}
Combining estimates~\eqref{helter skelter} and~\eqref{long, long, long} gives the left hand inequality  of~\eqref{blackbird fly}.

We now prove the right inequality of~\eqref{blackbird fly}. For $\ell\in\cb{1,\dots,m}$ define $\phi_\ell\in H^{1/2}\p{\Sigma_\ell;\mathbb{K}}\cap\dot{H}^{-1}\p{\Sigma_\ell;\mathbb{K}}$ via
\begin{equation}
    \mathscr{F}\sb{\phi_\ell}\p{\xi}=\min\{\abs{\xi}^2,\abs{\xi}^{-1}\}\mathscr{F}\sb{\psi_\ell}\p{\xi}\text{ for }\xi\in\R^{n-1}.
\end{equation}
We bound the norm of $\phi_\ell$ as follows:
\begin{multline}
    \norm{\phi_\ell}^2_{H^{1/2}\cap\dot{H}^{-1}}\le 2\int_{\R^{n-1}}\max\{\abs{\xi}^{-2},\abs{\xi}\}|\min\{\abs{\xi}^2,\abs{\xi}^{-1}\}\mathscr{F}\sb{\psi_\ell}\p{\xi}|^2\;\m{d}\xi\\=2\int_{\R^{n-1}}\min\{\abs{\xi}^2,\abs{\xi}^{-1}\}\abs{\mathscr{F}\sb{\psi_\ell}\p{\xi}}^2\;\m{d}\xi.
\end{multline}
We are in a position to apply Theorem~\ref{solution operator to the multi-normal trace-divergence problem} to obtain $Q_m\p{0,\tp{\phi_\ell}_{\ell=1}}^m\in{_0}H^1\p{\Omega;\mathbb{K}^n}$ with the estimate
\begin{equation}\label{pear 11}
    \tnorm{Q_m\p{0,\tp{\phi_\ell}}_{\ell=1}^m}_{{_0}H^1}^2 
    \lesssim \ssum{\ell=1}{m}\norm{\varphi_\ell}^2_{H^{1/2}\cap\dot{H}^{-1}}
    \lesssim\int_{\R^{n-1}}\min\{\abs{\xi}^2,\abs{\xi}^{-1}\}\abs{\mathscr{F}\sb{\tp{\psi_\ell}_{\ell=1}^m}\p{\xi}}^2\;\m{d}\xi.
\end{equation}
Testing $v=Q_m\p{0,\tp{\phi_\ell}_{\ell=1}^m}$ in the weak formulation of the normal stress PDE in equation~\eqref{revolution 1} and using Proposition~\ref{(anti-)dual representation of Sobolev spaces} gives the identity
\begin{equation}\label{something}
    \ssum{\ell=1}{m}\int_{\R^{n-1}}\min\{\abs{\xi}^2,\abs{\xi}^{-1}\}\abs{\mathscr{F}\sb{\psi_\ell}\p{\xi}}^2\;\m{d}\xi=\ssum{\ell=1}{m}\int_{\Omega_\ell}\f{\mu_\ell}{2}\mathbb{D}u:\mathbb{D}Q_m\p{0,\tp{\phi_\ell}_{\ell=1}^m}+\gam\rho_\ell\pd_1u\cdot Q_m\p{0,\tp{\phi_\ell}_{\ell=1}^m}.
\end{equation}
The right inequality  of~\eqref{blackbird fly} now follows by applying the Cauchy-Schwarz inequality to the right hand side of~\eqref{something} and then utilizing estimate~\eqref{pear 11}.
\end{proof}
 
We are now in a position for finer asymptotic development of the symbol to the normal stress to normal Dirichlet $\Psi$DO.

\begin{thm}[Asymptotics of normal stress to normal Dirichlet multiplier]\label{finer asymptotic development of the multiplier of the normal stress to normal Dirichlet pseudodifferential operator}
For each $\gam\in\R$ there exists a constant $C\in\R^+$, depending only on the physical parameters, such that for a.e. $\xi\in\R^{n-1}$ the following hold.
\begin{enumerate}
    \item We have the estimate $\abs{\bf{n}_\gam\p{\xi}}\le C\min\{\abs{\xi}^2,\abs{\xi}^{-1}\}$.
    \item Letting $\pd B_{\C}\p{0,1}$ denote the unit sphere of $\C^{n-1}$, we have the bound 
\begin{equation}
\min_{a\in\pd B_{\C}\p{0,1}}\m{Re}\sb{\bf{n}_\gam\p{\xi}a\cdot a}\ge C^{-1}\min\{\abs{\xi}^{2},\abs{\xi}^{-1}\}.
\end{equation}
  \item The matrix $\bf{n}_\gam\p{\xi}$ is invertible, and $|\bf{n}_\gam\p{\xi}^{-1}|\le C\max\{\abs{\xi}^{-2},\abs{\xi}\}$.
\end{enumerate}
\end{thm}
\begin{proof}
For item one, we first use the divergence compatibility estimates from Proposition~\ref{divergence_divergence}. If $\tp{\psi_\ell}_{\ell=1}^m\in\prod_{\ell=1}^mH^{-1/2}\p{\Sigma_\ell;\mathbb{C}}$ are normal stress data then their associated velocity field $u$ solving~\eqref{normal_stress pde} is solenoidal and vanishing on $\Sigma_0$. Hence, the normal traces on the hyperplanes $\Sigma_\ell$, $\ell\in\cb{1,\dots,m}$ belong to $H^{1/2}\cap\dot{H}^{-1}$. In fact, we may bound with the divergence compatibility estimate and then the left inequality of Theorem~\ref{energy estimates for the normal stress problem} to obtain the bound
\begin{equation}\label{this guy is not an equation}
    \tsb{\upnu_\gam\tp{\psi_\ell}_{\ell=1}^m}_{\dot{H}^{-1}}^2\lesssim\norm{u}_{L^2}^2\lesssim\norm{u}_{{_0}H^1}^2\lesssim\int_{\R^{n-1}}\min\{\abs{\xi}^2,\abs{\xi}^{-1}\}\abs{\mathscr{F}\sb{\tp{\psi_\ell}_{\ell=1}^m}\p{\xi}}^2\;\m{d}\xi
\end{equation}
for all $\tp{\psi_\ell}_{\ell=1}^m\in\prod_{\ell=1}^mH^{-1/2}\p{\Sigma_\ell;\C}$.

Let $b=\p{b_1,\dots,b_m}\in(\Q+\ii\Q)^m$ and $\varphi\in L^1\tp{\R^{n-1};\R}$ such that $\varphi\p{\xi}\ge0$ for a.e. $\xi\in\R^{n-1}$ and the support of $\varphi$ is compact. Set $\phi\in\bigcap_{s\in\R} H^s\tp{\R^{n-1};\C}$ via $\phi=\mathscr{F}^{-1}\sb{\sqrt{\varphi}}$. We take $\tp{\psi_\ell}_{\ell=1}^m=\tp{b_\ell\phi}_{\ell=1}^m$ in inequality~\eqref{this guy is not an equation} and use Proposition~\ref{diagonalization of normal stress to normal dirichlet} to see that
\begin{equation}\label{checkmark}
    \int_{\R^{n-1}}\abs{\xi}^{-2}\abs{\bf{n}_\gam\p{\xi}b}^2\varphi\p{\xi}\;\m{d}\xi\le c\int_{\R^{n-1}}\min\{\abs{\xi}^2,\abs{\xi}^{-1}\}\abs{b}^2\varphi\p{\xi}\;\m{d}\xi.
\end{equation}
This inequality holds for all $\varphi$ as above. Hence there exists $E_b\subseteq B\p{0,1}\subset\R^{n-1}$ with $\Le^{n-1}\p{B\p{0,1}\setminus E_b}=0$ and
\begin{equation}\label{me me me}
    \abs{\bf{n}_\gam\p{\xi}b}^2\le c\abs{\xi}^4\abs{b}^2,\;\forall\;\xi\in E_b,\;\forall\;b\in(\Q+\ii\Q)^m.
\end{equation}
Set $E=\bigcap_{b\in(\Q+\ii\Q)^m}E_b$ and note that since  $(\Q+\ii\Q)^m$ is countable, $\Le^{n-1}\p{B\p{0,1}\setminus E}=0$.  Then \eqref{me me me} implies that $\abs{\bf{n}_\gam\p{\xi}b}\le c\abs{\xi}^2\abs{b}$ for all $\xi \in E$ and all $b \in (\Q+\ii\Q)^m$, but then by the density of $(\Q+\ii\Q)^m$ in $\C^m$ we find that this estimate continues to hold for all $\xi \in E$ and $b \in \C^m$.  In turn, taking the supremum over $b \in \C^m$ with $\abs{b}=1$ and using the equivalence of the operator norm and Euclidean norm on $\C^{m \times m}$, we deduce that there is a constant $c>0$, depending only on the physical parameters, such that $\abs{\bf{n}_\gam\p{\xi}}\le c\abs{\xi}^2$ for a.e. $\xi \in B(0,1) \subset \R^{n-1}$.  Combining this with the estimate from Proposition~\ref{diagonalization of normal stress to normal dirichlet} then proves the first item.

We next prove the second item. Again we let $b=\p{b_1,\dots,b_m}\in(\Q+\ii\Q)^m$ and $\varphi\in L^1\tp{\R^{n-1};\R}$ such that $\varphi\p{\xi}\ge0$ for a.e. $\xi\in\R^{n-1}$ and the support of $\varphi$ is compact. Set $\phi\in \bigcap_{s\in\R}H^s\tp{\R^{n-1};\C}$ via $\phi=\mathscr{F}^{-1}\sb{\sqrt{\varphi}}$. Notice that $\tp{b_\ell\phi}_{\ell=1}^m\in\prod_{\ell=1}^m H^{-1/2}\p{\Sigma_\ell;\C}$. Thanks to Proposition~\ref{weak solutions to the traveling stokes with stress boundary problem},  there are $\p{p,u}\in L^2\p{\Omega;\C}\times{_0}H^1\p{\Omega;\mathbb{K}^n}$, a weak solution to~\eqref{normal_stress pde} with data $\tp{b_\ell\phi}_{\ell=1}^m$,  i.e $(p,u)={\upchi_{-\gamma}}^{-1}(0,\mathscr{O}(b_\ell\phi)_{\ell=1}^m)$ for the operators $\mathscr{O}$ and ${\upchi_{-\gam}}^{-1}$ from Definition~\ref{strong to weak conversion notation} and Proposition~\ref{weak solutions to the traveling stokes with stress boundary problem}, respectively. We test $u$ itself in the weak formulation and use equation~\eqref{nowhere man} to obtain the identity
\begin{equation}
    \m{Re}\tsb{\tbr{(b_\ell\phi)_{\ell=1}^m,\upnu_\gam(b_\ell\phi)_{\ell=1}^m}_{H^{-1/2},H^{1/2}}}=\ssum{\ell=1}{m}\int_{\Omega_\ell}\f{\mu_\ell}{2}\abs{\mathbb{D}u}^2.
\end{equation}
Next we use the diagonalization of $\upnu_\gam$ from Proposition~\ref{diagonalization of normal stress to normal dirichlet}, (anti-)duality (Proposition~\ref{(anti-)dual representation of Sobolev spaces}), and finally the right inequality of Theorem~\ref{energy estimates for the normal stress problem} to obtain the estimate
\begin{multline}
    \textstyle\int_{\R^{n-1}}\m{Re}\sb{b\cdot\bf{n}_\gam b}\varphi=\m{Re}\sb{\sum_{\ell=1}^m\br{b_\ell\phi,e_n\cdot\m{Tr}_{\Sigma_\ell}u\cdot e_n}}\\\textstyle=\sum_{\ell=1}^m\int_{\Omega_\ell}\f{\mu_\ell}{2}\abs{\mathbb{D}u}^2\ge \tilde{c}\int_{\R^{n-1}}\min\{\abs{\xi}^2,\abs{\xi}^{-1}\}\abs{b}^2\varphi\p{\xi}\;\m{d}\xi,
\end{multline}
where $\tilde{c}\in\R^+$ depends only on the physical parameters. Therefore, there exists $F_b\subseteq\R^{n-1}$ such that $\Le^{n-1}(\R^{n-1}\setminus F_b)=0$ and
\begin{equation}
    \m{Re}\sb{b\cdot\bf{n}_\gam\p{\xi} b}\ge\tilde{c}\min\{\abs{\xi}^{2},\abs{\xi}^{-1}\}\abs{b}^2,\;\forall\;\xi\in F_b,\;\forall\;b\in(\Q+\ii\Q)^m.
\end{equation}
Set $F=\bigcap_{b\in(\Q+\ii\Q)^m}F_b$ and note $\Le^{n-1}\tp{\R^{n-1}\setminus F}=0$. If $\xi\in F$ then by density of $(\Q+\ii\Q)^m$ in $\C^m$ we have
\begin{equation}
    \min_{b\in\pd B_\C\p{0,1}}\m{Re}\sb{b\cdot\bf{n}_\gam\p{\xi}b}\ge\tilde{c}\min\{\abs{\xi}^2,\abs{\xi}^{-1}\}.
\end{equation}
This proves the second item.  In particular, this also shows that $\bf{n}_\gam\p{\xi}$ has trivial kernel and thus is invertible.  

It remains only to estimate the inverse. If $d\in\pd B_{\C}\p{0,1}$ then there exists $b\in\C^m$ such that $\bf{n}_\gam\p{\xi}b=d$. Thus
\begin{equation}
    \tilde{c}\min\{\abs{\xi}^2,\abs{\xi}^{-1}\}\abs{b}^2\le\m{Re}\sb{b\cdot\bf{n}_\gam\p{\xi} b}=\m{Re}\sb{b\cdot d}\le\abs{b}\imp\big|{\bf{n}_\gam\p{\xi}}^{-1}d\big|\le\tilde{c}^{-1}\max\{\abs{\xi}^{-2},\abs{\xi}\}.
\end{equation}
Taking the supremum over such $d$ and again using the equivalence of the Euclidean and operator norms on $\C^{m \times m}$, we complete the proof of the third item.
\end{proof}

\section{Overdetermined multilayer traveling Stokes}\label{section on overdetermined multilayer traveling Stokes}

In this and the other remaining sections we exclusively study the $\R$-valued solvability for the PDE systems considered. We next turn our attention to the following variant of system~\eqref{multilayer traveling stokes with stress boundary conditions}:
\begin{equation}\label{overdetermined multilayer traveling stokes}
    \begin{cases}
    \grad\cdot S^{\upmu}\p{p,u}-\gamma\rho_\ell\pd_1u=f&\text{in }\Omega_\ell,\;\ell\in\cb{1,\dots,m}\\
    \grad\cdot u=g&\text{in }\Omega\\
    \jump{S^{\upmu}\p{p,u}e_n}_\ell=k_\ell&\text{on }\Sigma_\ell,\;\ell\in\cb{1,\dots,m}\\
    u\cdot e_n=h_\ell&\text{on }\Sigma_\ell,\;\ell\in\cb{1,\dots,m}\\
    \jump{u}_{\ell}=0&\text{on }\Sigma_\ell,\;\ell\in\cb{1,\dots,m-1}\\
    u=0&\text{on }\Sigma_0,
    \end{cases}
\end{equation}
with unknown velocity $u$, pressure $p$, and  prescribed data $f$, $g$, $\tp{k_\ell}_{\ell=1}^m$, and $\tp{h_\ell}_{\ell=1}^m$.  We recall that we are continuing to use the abbreviated notation for $\Omega$, $\Omega_\ell$, and $\Sigma_\ell$ discussed at the start of Section~\ref{section on multilayer traveling Stokes with stress boundary and jump conditions} and that $\upmu=\cb{\mu_\ell}_{\ell=1}^m\subset\R^+$ are the viscosity parameters, $\cb{\rho_\ell}_{\ell=1}^m\subset\R^+$ are the density parameters, and $\gam\in\R$ is the wave speed. Our analysis in the previous section shows that the data $f$, $g$, and $\tp{k_\ell}_{\ell=1}^m$ entirely determine the pressure and velocity field and hence the normal traces $\tp{h_\ell}_{\ell=1}^m$. In this sense problem~\eqref{overdetermined multilayer traveling stokes} is overdetermined, so we cannot expect to solve it for general data.

\subsection{Data compatibility and associated isomorphism}\label{section on data compatibility and associated isomorphism}

In this subsection we find the range of appropriate data for system~\eqref{overdetermined multilayer traveling stokes}. We begin introducing the following notation.

\begin{defn}\label{compatibility form and overdetermined data spaces}
For $\gam\in\R$ we define the $\R$-bilinear mapping
\begin{multline}
    \textstyle\mathscr{H}^\gam:\sb{L^2\p{\Omega}\times({_0}H^1\p{\Omega;\R^n})^\ast\times\prod_{\ell=1}^mH^{1/2}\p{\Sigma_\ell}}\times\sb{\prod_{\ell=1}^mH^{-1/2}\p{\Sigma_\ell}}\to\R\\\text{via }\mathscr{H}^\gam\tsb{\tp{g,F,\tp{h_\ell}_{\ell=1}^m},\tp{\psi_\ell}_{\ell=1}^m}=\br{F,v}_{({_0}H^1)^\ast,{_0}H^1}-\int_{\Omega}gq-\ssum{\ell=1}{m}\br{\psi_\ell,h_\ell}_{H^{-1/2},H^{1/2}},
\end{multline}
for $\p{q,v}\in L^2\p{\Omega}\times{_0}H^1\p{\Omega;\R^n}$ the unique weak solution to the normal stress problem in equation~\eqref{normal_stress pde} with data $\tp{\psi_\ell}_{\ell=1}^m$, i.e $(q,v)={\upchi_{-\gamma}}^{-1}(0,\mathscr{O}(\psi_\ell e_n)_{\ell=1}^m)$ for the operators $\mathscr{O}$ and ${\upchi_{-\gam}}^{-1}$from Definition~\ref{strong to weak conversion notation} and Proposition~\ref{weak solutions to the traveling stokes with stress boundary problem}, respectively. Thanks to boundedness of ${\upchi_{-\gamma}}^{-1}$ and $\mathscr{O}$, we have that $\mathscr{H}^\gamma$ is continuous.

The set of data for which $\mathscr{H}^\gam$ vanishes identically as a linear functional of its right argument will be denoted with
\begin{equation}
    \textstyle\overset{\leftarrow}{\m{ker}}\mathscr{H}^\gam=\cb{\p{g,F,\tp{h_\ell}_{\ell=1}^m}\;:\;\mathscr{H}^\gam\sb{\p{g,F,\tp{h_\ell}_{\ell=1}^m,\tp{\psi_\ell}_{\ell=1}^m}}=0,\;\forall\;\tp{\psi_\ell}_{\ell=1}^m}
\end{equation}
and called the left kernel of $\mathscr{H}^\gam$.
\end{defn}

The following result characterizes the appropriate range of data for the system~\eqref{overdetermined multilayer traveling stokes} as exactly the left kernel of $\mathscr{H}^\gam$. In what follows recall that $\upchi_\gam$ and $\upchi_{-\gam}$ are the mappings from Proposition~\ref{weak solutions to the traveling stokes with stress boundary problem}.

\begin{prop}[Range of compatible data for~\eqref{overdetermined multilayer traveling stokes}]\label{weak overdetermined}
For $\gam\in\R$ the mapping 
\begin{multline}
    \textstyle\tilde{\upchi}_\gam:L^2\p{\Omega}\times{_0}H^1\p{\Omega;\R^n}\to L^2\p{\Omega}\times({_0}H^1\p{\Omega;\R^n})^\ast\times\prod_{\ell=1}^mH^{1/2}\p{\Sigma_\ell}\\\textstyle\text{with assignment }\tilde{\upchi}_\gam\p{p,u}=\tp{\upchi_\gam\p{p,u},\tp{\m{Tr}_{\Sigma_\ell}u\cdot e_n}_{\ell=1}^m}
\end{multline}
is an injection with closed range  $\overset{\leftarrow}{\m{ker}}\mathscr{H}^\gam$.
\end{prop}
\begin{proof}
Proposition~\ref{weak solutions to the traveling stokes with stress boundary problem} tells us that $\tilde{\upchi}_\gam$ is injective, and $\overset{\leftarrow}{\m{ker}}\mathscr{H}^\gam$ is closed by inspection. It remains only to show that the range of this mapping is the left kernel of $\mathscr{H}^\gam$.

Suppose first that $\p{g,F,\tp{h_\ell}_{\ell=1}^m}\in\overset{\leftarrow}{\m{ker}}\mathscr{H}^\gam$ and define $\p{p,u}\in L^2\p{\Omega}\times{_0}H^1\p{\Omega;\R^n}$ through $\p{p,u}={\upchi_\gam}^{-1}\p{g,F}$. If $\tp{\psi_\ell}_{\ell=1}^m\in\prod_{\ell=1}^mH^{-1/2}\p{\Sigma_\ell}$ and $\p{q,v}\in L^2\p{\Omega}\times{_0}H^1\p{\Omega;\R^n}$ are the associated solution to the normal stress PDE in equation~\eqref{normal_stress pde}, i.e. ${\upchi_{-\gam}}^{-1}(0,\mathscr{O}(\psi_\ell e_n)_{\ell=1}^m)$, then the identity
\begin{equation}
    \mathscr{H}^\gam\sb{\p{g,F,\tp{h_\ell}_{\ell=1}^m},\tp{\psi_\ell}_{\ell=1}^m}=\br{F,v}_{({_0}H^1)^\ast,{_0}H^1}-\int_{\Omega}gq-\ssum{\ell=1}{m}\br{\psi_\ell,h_\ell}_{H^{-1/2},H^{1/2}}=0
\end{equation}
implies that
\begin{multline}
    \ssum{\ell=1}{m}\br{\psi_\ell,\m{Tr}_{\Sigma_\ell}u\cdot e_n}_{H^{-1/2},H^{1/2}}=\ssum{\ell=1}{m}\int_{\Omega_\ell}\f{\mu_\ell}{2}\mathbb{D}v:\mathbb{D}u-\gam\rho_\ell\pd_1v\cdot u-\int_{\Omega}qg\\
    \ssum{\ell=1}{m}\int_{\Omega_\ell}\f{\mu_\ell}{2}\mathbb{D}u:\mathbb{D}v+\gam\rho_\ell\pd_1u\cdot v-\int_{\Omega}gq
    =\br{F,v}_{({_0}H^1)^\ast,{_0}H^1}-\int_{\Omega}gq=\ssum{\ell=1}{m}\br{\psi_\ell,h_\ell}_{H^{-1/2},H^{1/2}}.
\end{multline}
As $\tp{\psi_\ell}_{\ell=1}^m$ is an arbitrary member of $\prod_{\ell=1}^mH^{-1/2}\p{\Sigma_\ell}$ we learn that $\m{Tr}_{\Sigma_\ell}u\cdot e_n=h_\ell$ for each $\ell\in\cb{1,\dots,m}$ Therefore, $\tilde{\upchi}_\gam\p{p,u}=\p{g,F,\tp{h_\ell}_{\ell=1}^m}$.

On the other hand, if $\p{p,u}\in L^2\p{\Omega}\times{_0}H^1\p{\Omega;\R^n}$ we let $(g,F,(h_\ell)_{\ell=1}^m)=\tilde{\upchi}_\gam(p,u)$ and $(q,v)={\upchi_{-\gam}}^{-1}(0,\mathscr{O}(\psi_\ell e_n)_{\ell=1}^m)$ and compute, for $(\psi_\ell)_{\ell=1}^m \in \prod_{\ell=1}^mH^{-1/2}\p{\Sigma_\ell}$,
\begin{multline}
    \mathscr{H}^{\gam}\tsb{(g,F,(h_\ell)_{\ell=1}^m),(\psi_\ell)_{\ell=1}^m}
    = \br{F,v}_{({_0}H^1)^\ast,{_0}H^1}-\int_{\Omega}gq-\ssum{\ell=1}{m}\br{\psi_\ell,h_\ell}_{H^{-1/2},H^{1/2}} \\
    =\br{F,v}_{({_0}H^1)^\ast,{_0}H^1}-\ssum{\ell=1}{m}\int_{\Omega_\ell}\f{\mu_\ell}{2}\mathbb{D}v:\mathbb{D}u+\gam\rho_\ell\pd_1 v\cdot u \\
    =\br{F,v}_{({_0}H^1)^\ast,{_0}H^1}-\ssum{\ell=1}{m}\int_{\Omega_\ell}\f{\mu_\ell}{2}\mathbb{D}u:\mathbb{D}v-\gam\rho_\ell\pd_1 u\cdot v=0.
\end{multline}
As this holds for all such $(\psi_\ell)_{\ell=1}^m$, we conclude that $\tilde{\upchi}_\gam\p{p,u}\in\overset{\leftarrow}{\m{ker}}\mathscr{H}^\gam$.
\end{proof}

We now arrive at the isomorphism of Hilbert spaces associated to problem~\eqref{overdetermined multilayer traveling stokes}.

\begin{thm}[Existence and uniqueness of solutions to~\eqref{overdetermined multilayer traveling stokes}]\label{isomorphism associated to overdetermined multilayer traveling stokes}
Let $\gam\in\R$ and $s\in\R^+\cup\cb{0}$. Consider the bounded linear injection
\begin{equation}
    \textstyle\Uppsi_\gam:H^{1+s}\p{\Omega}\times {_0}H^{2+s}\p{\Omega_\ell;\R^n}\to H^{1+s}\p{\Omega}\times H^s\p{\Omega;\R^n}\times\prod_{\ell=1}^mH^{1/2+s}\p{\Sigma_\ell;\R^n}\times\prod_{\ell=1}^mH^{3/2+s}\p{\Sigma_\ell}
\end{equation}
with assignment $\Uppsi_\gam\p{p,u}=\p{\Upphi_\gam\p{p,u},\tp{\m{Tr}_{\Sigma_\ell}u\cdot e_n}_{\ell=1}^m}$, where $\Upphi_\gam$ is from Theorem~\ref{isomorphism associated to multilayer traveling stokes with stress boundary conditions}. The following are equivalent for $\p{g,f,\tp{k_\ell}_{\ell=1}^m,\tp{h_\ell}_{\ell=1}^m}$ belonging to the codomain of $\Uppsi_\gam$.
\begin{enumerate}
    \item $\p{g,f,\tp{k_\ell}_{\ell=1}^m,\tp{h_\ell}_{\ell=1}^m}$ belongs to the range of $\Uppsi_\gam$.
    \item The data tuple $\p{g,\mathscr{P}(f,\tp{k_\ell}_{\ell=1}^m),\tp{h_\ell}_{\ell=1}^m}$ belongs to $\overset{\leftarrow}{\m{ker}}\mathscr{H}^\gam$, where $\mathscr{P}$ is from Definition~\ref{strong to weak conversion notation}.
\end{enumerate}
\end{thm}
\begin{proof}
    Recall that $\tilde{\chi}_\gam$ is the mapping from Proposition~\ref{weak overdetermined}. If the first item holds then $\tilde{\upchi}_\gam\p{p,u}=\p{g,\mathscr{P}(f,(k_\ell)_{\ell=1}^m),\tp{h_\ell}_{\ell=1}^m}$ for the unique $\p{p,u}$ belonging to the domain of $\Uppsi_\gam$ such that $\Uppsi_\gam\p{p,u}=\p{g,f,\tp{k_\ell}_{\ell=1}^m,\tp{h_\ell}_{\ell=1}^m}$. Thus, by Proposition~\ref{weak overdetermined}, $\p{g,\mathscr{P}(f,(k_\ell)_{\ell=1}^m),\tp{h_\ell}_{\ell=1}^m}\in\overset{\leftarrow}{\m{ker}}\mathscr{H}^\gam$ and the second item follows.
    
    If the second item holds then, by Proposition~\ref{weak overdetermined} again, we learn that there are $\p{p,u}\in L^2\p{\Omega}\times{_0}H^1\p{\Omega;\R^n}$ such that $\tilde{\upchi}_\gam\p{p,u}=\p{g,\mathscr{P}(f,(k_\ell)_{\ell=1}^m),\tp{h_\ell}_{\ell=1}^m}$. In particular $\upchi_\gam\p{p,u}=\p{g,\mathscr{P}(f,(k_\ell)_{\ell=1}^m)}$ (for $\upchi_\gam$ from Proposition~\ref{weak solutions to the traveling stokes with stress boundary problem}) and hence we may apply Proposition~\ref{regularity of the multilayer traveling stokes with stress boundary conditions} to deduce that $\p{p,u}$ belongs to the domain of $\Uppsi_\gam$ and $\Uppsi_\gam\p{p,u}=\p{g,f,\tp{k_\ell}_{\ell=1}^m,\tp{h_\ell}_{\ell=1}^m}$. This shows the first item holds.
\end{proof}
\subsection{Measuring data compatibility}\label{section on measuring data compatibility}

The previous subsection showed us that a nontrivial compatibility condition must be satisfied by the data in order for a solution to~\eqref{overdetermined multilayer traveling stokes} to exist. In this subsection we further explore this compatibility condition. We associate to each data tuple a tuple of functions that quantify how `close' the data are to being compatible. We then study the dependence of the regularity and low Fourier mode behavior of the function tuple on the data.

The sense in which this association quantifies compatibility will be made clearer in the next section; however, the main idea is that the introduction of the free surface functions in the multilayer traveling Stokes with gravity-capillary boundary and jump conditions problem (equation~\eqref{multilayer traveling Stokes with gravity-capillary boundary and jump conditions}) modify the data in a way that results in inclusion within the range of $\Uppsi_\gam$. This is achieved by the free surface functions solving  certain $\Psi$DEs with forcing exactly this measurement of compatibility.

\begin{prop}\label{measurement of compatibility}
For $\gam\in\R$ there is a bounded linear mapping
\begin{equation}
    \textstyle\mathscr{I}^\gam: L^2\p{\Omega}\times({_0}H^1\p{\Omega;\R^n})^\ast\times\prod_{\ell=1}^mH^{1/2}\p{\Sigma_\ell}\to\prod_{\ell=1}^mH^{1/2}\p{\Sigma_\ell}
\end{equation}
such that for all data tuples $\p{g,F,\tp{h_\ell}_{\ell=1}^m}\in L^2\p{\Omega}\times({_0}H^1\p{\Omega;\R^n})^\ast\times\prod_{\ell=1}^mH^{1/2}\p{\Sigma_\ell}$ the following identity holds for all $\tp{\psi_\ell}_{\ell=1}^m\in\prod_{\ell=1}^mH^{-1/2}\p{\Sigma_\ell}$:
\begin{multline}\label{equation sensation}
   \mathscr{H}^\gam\tsb{\p{g,F,\tp{h_\ell}_{\ell=1}^m},\tp{\psi_\ell}_{\ell=1}^m}=\br{\tp{\psi_\ell}_{\ell=1}^m,\mathscr{I}^\gam\p{g,F,\tp{h_\ell}_{k=1}^m}}_{H^{-1/2},H^{1/2}}\\=\ssum{k=1}{m}\br{\psi_k,\mathscr{I}^\gam\p{g,F,\tp{h_\ell}_{\ell=1}^m}\cdot e_k}_{H^{-1/2},H^{1/2}},
\end{multline}
where $\mathscr{H}^\gam$ is the bilinear mapping from Definition~\ref{compatibility form and overdetermined data spaces}.
\end{prop}
\begin{proof}
    Recall that $\mathscr{H}^\gam$ is continuous thanks to Proposition~\ref{weak solutions to the traveling stokes with stress boundary problem}. Thus there is a constant $c_0\in\R^+$, depending only on the physical parameters, such that
    \begin{equation}
        \tabs{\mathscr{H}^\gam\tsb{\p{g,F,\tp{h_\ell}_{\ell=1}^m},\tp{\psi_\ell}_{\ell=1}^m}}\le c_0\bp{\norm{g}_{L^2}+\norm{F}_{({_0}H^1)^\ast}+\ssum{\ell=1}{m}\norm{h_\ell}_{H^{1/2}}}\bp{\ssum{\ell=1}{m}\norm{\psi_\ell}_{H^{-1/2}}}
    \end{equation}
    for all $\p{g,F,\tp{\psi_\ell}_{\ell=1}^m}$ and $\tp{\psi_\ell}_{\ell=1}^m$ belonging to the domain of $\mathscr{H}^\gam$. In particular, for fixed $\p{g,F,\tp{h_\ell}_{\ell=1}^m}$ the assignment $\tp{\psi_\ell}_{\ell=1}^m\mapsto\mathscr{H}^\gam\sb{\p{g,F,\tp{h_\ell}_{\ell=1}^m},\tp{\psi_\ell}_{\ell=1}^m}$ is bounded and linear with operator norm at most $c_0(\norm{g}_{L^2}+\norm{F}_{({_0}H^1)^\ast}+\sum_{\ell=1}^m\norm{h_\ell}_{H^{1/2}})$. By duality ( Proposition~\ref{(anti-)dual representation of Sobolev spaces}) there is a unique $\mathscr{I}^\gam\p{g,F,\cb{h_\ell}_{\ell=1}^m}\in\prod_{\ell=1}^mH^{1/2}\p{\Sigma_\ell}$ such that~\eqref{equation sensation} holds for all $\tp{\psi_\ell}_{\ell=1}^m\in\prod_{\ell=1}^mH^{-1/2}\p{\Sigma_\ell}$; moreover,
    \begin{equation}
        \ssum{k=1}{m}\tnorm{\mathscr{I}^\gam\tp{g,F,\tp{h_\ell}_{\ell=1}^m}\cdot e_k}_{H^{1/2}}\le c_0\bp{\tnorm{g}_{L^2}+\tnorm{F}_{({_0}H^1)^\ast}+\ssum{\ell=1}{m}\tnorm{h_\ell}_{H^{1/2}}}.
    \end{equation}
   It is also clear that $\mathscr{I}^\gam$ is linear.
\end{proof}
We next show that $\mathscr{I}^\gam$ commutes with tangential Fourier multipliers, which are defined in Appendix~\ref{appendix on (tangential) multipliers}.

\begin{lem}\label{commutes with tangential multipliers 2}
If $\gam\in\R$, $\p{g,F,\tp{\psi_\ell}_{\ell=1}^m}\in L^2\p{\Omega}\times({_0}H^1\p{\Omega;\R^n})^\ast\times\prod_{\ell=1}^mH^{1/2}\p{\Sigma_\ell}$, and $\omega\in L^\infty\tp{\R^{n-1};\C}$ satisfies $\omega\p{-\xi}=\Bar{\omega\p{\xi}}$ for a.e. $\xi\in\R^{n-1}$, then
\begin{equation}
    M_\omega\mathscr{I}^\gam\p{g,F,\tp{h_\ell}_{\ell=1}^m}=\mathscr{F}^{-1}\tsb{\omega\mathscr{F}\tsb{\mathscr{I}^\gam\p{g,F,\tp{h_\ell}_{\ell=1}^m}}}=\mathscr{I}^\gam\p{M_\omega g,M_\omega F,\tp{M_\omega h_\ell}_{
    \ell=1}^m},
\end{equation}
where the above is understood in the sense of Definitions~\ref{definition of tangential multipliers} and~\ref{definition of tangential fourier multipliers 2} and is $\R$-valued by Proposition~\ref{characterizations of real-valued tempered distruibutions}.
\end{lem}
\begin{proof}
Let $\tp{\psi_\ell}_{\ell=1}^m\in\prod_{\ell=1}^mH^{-1/2}\p{\Sigma_\ell}$ be normal stress data and denote the corresponding solution to~\eqref{normal_stress pde} with $\p{q,v}\in L^2\p{\Omega}\times{_0}H^1\p{\Omega;\R^n}$, i.e. $(q,v)={\upchi_{-\gam}(0,\mathscr{O}(\psi_\ell e_n)_{\ell=1}^m)}^{-1}$. Then by the definition of $\mathscr{I}^\gam$,  Proposition~\ref{measurement of compatibility}, and then Lemma~\ref{commutes with tangential multipliers 1} we have that
\begin{multline}
   \br{\tp{\psi_\ell}_{\ell=1}^m,M_\omega\mathscr{I}^\gam\p{g,F,\tp{h_\ell}_{\ell=1}^m}}_{H^{-1/2},H^{1/2}}=\br{\tp{M_{\Bar{\omega}}\psi_\ell}_{\ell=1}^m,\mathscr{I}^\gam\p{g,F,\tp{h_\ell}_{\ell=1}^m}}_{H^{-1/2},H^{1/2}}\\=\mathscr{H}^\gam\sb{\p{g,F,\tp{h_\ell}_{\ell=1}^m},\tp{M_{\Bar{\omega}}\psi_\ell}_{\ell=1}^m}=\br{F,M_{\Bar{\omega}}v}_{({_0}H^1)^\ast,{_0}H^1}-\int_{\Omega}gM_{\Bar{\omega}}q-\ssum{\ell=1}{m}\br{M_{\Bar{\omega}}\psi_\ell,h_\ell}_{H^{-1/2},H^{1/2}}\\
    =\br{M_{\omega}F,v}_{({_0}H^1)^\ast,{_0}H^1}-\int_{\Omega}M_{\omega}gq-\ssum{\ell=1}{m}\br{\psi_\ell,M_{\omega}h_\ell}_{H^{-1/2},H^{1/2}}=\mathscr{H}^\gam\sb{\p{M_\omega g,M_\omega F,\tp{M_\omega h_\ell}_{\ell=1}^m},\tp{\psi_\ell}_{\ell=1}^m}\\=\br{\tp{\psi_\ell}_{\ell=1}^m,\mathscr{I}^\gam\p{M_\omega g,M_\omega F,\tp{M_\omega h_\ell}_{\ell=1}^m}}_{H^{-1/2},H^{1/2}}.
\end{multline}
As this holds for all $\tp{\psi_\ell}_{\ell=1}^m\in\prod_{\ell=1}^mH^{-1/2}\p{\Sigma_\ell}$, the result follows.
\end{proof}

The previous lemma allows us to deduce the regularity properties of $\mathscr{I}^\gam$.  We record these now.

\begin{prop}\label{regularity of compatibility measure}
If $s\in\R^+\cup\cb{0}$, $\gam\in\R$, and
\begin{equation}
    \textstyle\p{g,f,\tp{k_\ell}_{\ell=1}^m,\tp{h_\ell}_{\ell=1}^m}\in H^{1+s}\p{\Omega}\times H^{s}\p{\Omega;\R^n}\times\prod_{\ell=1}^mH^{1/2+s}\p{\Sigma_\ell;\R^n}\times\prod_{\ell=1}^mH^{3/2+s}\p{\Sigma_\ell},
\end{equation}
then $\mathscr{I}^\gam\p{g,\mathscr{P}(f,(k_\ell)_{\ell=1}^m),\tp{h_\ell}_{\ell=1}^m}\in\prod_{\ell=1}^mH^{3/2+s}\p{\Sigma_\ell}$, where $\mathscr{P}$ is described in Definition~\ref{strong to weak conversion notation}. Moreover, we have the universal estimate
\begin{equation}
    \ssum{k=1}{m}\tnorm{\mathscr{I}^\gam\p{g,F,\tp{h_\ell}_{\ell=1}^m}\cdot e_k}_{H^{3/2+s}}\lesssim\ssum{\ell=1}{m}\ssb{\norm{g}_{H^{1+s}\p{\Omega_\ell}}+\norm{f}_{H^s\p{\Omega_\ell}}+\norm{k_\ell}_{H^{1/2+s}}+\norm{h_\ell}_{H^{3/2+s}}}.
\end{equation}
\end{prop}
\begin{proof}
For $k\in\N^+$ we define the real valued radial function $\omega_k\in L^\infty\tp{\R^{n-1};\R}$ with the assignment $\omega_k\p{\xi}=\mathbbm{1}_{B\p{0,k}}\p{\xi}(1+\abs{\xi}^2)^{\f{s+1}{2}}$. Using first Lemma~\ref{commutes with tangential multipliers 2} and then continuity of $\mathscr{I}^\gam$, we arrive at the estimate
\begin{multline}\label{no no}
    \tnorm{M_{\omega_k}\mathscr{I}^\gam\p{g,F,\tp{h_\ell}_{\ell=1}^m}}_{H^{1/2}}=\tnorm{\mathscr{I}^\gam\p{M_{\omega_k}g,M_{\omega_k}F,\tp{M_{\omega_k}h_\ell}_{\ell=1}^m}}_{H^{1/2}}\\\le c_0\ssb{\tnorm{M_{\omega_k}g}_{L^2}+\tnorm{M_{\omega_k}F}_{({_0}H^{1})^\ast}+\tnorm{\tp{M_{\omega_k}h_\ell}_{\ell=1}^m}_{H^{1/2}}},
\end{multline}
for $c_0\in\R^+$ depending only on the physical parameters. By Proposition~\ref{culmination of the multiplier appendix} there is $c_1\in\R^+$ depending only $s$ and physical parameters such that
\begin{multline}\label{no no no}
    \tnorm{M_{\omega_k}g}_{L^2}+\norm{M_{\omega_k}F}_{({_0}H^{1})^\ast}+\tnorm{\tp{M_{\omega_k}h_\ell}_{\ell=1}^m}_{H^{1/2}}\\\le c_1\ssum{\ell=1}{m}\ssb{\tnorm{g}_{H^{1+s}\p{\Omega_\ell}}+\tnorm{f}_{H^s\p{\Omega_\ell}}+\tnorm{k_\ell}_{H^{1/2+s}}+\tnorm{h_\ell}_{H^{3/2+s}}}.
\end{multline}
Paring equations~\eqref{no no} and~\eqref{no no no} with Parseval's theorem and Fatou's lemma then yields the bound
\begin{multline}
    \tnorm{\mathscr{I}^\gam\p{g,F,\tp{h_\ell}_{\ell=1}^m}}_{H^{3/2+s}}\le\liminf_{k\to\infty}\tnorm{M_{\omega_k}\mathscr{I}^\gam\p{g,F,\tp{h_\ell}_{\ell=1}^m}}_{H^{1/2}}\\\le c_0c_1\ssum{\ell=1}{m}\ssb{\norm{g}_{H^{1+s}\p{\Omega_\ell}}+\norm{f}_{H^s\p{\Omega_\ell}}+\norm{k_\ell}_{H^{1/2+s}}+\norm{h_\ell}_{H^{3/2+s}}},
\end{multline}
which completes the proof.
\end{proof}

For technical reasons that will become clear in the next section, we want to restrict to a smaller subspace of the domain of $\mathscr{I}^\gam$ that guarantees an image whose members' low Fourier modes are more tame. We label this subspace as follows.

\begin{defn}\label{subspace with divergence compatibility condition enforced}
We define the Hilbert space
\begin{equation}
    \textstyle\mathcal{Y}\p{\Omega}=\cb{\p{g,F,\tp{h_\ell}_{\ell=1}^m}\in L^2\p{\Omega}\times({_0}H^1\p{\Omega;\R^n})^\ast\times\prod_{\ell=1}^mH^{1/2}\p{\Sigma_\ell}\;:\;\norm{\p{g,F,\tp{h_\ell}_{\ell=1}^m}}_{\mathcal{Y}}<\infty}
\end{equation}
for the norm
\begin{equation}
    \tnorm{\p{g,F,\tp{h_\ell}_{\ell=1}^m}}^2_{\mathcal{Y}}=\norm{g}_{L^2}^2+\norm{F}_{({_0}H^1)^\ast}^2+\ssum{\ell=1}{m}\bsb{\tnorm{h_\ell}_{H^{1/2}}^2+\bsb{h_\ell-\int_{\p{0,a_\ell}}g}_{\dot{H}^{-1}}^2}.
\end{equation}
\end{defn}

In our analysis of the action of $\mathscr{I}^\gam$ over $\mathcal{Y}\p{\Omega}$ we utilize the following energy estimates of the normal stress problem with band limited data.

\begin{lem}\label{energy estimates 2 for normal stress problem}
Let $\gam\in\R$. There exists $C\in\R^+$, depending only on the physical parameters, such that for all tuples $\tp{\psi_\ell}_{\ell=1}^m\in\prod_{\ell=1}^mH^{-1/2}\p{\Sigma_\ell}$ satisfying $\m{supp}\mathscr{F}\sb{\tp{\psi_\ell}_{\ell=1}^m}\subseteq\Bar{B\p{0,1}}\subset\R^{n-1}$ we may estimate
\begin{equation}
        \norm{v}_{{_0}H^1}^2\le C^2\int_{\R^{n-1}}\abs{\xi}^2\abs{\mathscr{F}\sb{\tp{\psi_\ell}_{\ell=1}^m}\p{\xi}}^2\;\m{d}\xi
\end{equation}
  and
\begin{equation}
    \int_{\Omega}\babs{q+\sum_{\ell=1}^m\psi_\ell\mathbbm{1}_{\p{0,a_\ell}}}^2\le C^2\int_{\R^{n-1}}\abs{\xi}^2\abs{\mathscr{F}\sb{\tp{\psi_\ell}_{\ell=1}^m}\p{\xi}}^2\;\m{d}\xi,
\end{equation}
where for each $\ell\in\cb{1,\dots,m}$ we have that $\psi_\ell\mathbbm{1}_{\p{0,a_\ell}}\in L^2\p{\Omega}$ is defined via $\R^{n-1}\times\p{0,a_m}\ni\p{x,y}\mapsto\psi_\ell\p{x}\mathbbm{1}_{\p{0,a_\ell}}\p{y}\in\R$, and $\p{q,v}= {\upchi_{-\gam}}^{-1}(0,\mathscr{O}(\psi_\ell e_n)_{\ell=1}^m) \in L^2\p{\Omega}\times{_0}H^1\p{\Omega;\R^n}$ are the solution to~\eqref{normal_stress pde} with data $\tp{\psi_\ell}_{\ell=1}^m$. 
\end{lem}
\begin{proof}
The band limited assumption on the data, paired with the left hand inequality in the energy estimate of Theorem~\ref{energy estimates for the normal stress problem}, gives the first estimate.

For the second estimate we test $w\in{_0}H^1\p{\Omega;\R^n}$ in the weak formulation of~\eqref{normal_stress pde}, write $q=q+\sum_{\ell=1}^m\psi_\ell\mathbbm{1}_{\p{0,a_\ell}}-\sum_{\ell=1}^m\psi_\ell\mathbbm{1}_{\p{0,a_\ell}}$, and  rearrange to arrive at the identity
\begin{equation}
    \int_{\Omega}\grad\cdot w\bp{q+\ssum{\ell=1}{m}\psi_\ell\mathbbm{1}_{\p{0,a_\ell}}}=\ssum{\ell=1}{m}\int_{\Omega_\ell}\f{\mu_\ell}{2}\mathbb{D}v:\mathbb{D}w-\gam\rho_\ell\pd_1v\cdot w-\int_{\R^{n-1}}\ssum{\ell=1}{m}\psi_\ell\bsb{\m{Tr}_{\Sigma_\ell}w\cdot e_n-\int_{\p{0,a_\ell}}\grad\cdot w}.
\end{equation}
Then by the first estimate and the divergence compatibility estimate from Proposition~\ref{divergence_divergence} we may bound
\begin{equation}
    \babs{\int_{\Omega}\grad\cdot w\bp{q+\ssum{\ell=1}{m}\psi_\ell\mathbbm{1}_{\p{0,a_\ell}}}}\lesssim\tnorm{\tp{\psi_\ell}_{\ell=1}^m}_{\dot{H}^1}\tsb{\tnorm{w}_{{_0}H^1}+\tnorm{w}_{L^2}}.
\end{equation}
The second estimate now follows by taking $w=\Pi_\Omega\tp{q+\sum_{\ell=1}^m\psi_\ell\mathbbm{1}_{\p{0,a_\ell}}}$, where $\Pi_\Omega$ is the bounded right inverse of the divergence from Lemma~\ref{simple right inverse to the divergence}.
\end{proof}

We are now in a position to analyze the low frequency behavior of the image of $\mathcal{Y}\p{\Omega}$ under $\mathscr{I}^\gam$.

\begin{prop}\label{low mode behavior of compatibility measure}
If $\p{g,F,\tp{h_\ell}_{\ell=1}^m}\in\mathcal{Y}\p{\Omega}$ and $\gam\in\R$, then $\mathscr{I}^\gam\p{g,F,\tp{h_\ell}_{\ell=1}^m}\in\prod_{\ell=1}^m\dot{H}^{-1}\p{\Sigma_\ell}\cap H^{1/2}\p{\Sigma_\ell}$ with the universal estimate $\norm{\mathscr{I}^\gam\p{g,F,\tp{h_\ell}_{\ell=1}^m}}_{\dot{H}^{-1}\cap H^{1/2}}\lesssim\norm{\p{g,F,\tp{h_\ell}_{\ell=1}^m}}_{\mathcal{Y}}$.
\end{prop}
\begin{proof}
Recall that Lemma~\ref{commutes with tangential multipliers 2} guarantees that $\mathscr{I}^\gam$ commutes with tangential multipliers.  We use this with the  continuity of $\mathscr{I}^\gam$ and the boundedness of tangential multipliers (Proposition~\ref{tilo-fourier multiplier L2 characterization}, Definitions~\ref{definition of tangential multipliers} and~\ref{definition of tangential fourier multipliers 2}, and Lemma~\ref{lemma on fractional sobolev boundedness of tangential multipliers in the stip}) to estimate
\begin{multline}
    \textstyle\norm{\mathscr{I}^\gam\p{g,F,\tp{h_\ell}_{\ell=1}^m}}_{\dot{H}^{-1}\cap H^{1/2}}\lesssim\ssb{M_{\mathbbm{1}_{B\p{0,1}}}\mathscr{I}^\gam\p{g,F,\tp{h_\ell}_{\ell=1}^m}}_{\dot{H}^{-1}}\\\textstyle+\snorm{\mathscr{I}^\gam\sp{M_{\mathbbm{1}_{\R^{n-1}\setminus B\p{0,1}}}g,M_{\mathbbm{1}_{\R^{n-1}\setminus B\p{0,1}}}F,\tp{M_{\mathbbm{1}_{\R^{n-1}\setminus B\p{0,1}}}h_\ell}_{\ell=1}^m}}_{H^{1/2}}\lesssim\ssb{M_{\mathbbm{1}_{B\p{0,1}}}\mathscr{I}^\gam\p{g,F,\tp{h_\ell}_{\ell=1}^m}}_{\dot{H}^{-1}}\\+\tnorm{M_{\mathbbm{1}_{\R^{n-1}\setminus B\p{0,1}}}g}_{L^2}+\tnorm{M_{\mathbbm{1}_{\R^{n-1}\setminus B\p{0,1}}}F}_{({_0}H^1)^\ast}+\tnorm{\tp{M_{\mathbbm{1}_{\R^{n-1}\setminus B\p{0,1}}}h_\ell}_{\ell=1}^m}_{H^{1/2}}\\
    +\ssb{M_{\mathbbm{1}_{B\p{0,1}}}\mathscr{I}^\gam\p{g,F,\tp{h_\ell}_{\ell=1}^m}}_{\dot{H}^{-1}}+\tnorm{g}_{L^2}+\tnorm{F}_{({_0}H^1)^{\ast}}+\tnorm{(h_{\ell})_{\ell=1}^m}_{H^{1/2}}.
\end{multline}
Thus, it sufficient to show that $M_{\mathbbm{1}_{B\p{0,1}}}\mathscr{I}^\gam\p{g,F,\tp{h_\ell}_{\ell=1}^m}\in\prod_{\ell=1}^m\dot{H}^{-1}\p{\Sigma_\ell}$ with a bounded estimate. We do this via duality. 

Let $\tp{\psi_\ell}_{\ell=1}^m\in\prod_{\ell=1}^mH^{-1/2}\p{\Sigma_\ell}$ and denote the corresponding solution to the normal stress problem~\eqref{normal_stress pde} via $\p{q,v} = {\upchi_{-\gamma}}^{-1}(0,\mathscr{O}(\psi_\ell e_n)_{\ell=1}^m)\in L^2\p{\Omega}\times{_0}H^1\p{\Omega;\R^n}$. We compute
\begin{multline}
    \int_{B\p{0,1}}\mathscr{F}\tsb{\tp{\psi_\ell}_{\ell=1}^m}\p{\xi}\cdot\mathscr{F}\tsb{\mathscr{I}^\gam\tp{g,F,\tp{h_\ell}_{\ell=1}^m}}\p{\xi}\;\m{d}\xi \\
    = \tbr{\tp{\psi_\ell}_{\ell=1}^m,M_{\mathbbm{1}_{B\p{0,1}}}\mathscr{I}^\gam\tp{g,F,\tp{h_\ell}_{\ell=1}^m}}_{H^{-1/2},H^{1/2}} 
    = \mathscr{H}^\gam\ssb{\p{g,F,\tp{h_\ell}_{\ell=1}^m},\tp{M_{\mathbbm{1}_{B\p{0,1}}}\psi_\ell}_{\ell=1}^m} \\ 
    = \tbr{F,M_{\mathbbm{1}_{B\p{0,1}}}v}_{({_0}H^1)^\ast,{_0}H^1}-\int_{\Omega}gM_{\mathbbm{1}_{B\p{0,1}}}q-\ssum{\ell=1}{m}\tbr{M_{\mathbbm{1}_{B\p{0,1}}}\psi_\ell,h_\ell}_{H^{-1/2},H^{1/2}} \\
    =
    \tbr{F,M_{\mathbbm{1}_{B\p{0,1}}}v}_{({_0}H^1)^\ast,{_0}H^1}-\int_{\Omega}gM_{\mathbbm{1}_{B\p{0,1}}}\bp{q+\ssum{\ell=1}{m}\psi_\ell\mathbbm{1}_{\p{0,a_\ell}}}+\int_{\R^{n-1}}\ssum{\ell=1}{m}M_{\mathbbm{1}_{B\p{0,1}}}\psi_\ell\bp{-h_\ell+\int_{\p{0,a_\ell}}g}.
\end{multline}
Hence,  
\begin{multline}\label{end of the line}
    \babs{\int_{B\p{0,1}}\mathscr{F}\tsb{\tp{\psi_\ell}_{\ell=1}^m}\p{\xi}\cdot\mathscr{F}\tsb{\mathscr{I}^\gam\tp{g,F,\tp{h_\ell}_{\ell=1}^m}}\p{\xi}\;\m{d}\xi}\le\tnorm{F}_{({_0}H^1)^\ast}\tnorm{M_{\mathbbm{1}_{B\p{0,1}}}v}_{{_0}H^1}\\+\tnorm{g}_{L^2}\tnorm{M_{\mathbbm{1}_{B\p{0,1}}}q+\ssum{\ell=1}{m}M_{\mathbbm{1}_{B\p{0,1}}}\psi_\ell\mathbbm{1}_{\p{0,a_\ell}}}_{L^2}+\ssum{\ell=1}{m}\tsb{M_{\mathbbm{1}_{B\p{0,1}}}\psi_\ell}_{\dot{H}^{1}}\bsb{h_\ell-\int_{\p{0,a_\ell}}g}_{\dot{H}^{-1}}.
\end{multline}
Lemma~\ref{commutes with tangential multipliers 1} ensures us that $(M_{\mathbbm{1}_{B(0,1)}}q,M_{\mathbbm{1}_{B(0,1)}}v)={\upchi_{-\gamma}}^{-1}(0,\mathscr{O}(M_{\mathbbm{1}_{B(0,1)}}\psi_\ell e_n)_{\ell=1}^m)$. As $(M_{\mathbbm{1}_{B(0,1)}}\psi_\ell)_{\ell=1}^m$ is admissible band limited data, we may apply the second estimate of Lemma~\ref{energy estimates 2 for normal stress problem} to  $(M_{\mathbbm{1}_{B(0,1)}}\psi_\ell)_{\ell=1}^m$ to bound
\begin{equation}\label{something to say}
    \bnorm{M_{\mathbbm{1}_{B\p{0,1}}}q+\ssum{\ell=1}{m}M_{\mathbbm{1}_{B\p{0,1}}}\psi_\ell\mathbbm{1}_{\p{0,a_\ell}}}_{L^2}\lesssim\tsb{(M_{\mathbbm{1}_{B(0,1)}}\psi_\ell)_{\ell=1}^m}_{\dot{H}^{1}}\le\ssum{\ell=1}{m}\tsb{M_{\mathbbm{1}_{B(0,1)}}\psi_\ell}_{\dot{H}^{1}}.
\end{equation}
Therefore, by~\eqref{end of the line} and~\eqref{something to say},
\begin{equation}
    \babs{\int_{B\p{0,1}}\mathscr{F}\tsb{\tp{\psi_\ell}_{\ell=1}^m}\p{\xi}\cdot\mathscr{F}\tsb{\mathscr{I}^\gam\tp{g,F,\tp{h_\ell}_{\ell=1}^m}}\p{\xi}\;\m{d}\xi}\lesssim\tnorm{(g,F,\tp{h_\ell}_{\ell=1}^m)}_{\mathcal{Y}}\ssum{\ell=1}{m}\tsb{M_{\mathbbm{1}_{B\p{0,1}}}\psi_\ell}_{\dot{H}^{1}},
\end{equation}
and so we conclude that
\begin{multline}
    \tsb{M_{\mathbbm{1}_{B\p{0,1}}}\mathscr{I}^\gam\tp{g,F,\tp{h_{\ell=1}^m}}}_{\dot{H}^{-1}}\\\lesssim\sup\bcb{\babs{\int_{B\p{0,1}}\mathscr{F}\tsb{\tp{\psi_\ell}_{\ell=1}^m}\p{\xi}\cdot\mathscr{F}\tsb{\mathscr{I}^\gam\tp{g,F,\tp{h_\ell}_{\ell=1}^m}}\p{\xi}\;\m{d}\xi}\::\;\ssum{\ell=1}{m}\tsb{M_{\mathbbm{1}_{B\p{0,1}}}\psi_\ell}_{\dot{H}^1}\le 1}\\
    \lesssim\tnorm{\tp{g,F,\tp{h_\ell}_{\ell=1}^m}}_{\mathcal{Y}}.
\end{multline}
This completes the proof.
\end{proof}

\begin{rmk}\label{inclusion remark}
For $s\in\R^+\cup\cb{0}$ we may view
\begin{multline}
    \textstyle H^{1+s}\p{\Omega}\times H^s\p{\Omega;\R^n}\times\prod_{\ell=1}^mH^{1/2+s}\p{\Sigma_\ell;\R^n}\times\prod_{\ell=1}^mH^{3/2+s}\p{\Sigma_\ell}\\\textstyle\emb L^2\p{\Omega}\times({_0}H^1\p{\Omega;\R^n})^\ast\times\prod_{\ell=1}^mH^{1/2}\p{\Sigma_\ell}
\end{multline}
through the inclusion mapping $\p{g,f,\tp{k_\ell}_{\ell=1}^m,\tp{h_\ell}_{\ell=1}^m}\mapsto\p{g,\mathscr{P}(f,(k_\ell)_{\ell=1}^m),\tp{h_\ell}_{\ell=1}^m}$, for $\mathscr{P}$ as in Definition~\ref{strong to weak conversion notation}.
\end{rmk}

We now synthesize the results of this subsection into a single result.

\begin{thm}\label{theorem on measurement of compatibility}
Let $\gam\in\R$ and $s\in\R^+\cup\cb{0}$. The linear mapping
\begin{multline}
    \textstyle\mathscr{K}^\gam:\mathcal{Y}\p{\Omega}\cap\sb{H^{1+s}\p{\Omega}\times H^{s}\p{\Omega;\R^n}\times\prod_{\ell=1}^mH^{1/2+s}\p{\Sigma_\ell;\R^n}\times\prod_{\ell=1}^mH^{3/2+s}\p{\Sigma_\ell}}\\\textstyle\to\prod_{\ell=1}^m\dot{H}^{-1}\p{\Sigma_\ell}\cap H^{3/2+s}\p{\Sigma_\ell} 
\end{multline}
given by $\mathscr{K}^\gam\p{g,f,\tp{k_\ell}_{\ell=1}^m,\tp{h_\ell}_{\ell=1}^m}=\mathscr{I}^\gam\p{g,\mathscr{P}(f,(k_\ell)_{\ell=1}^m),\tp{h_\ell}_{\ell=1}^m}$ is both well-defined and continuous.
\end{thm}
\begin{proof}
The result follows from Remark~\ref{inclusion remark}, Proposition~\ref{regularity of compatibility measure}, and Proposition~\ref{low mode behavior of compatibility measure}.
\end{proof}
\section{Multilayer traveling Stokes with gravity-capillary boundary and jump conditions}\label{section on linearized flattened problem}
Our linear analysis culminates in this section with the study of the linearized flattened free boundary problem~\eqref{flattened problem at the nonlinear level}. More precisely, we study the system
\begin{equation}\label{multilayer traveling Stokes with gravity-capillary boundary and jump conditions}
    \begin{cases}
        \grad\cdot S^{\upmu}\p{p,u}-\gam\rho_\ell\pd_1u=f&\text{in }\Omega_\ell,\;\ell\in\cb{1,\dots,m}\\
        \grad\cdot u=g&\text{in }\Omega\\
        \jump{S^{\upmu}\p{p,u}e_n}_\ell=k_\ell+(\mathfrak{g}\jump{\uprho}_\ell+\sig_\ell\Delta_\|)\eta_\ell e_n&\text{on }\Sigma_\ell,\;\ell\in\cb{1,\dots,m}\\
        u\cdot e_n=h_\ell-\gam\pd_1\eta_\ell&\text{on }\Sigma_\ell,\;\ell\in\cb{1,\dots,m}\\
        \jump{u}_{\ell}=0&\text{on }\Sigma_\ell,\;\ell\in\cb{1,\dots,m-1}\\
        u=0&\text{on }\Sigma_0.
    \end{cases}
\end{equation}
We remind the reader that we are still using the abbreviated notation for $\Omega$, $\Omega_\ell$, and $\Sigma_\ell$ discussed at the start of Section~\ref{section on multilayer traveling Stokes with stress boundary and jump conditions} and that the unknown velocity is $u$, the pressure is $p$, and the free surface functions are in the tuple $\tp{\eta_\ell}_{\ell=1}^m$. The prescribed data are $f$, $g$, $\tp{k_\ell}_{\ell=1}^m$, and $\tp{h_\ell}_{\ell=1}^m$. The viscosity parameters are $\upmu=\cb{\mu_\ell}_{\ell=1}^m\subset\R^+$, the fluid densities are $\cb{\rho_\ell}_{\ell=1}^m\subset\R^+$, $\upsigma=\cb{\sig_\ell}_{\ell=1}^m\subset\R^+\cup\cb{0}$ are the surface tensions, $\gam\in\R$ is the signed wave speed, $\mathfrak{g}\in\R^+$ is the magnitude gravitational acceleration, and $\uprho=\sum_{\ell=1}^m\rho_\ell\mathbbm{1}_{\Omega_\ell}$. Unlike the previous two sections, we now assume that the wave speed is non-trivial, i.e. $\gam\in\R\setminus\cb{0}$ and that the density coefficients are strictly decreasing with layer number, i.e. $0<\rho_m<\cdots<\rho_1$ (this is consistent with the assumptions made in the introduction). Note that $\jump{\uprho}_\ell=\rho_{\ell+1}-\rho_\ell<0$ for $\ell\in\cb{1,\dots,m-1}$ and $\jump{\uprho}_m=-\rho_m<0$.

Our goal in this section is to prove that the above system induces a linear isomorphism between an appropriate pair of Banach spaces. It turns out that the estimates obtained from~\eqref{multilayer traveling Stokes with gravity-capillary boundary and jump conditions} are too weak to guarantee that the free surface functions and the pressure belong to standard $L^2$-based Sobolev spaces in dimensions three or higher. The resolution of this issue requires developing families of specialized Sobolev spaces to serve as the container spaces for the free surface functions and pressure. In this section and the next we establish and utilize variants of the specialized spaces developed in the single layer analysis of~\cite{leoni2019traveling} that are appropriate for the multilayer context.

\subsection{Specialized Sobolev space interlude, well-definedness, and injectivity}\label{section on specialized Sobolev space interlude, well-definedness, and injectivity}

We first label the space of data for which we will solve~\eqref{multilayer traveling Stokes with gravity-capillary boundary and jump conditions}.
\begin{defn}\label{data space for the linearized problem}
For $s\in\R^+\cup\cb{0}$ we define the space
\begin{equation}
    \textstyle\mathcal{Y}^s=\mathcal{Y}\p{\Omega}\cap\sb{H^{1+s}\p{\Omega}\times H^{s}\p{\Omega;\R^n}\times\prod_{\ell=1}^mH^{1/2+s}\p{\Sigma_\ell;\R^n}\times\prod_{\ell=1}^mH^{3/2+s}\p{\Sigma_\ell}},
\end{equation}
where $\mathcal{Y}\p{\Omega}$ is from Definition~\ref{subspace with divergence compatibility condition enforced} and the intersection is understood in the sense of the inclusion from Remark~\ref{inclusion remark}. This space is Hilbert and as an equivalent norm we set
\begin{multline}
    \tnorm{\p{g,f,\tp{k_\ell}_{\ell=1}^m,\tp{h_\ell}_{\ell=1}^m}}_{\mathcal{Y}^s}^2\\=\ssum{\ell=1}{m}\bsb{\tnorm{g}_{H^{1+s}\p{\Omega_\ell}}^2+\tnorm{f}_{H^s\p{\Omega_\ell}}^2+\tnorm{k_\ell}_{H^{1/2+s}}^2+\tnorm{h_\ell}_{H^{3/2+s}}^2+\bsb{h_\ell-\int_{\p{0,a_\ell}}g}^2_{\dot{H}^{-1}}}.
\end{multline}
\end{defn}

Next we define the container space for the free surface functions, which is an anisotropic Sobolev space introduced in \cite{leoni2019traveling}.  Note that for notational convenience we denote this space with a name different from the one used in \cite{leoni2019traveling}.

\begin{defn}\label{space for the free interface and free surface functions}
For $s\in\R^+\cup\cb{0}$ we define the normed space
\begin{equation}
    \textstyle\mathcal{H}^s\p{\R^{n-1}}=\{\zeta\in(\mathscr{S}\tp{\R^{n-1};\R})^\ast\;:\;\mathscr{F}\sb{\zeta}\in L^1_{\loc}\tp{\R^{n-1};\C}\text{ and }\norm{\zeta}_{\mathcal{H}^s}<\infty\}
\end{equation}
for
\begin{equation}
    \tnorm{\zeta}_{\mathcal{H}^s}^2=\int_{B\p{0,1}}\abs{\xi}^{-2}\sp{{\xi_1}^2+\abs{\xi}^4}\abs{\mathscr{F}\sb{\zeta}\p{\xi}}^2\;\m{d}\xi+\int_{\R^{n-1}\setminus B\p{0,1}}\abs{\xi}^{2s}\abs{\mathscr{F}\sb{\zeta}\p{\xi}}^2\;\m{d}\xi.
\end{equation}
\end{defn}

The following result summarizes the essential properties of this space.

\begin{prop}\label{linear topological properties of container for free surface and free interface functions}
Let $s\in\R^+\cup\cb{0}$.  Then the following hold for the space $\mathcal{H}^s\tp{\R^{n-1}}$. 
\begin{enumerate}
    \item $\mathcal{H}^s\tp{\R^{n-1}}$ is Hilbert.
    \item If $k\in\N$ then $\mathcal{H}^s\tp{\R^{n-1}}\emb C^k_0\tp{\R^{n-1}}+H^s\tp{\R^{n-1}}$; in particular if $(n-1)/2+k<s$ then we have the embedding $\mathcal{H}^s\tp{\R^{n-1}}\emb C^k_0\tp{\R^{n-1}}$. We remind the reader that $C^k_0$ is defined in Section~\ref{section on conventions of notation}.
    \item If $\eta\in\mathcal{H}^{5/2+s}\tp{\R^{n-1}}$ then $\pd_1\eta\in H^{3/2+s}\tp{\R^{n-1}}\cap\dot{H}^{-1}\tp{\R^{n-1}}$ and $\Delta\eta\in H^{1/2+s}\tp{\R^{n-1}}$; moreover, these mappings are continuous.
    \item (\emph{Fourier reconstruction}) If $\vartheta\in L^1_{\loc}\tp{\R^{n-1};\C}$ satisfies $\vartheta\p{-\xi}=\Bar{\vartheta\p{\xi}}$ for a.e. $\xi\in\R^{n-1}$ and
\begin{equation}
    \int_{B\p{0,1}}\abs{\xi}^{-2}\sp{{\xi_1}^2+\abs{\xi}^4}\abs{\vartheta\p{\xi}}^2\;\m{d}\xi+\int_{\R^{n-1}\setminus B\p{0,1}}\abs{\xi}^{2s}\abs{\vartheta\p{\xi}}^2\;\m{d}\xi<\infty
\end{equation}
then there exists $\zeta_\vartheta\in\mathcal{H}^{s}\tp{\R^{n-1}}$ with $\mathscr{F}\sb{\zeta_\vartheta}=\vartheta$.
\item In the case $n=2$ we have the equality of vector spaces with equivalence of norms: $\mathcal{H}^s\tp{\R^{n-1}}=H^s\tp{\R^{n-1}}$.
\item If $\zeta\in\mathcal{H}^s(\R^{n-1})$ satisfies $\m{supp}\mathscr{F}[\zeta]\subseteq\R^{n-1}\setminus B(0,\ep)$, for some $\ep\in\R^+$, then, in fact, we have the inclusion $\zeta\in H^s(\R^{n-1})$.
\end{enumerate}
\end{prop}
\begin{proof}
Item $(1)$, $(2)$, $(3)$, and $(5)$ follow from Proposition 5.3 and Theorems 5.6 and 5.7 in~\cite{leoni2019traveling}. Item $(4)$ follows from the definition of $\mathcal{H}^s$ and completeness. Item $(6)$ is clear given the definitions of the norms on $\mathcal{H}^s$ and $H^s$.
\end{proof}
 
Now we are ready to define and study the container space for the pressure, which is a multilayer variant of the space introduced in \cite{leoni2019traveling}, again given a different name for notational convenience.  For the next definition and the subsequent proposition we switch back to the unabbreviated notation for multilayer domains as defined in Section~\ref{section on Eulerian coordinate formulation}.

 \begin{defn}\label{container space for the pressure}
Let $s\in\R^+\cup\cb{0}$ and $\upzeta=\cb{\zeta_\ell}_{\ell=1}^m\in\tp{C^0_b\tp{\R^{n-1}}}^m$ be a tuple of continuous functions satisfying
\begin{equation}\label{helium}
    \textstyle\max\{\norm{\zeta_1}_{C^0_b},\dots,\norm{\zeta_m}_{C^0_b}\}\le\f14\min\cb{a_1,a_2-a_1,\dots,a_m-a_{m-1}}.
\end{equation}
We define the normed vector space
\begin{multline}
    \textstyle\mathcal{P}^s\p{\Omega[\upzeta]}=\{q\in L^1_{\loc}\p{\Omega[\upzeta]}\;:\;\exists\;\p{p,\tp{\eta_\ell}_{\ell=1}^m}\in H^{s}\p{\Omega[\upzeta]}\times\tp{\mathcal{H}^{s}\tp{\R^{n-1}}}^m\\\textstyle\text{such that}\;q=p-\mathfrak{g}\sum_{\ell=1}^m\jump{\uprho}_\ell\eta_\ell\mathbbm{1}_{\Omega_1[\upzeta]\cup\cdots\cup\Omega_\ell[\upzeta]}\}
\end{multline}
equipped with the norm
\begin{equation}
    \textstyle\norm{q}_{\mathcal{P}^s}=\inf\bcb{\sum_{\ell=1}^{m}\ssb{\norm{p}_{H^{s}\p{\Omega_\ell[\upzeta]}}+\norm{\eta_\ell}_{\mathcal{H}^s}}\;:\;q=p-\mathfrak{g}\sum_{\ell=1}^{m}\jump{\uprho}_\ell\eta_\ell\mathbbm{1}_{\Omega_1[\upzeta]\cup\cdots\cup\Omega_\ell[\upzeta]}}.
\end{equation}
When $\upzeta=0$ we will sometimes write $\mathcal{P}^s\p{\Omega}$ in place of $\mathcal{P}^s\p{\Omega[0]}$.
\end{defn}

The following result records the essential properties of these spaces.

\begin{prop}\label{linear topological properties of container for pressure}
The following properties hold for the scale of spaces $\mathcal{P}^s\p{\Omega[\upzeta]}$ for  $s\in\R^+\cup\cb{0}$ and $\upzeta\in\tp{C^0_b\tp{\R^{n-1}}}^m$ satisfying~\eqref{helium}.
\begin{enumerate}
    \item $\mathcal{P}^s\p{\Omega[\upzeta]}$ is Banach.
    \item If $k\in\N$ then $\mathcal{P}^s\p{\Omega[\upzeta]}\emb C^k_b(\Omega[\upzeta])+H^s\p{\Omega[\upzeta]}$; in particular if $n/2+k<s$ then $\mathcal{P}^s\p{\Omega[\upzeta]}\emb C^k_b(\Omega_\ell[\upzeta])$.
    \item If $p\in\mathcal{P}^{1+s}\p{\Omega[\upzeta]}$ then $\sum_{\ell=1}^m\mathbbm{1}_{\Omega_\ell[\upzeta]}\grad p\in H^{s}\p{\Omega[\upzeta];\R^n}$ and this map is continuous.
    \item For $\ell\in\cb{1,\dots,m}$ there are bounded trace operators: $\m{Tr}_{\Sigma_\ell[0]}^{\uparrow,\downarrow}:\mathcal{P}^{1+s}\p{\Omega[0]}\to\mathcal{H}^{1/2+s}\tp{\R^{n-1}}$.
    \item In the case $n=2$ we have the equality of vector spaces with equivalence of norms: $\mathcal{P}^s\p{\Omega[\upzeta]}=H^s\p{\Omega_\ell[\upzeta]}$.
\end{enumerate}
\end{prop}
\begin{proof}
The claims follow from simple multilayer adaptations of Theorems 5.9, 5.11, and 5.13 and Remark 5.10 of~\cite{leoni2019traveling}.
\end{proof}

We have all the tools we need to label the spaces which hold the velocity, pressure, and free surface tuple. 

\begin{defn}\label{domain of multilayer traveling Stokes with gravity-capillary boundary and jump conditions solution operator}
For $s\in\R^+\cup\cb{0}$ we define the Banach space
\begin{multline}
    \textstyle\mathcal{X}^s=\{\p{p,u,\tp{\eta_\ell}_{\ell=1}^m}\in\mathcal{P}^{1+s}\tp{\Omega}\times{_0}H^{2+s}\tp{\Omega;\R^n}\times\tp{\mathcal{H}^{5/2+s}\tp{\R^{n-1}}^m}\;:\\\textstyle p+\mathfrak{g}\sum_{\ell=1}^m\jump{\uprho}_\ell\eta_\ell\mathbbm{1}_{\p{0,a_\ell}}\in H^{1+s}\p{\Omega}\},
\end{multline}
which  we endow the norm
\begin{equation}
   \tnorm{\p{p,u,\tp{\eta_\ell}_{\ell=1}^m}}_{\mathcal{X}^s}=\tnorm{p}_{\mathcal{P}^{1+s}}+\ssum{\ell=1}{m}\ssb{\tnorm{u}_{H^{2+s}\p{\Omega_\ell}}+\tnorm{\eta_\ell}_{\mathcal{H}^{5/2+s}}+\tnorm{p+\mathfrak{g}{\textstyle\sum_{k=1}^m}\jump{\uprho}_k\eta_k\mathbbm{1}_{\p{0,a_k}}}_{H^{1+s}\p{\Omega_\ell}}}.
\end{equation}
\end{defn}

Next we introduce the linear map that will turn out to be the Banach isomorphism solution operator associated to the problem~\eqref{multilayer traveling Stokes with gravity-capillary boundary and jump conditions}.

\begin{prop}[Uniqueness of solutions to~\eqref{multilayer traveling Stokes with gravity-capillary boundary and jump conditions}]\label{well-definedness and injectivity}
For $\gam\in\R\setminus\cb{0}$, $\upsigma=\cb{\sig_\ell}_{\ell=1}^m\subset\R^+\cup\cb{0}$, and $s\in\R^+\cup\cb{0}$ the linear mapping $\Upupsilon_{\gam,\upsigma}:\mathcal{X}^s\to\mathcal{Y}^s$ with action
\begin{multline}\label{ready set action}
    \textstyle\Upupsilon_{\gam,\upsigma}\p{p,u,\tp{\eta_\ell}_{\ell=1}^m}=\big(\grad\cdot u,\sum_{\ell=1}^m\mathbbm{1}_{\Omega_\ell}\sb{\grad\cdot S^{\upmu}\p{p,u}-\gam\rho_\ell\pd_1u},\\\textstyle\tp{\jump{S^{\upmu}\p{p,u}e_n}_\ell-(\mathfrak{g}\jump{\uprho}_\ell+\sig_\ell\Delta_\|)\eta_\ell e_n}_{\ell=1}^m,\tp{\m{Tr}_{\Sigma_\ell}u\cdot e_n+\gam\pd_1\eta_\ell}_{\ell=1}^m\big)
\end{multline}
is well-defined, continuous, and injective.
\end{prop}
\begin{proof}
We begin by checking that the mapping is well-defined and continuous. This is clear for the first component. The only possible point of contention in the second component is the expression with the pressure, $\sum_{\ell=1}^m\mathbbm{1}_{\Omega_\ell}\grad p$; however, we are in the clear thanks to item two of Proposition~\ref{linear topological properties of container for pressure}. For the third component we use that $p+\mathfrak{g}\sum_{\ell=1}^m\jump{\uprho}_\ell\eta_\ell\mathbbm{1}_{\p{0,a_\ell}}\in H^{1+s}\p{\Omega}$, paired with the usual trace theory and the jump calculation:
\begin{equation}
    \bjump{\ssum{k=1}{m}\tjump{\uprho}_k\eta_k\mathbbm{1}_{(0,a_k)}}_\ell=\ssum{j=\ell+1}{m}\jump{\uprho}_j\eta_j-\ssum{k=\ell}{m}\jump{\uprho}_k\eta_k=-\jump{\uprho}_\ell\eta_\ell
\end{equation}
to deduce the bounded inclusion
\begin{equation}
    \textstyle \jump{p}_\ell-\mathfrak{g}\jump{\uprho}_\ell\eta_\ell = \jump{p+\mathfrak{g}\sum_{k=1}^m\jump{\uprho}_k\eta_k\mathbbm{1}_{\p{0,a_k}}}_\ell   \in H^{1/2+s}\p{\Sigma_\ell}.
\end{equation}
Item two of Proposition~\ref{linear topological properties of container for free surface and free interface functions} tells us that $\sig_\ell\Delta\eta_\ell\in H^{1/2+s}\p{\Sigma_\ell}$ boundedly as well; hence the third component of $\Upupsilon_{\gam,\upsigma}$ is well-defined and continuous. Using again item two of Proposition~\ref{linear topological properties of container for free surface and free interface functions}, we learn that $\gam\pd_1\eta_\ell\in\dot{H}^{-1}\p{\Sigma_\ell}\cap H^{3/2+s}\p{\Sigma_\ell}$; moreover, thanks to Proposition~\ref{divergence_divergence} we have the bound
\begin{equation}
   \bsb{\m{Tr}_{\Sigma_\ell}u\cdot e_n+\gam\pd_1\eta-\int_{\p{0,a_\ell}}\grad\cdot u}_{\dot{H}^{-1}}\le\abs{\gam}\sb{\pd_1\eta_\ell}_{\dot{H}^{-1}}+2\pi\sqrt{a_\ell}\norm{u}_{L^2}.
\end{equation}
We thus conclude that $\Upupsilon_{\gam,\upsigma}\p{p,u,\tp{\eta_\ell}_{\ell=1}^m}\in\mathcal{Y}^s$ with $\norm{\Upupsilon_{\gam,\upsigma}\p{p,u,\tp{\eta_\ell}_{\ell=1}^m}}_{\mathcal{Y}^s}\lesssim\norm{\p{p,u,\tp{\eta_\ell}_{\ell=1}^m}}_{\mathcal{X}^s}$.

We next prove that $\Upupsilon_{\gam,\upsigma}$ is injective.  Suppose that $\p{p,u,\tp{\eta_\ell}_{\ell=1}^m} \in \m{ker}(\Upupsilon_{\gam,\upsigma})$. Fix $r\in\N^+$ and define $\tp{\psi_\ell}_{\ell=1}^m\in\bigcap_{t\in\R^+}\prod_{\ell=1}^m H^{-1/2+t}\p{\Sigma_\ell}$ via
\begin{equation}\label{hey hey hey and i do}
    \tp{\psi_\ell}_{\ell=1}^m=\tp{M_{\mathbbm{1}_{B\p{0,2^r}\setminus B\p{0,2^{-r}}}}(\mathfrak{g}\jump{\uprho}_{\ell}+\sig_\ell\Delta_\|)\eta_\ell}_{\ell=1}^m,
\end{equation}
where we recall that tangential Fourier multipliers are defined in Appendix~\ref{appendix on (tangential) multipliers}.  Now let  
\begin{equation}
    \textstyle\p{q,v} = {\upchi_{-\gamma}}^{-1}(0,\mathscr{O}(\psi_\ell e_n)_{\ell=1}^m) \in\bigcap_{t\in\R^+}\tsb{H^{1+t}(\Omega)\times{_0}H^{2+t}(\Omega;\mathbb{R}^n)}
\end{equation}
be the corresponding solution to the normal stress PDE in equation~\eqref{normal_stress pde}. Note that we have introduced the band-limited approximation, in part, so that the above application of ${\upchi_{-\gam}}^{-1}$, as defined in Proposition~\ref{weak solutions to the traveling stokes with stress boundary problem}, is well-defined.

Since $\Upupsilon_{\gam,\upsigma}\p{p,u,\tp{\eta_\ell}_{\ell=1}^m}=0$ we obtain the following string of identities by testing $u$ in the weak formulation for $q,v$ and integrating by parts (recall that $v$ has vanishing divergence): 
\begin{multline}\label{aaaa}
    -\gam\ssum{\ell=1}{m}\int_{\Sigma_\ell}\psi_\ell\pd_1\eta_\ell=\ssum{\ell=1}{m}\int_{\Omega_\ell}\f{\mu_\ell}{2}\mathbb{D}v:\mathbb{D}u+\gam\rho_\ell\pd_1v\cdot u=\ssum{\ell=1}{m}\int_{\Omega_\ell}\f{\mu_\ell}{2}\mathbb{D}u:\mathbb{D}v-\gam\rho_\ell\pd_1u\cdot v\\=\ssum{\ell=1}{m}\int_{\Omega_\ell}\f{\mu_\ell}{2}\mathbb{D}u:\mathbb{D}v-\gam\rho_\ell\pd_1u\cdot v-\int_{\Omega}\bp{p+\mathfrak{g}\ssum{k=1}{m}\jump{\uprho}_k\eta_k\mathbbm{1}_{\p{0,a_k}}}\grad\cdot v\\
    =-\ssum{\ell=1}{m}\int_{\Omega_\ell}S^{\upmu}\bp{p+\mathfrak{g}\ssum{k=1}{m}\jump{\uprho}_k\eta_k\mathbbm{1}_{\p{0,a_k}},u}:\grad v-\gam\rho_\ell\pd_1u\cdot v\\=\ssum{\ell=1}{m}\int_{\Omega_\ell}\bsb{\grad\cdot S^{\upmu}\bp{p+\mathfrak{g}\ssum{k=1}{m}\jump{\uprho}_k\eta_k\mathbbm{1}_{\p{0,a_k}},u}-\gam\rho_\ell\pd_1u}\cdot v\\+\ssum{\ell=1}{m}\int_{\Sigma_\ell}\bjump{S^{\upmu}\bp{p+\mathfrak{g}\ssum{k=1}{m}\jump{\uprho}_k\eta_k\mathbbm{1}_{\p{0,a_k}},u}e_n}_\ell\cdot v=\mathfrak{g}\ssum{\ell=1}{m}\int_{\Omega_\ell}\grad\bsb{\ssum{k=1}{m}\jump{\uprho}_k\eta_k\mathbbm{1}_{\p{0,a_k}}}\cdot v\\+\ssum{\ell=1}{m}\int_{\Sigma_\ell}\bjump{S^{\upmu}\bp{p+\mathfrak{g}\ssum{k=1}{m}\jump{\uprho}_k\eta_k\mathbbm{1}_{\p{0,a_k}},u}e_n}_\ell\cdot v.
\end{multline}
The above manipulations are justified by the fact that $p+\mathfrak{g}\sum_{k=1}^m\jump{\uprho}_k\eta_k\mathbbm{1}_{(0,a_k)}\in H^{1+s}(\Omega)$ (Definition~\ref{domain of multilayer traveling Stokes with gravity-capillary boundary and jump conditions solution operator}) and that $\grad\ssb{\sum_{k=1}^m\jump{\uprho}_k\eta_k\mathbbm{1}_{(0,a_k)}}\in H^s(\Omega;\R^n)$ (Proposition~\ref{linear topological properties of container for pressure}, item 3).

Proposition~\ref{commutes with tangential multipliers 1}, together with the fact that $(\psi_\ell)_{\ell=1}^m$, defined in~\eqref{hey hey hey and i do}, is band limited yields the implications:
\begin{multline}
    (\psi_\ell)_{\ell=1}^m=(M_{\mathbbm{1}_{B(0,2^r)\setminus B(0,2^{-r})}}\psi_\ell)_{\ell=1}^m\imp\m{Tr}_{\Sigma_\ell}v=M_{\mathbbm{1}_{B(0,2^r)\setminus B(0,2^{-r})}}\m{Tr}_{\Sigma_\ell}v\\\imp\m{supp}\mathscr{F}\tsb{\m{Tr}_{\Sigma_\ell}v}\subseteq\Bar{B(0,2^r)}\setminus B(0,2^{-r}).
\end{multline}
Hence, Proposition~\ref{linear topological properties of container for pressure}, item four, and Proposition~\ref{linear topological properties of container for free surface and free interface functions}, item six, may be invoked to see that
\begin{equation}\label{L1_inclusions}
    \mathscr{F}\tsb{\jump{S^{\upmu}(p,u)e_n}_{\ell}}\cdot\mathscr{F}\tsb{\m{Tr}_{\Sigma_\ell}v},\;\mathscr{F}\bsb{\bjump{\mathfrak{g}\ssum{k=1}{m}\jump{\uprho}_k\eta_k\mathbbm{1}_{\p{0,a_k}}}_\ell}\cdot\mathscr{F}\tsb{\m{Tr}_{\Sigma_\ell}v\cdot e_n}\in L^1(\R^{n-1};\C),
\end{equation}
where we recall that over $\C$ we take the inner-product $\cdot$ to be sesquilinear on vectors with linearity in the left argument.  With these inclusions in hand we can simplify the $\ell^{\m{th}}$ term in the final series of~\eqref{aaaa} Parseval's theorem:
\begin{multline}\label{bbbb}
    \int_{\Sigma_\ell}\bjump{S^{\upmu}\bp{p+\mathfrak{g}\ssum{k=1}{m}\jump{\uprho}_k\eta_k\mathbbm{1}_{\p{0,a_k}},u}e_n}_\ell\cdot v=\int_{\R^{n-1}}\mathscr{F}\bsb{\bjump{S^{\upmu}\bp{p+\mathfrak{g}\ssum{k=1}{m}\jump{\uprho}_k\eta_k\mathbbm{1}_{\p{0,a_k}},u}e_n}_\ell}\cdot\mathscr{F}\tsb{\m{Tr}_{\Sigma_\ell}v}\\
    =\int_{\R^{n-1}}\mathscr{F}\tsb{\tjump{{S}^{\upmu}\p{p,u}e_n}_\ell}\cdot\mathscr{F}\tsb{\m{Tr}_{\Sigma_\ell}v}+\int_{\R^{n-1}}\mathscr{F}\bsb{\bjump{\mathfrak{g}\ssum{k=1}{m}\jump{\uprho}_k\eta_k\mathbbm{1}_{\p{0,a_k}}}_\ell}\cdot\mathscr{F}\tsb{\m{Tr}_{\Sigma_\ell}v\cdot e_n}\\
    =\int_{\R^{n-1}}\mathscr{F}\tsb{(\mathfrak{g}\jump{\uprho}_\ell+\sig_\ell\Delta_\|)\eta_\ell}\cdot\mathscr{F}\tsb{\m{Tr}_{\Sigma_\ell}v\cdot e_n}-\mathfrak{g}\jump{\uprho_\ell}\int_{\R^{n-1}}\mathscr{F}\sb{\eta_\ell}\cdot\mathscr{F}\sb{\m{Tr}_{\Sigma_\ell}v\cdot e_n},
\end{multline}
where the decomposition into the two integrands in the middle line above is justified by~\eqref{L1_inclusions}.

In the final term in equation~\eqref{bbbb}, we would like to use the vanishing divergence of and $\Sigma_0$-trace of $v$ to simplify further. We first compute 
\begin{multline}\label{dig a pony}
    \mathscr{F}[\m{Tr}_{\Sigma_\ell}v\cdot e_n](\xi)=\mathscr{F}\bsb{\int_{(0,a_\ell)}\pd_nv(\cdot,y)\;\m{d}y}(\xi)\\=-\mathscr{F}\bsb{\int_{(0,a_\ell)}\ssum{j=1}{n-1}\pd_j(v(\cdot,y)\cdot e_j)\;\m{d}y}(\xi)=\int_{\p{0,a_\ell}}\mathscr{F}\sb{v\p{\cdot,y}}(\xi)\cdot 2\pi\ii\p{\xi,0}\;\m{d}y.
\end{multline}
Using~\eqref{dig a pony}, we may then rewrite
\begin{multline}\label{cccc}
    -\mathfrak{g}\jump{\uprho_\ell}\int_{\R^{n-1}}\mathscr{F}\sb{\eta_\ell}\cdot\mathscr{F}\sb{\m{Tr}_{\Sigma_\ell}v\cdot e_n}=-\mathfrak{g}\jump{\uprho_\ell}\int_{\R^{n-1}}\mathscr{F}\tsb{\eta_\ell}(\xi)\cdot\bsb{\int_{\p{0,a_\ell}}\mathscr{F}\sb{v\p{\cdot,y}}(\xi)\cdot 2\pi\ii\p{\xi,0}\;\m{d}y}\;\m{d}\xi\\=-\mathfrak{g}\jump{\uprho}_\ell\int_{\p{0,a_m}}\int_{\R^{n-1}}\mathbbm{1}_{\p{0,a_\ell}}\mathscr{F}\tsb{\grad_{\|}\eta_\ell}\cdot\mathscr{F}\sb{v}.
\end{multline}
Summing over $\ell\in\cb{1,\dots,m}$ in~\eqref{bbbb} and implementing~\eqref{cccc} yields the identity
\begin{multline}\label{dddd}
   \ssum{\ell=1}{m}\int_{\Sigma_\ell}\bjump{S^{\upmu}\bp{p+\mathfrak{g}\ssum{k=1}{m}\jump{\uprho}_k\eta_k\mathbbm{1}_{\p{0,a_k}},u}}_\ell\cdot v=\ssum{\ell=1}{m}\int_{\R^{n-1}}\mathscr{F}\tsb{(\mathfrak{g}\jump{\uprho}_\ell+\sig_\ell\Delta_\|)\eta_\ell}\cdot\mathscr{F}\tsb{\m{Tr}_{\Sigma_\ell}v\cdot e_n}\\
   -\mathfrak{g}\ssum{\ell=1}{m}\int_{\Omega_\ell}\grad\bsb{\ssum{k=1}{m}\jump{\uprho}_k\eta_k\mathbbm{1}_{\p{0,a_k}}}\cdot v,
\end{multline}
and upon substituting~\eqref{dddd} into~\eqref{aaaa} we deduce the equality
\begin{equation}\label{eeee}
    -\gam\ssum{\ell=1}{m}\int_{\Sigma_\ell}\psi_\ell\pd_1\eta_\ell=\ssum{\ell=1}{m}\int_{\R^{n-1}}\mathscr{F}\tsb{(\mathfrak{g}\jump{\uprho}_\ell+\sig_\ell\Delta_\|)\eta_\ell}\cdot\mathscr{F}\tsb{\m{Tr}_{\Sigma_\ell}v\cdot e_n}.
\end{equation}
Next we use the definition of $\tp{\psi_\ell}_{\ell=1}^m$ on the left hand side and observe that
\begin{multline}
    -\gam\ssum{\ell=1}{m}\int_{\Sigma_\ell}\psi_\ell\pd_1\eta_\ell=-\gam\ssum{\ell=1}{m}\int_{\Sigma_\ell}\mathfrak{g}\jump{\uprho}_\ell M_{\mathbbm{1}_{B\p{0,2^r}\setminus B\p{0,2^{-r}}}}\eta_\ell\pd_1M_{\mathbbm{1}_{B\p{0,2^r}\setminus B\p{0,2^{-r}}}}\eta_\ell\\+\gam\ssum{\ell=1}{m}\int_{\Sigma_\ell}\sig_\ell\grad M_{\mathbbm{1}_{B\p{0,2^r}\setminus B\p{0,2^{-r}}}}\eta_\ell\cdot\pd_1\grad M_{\mathbbm{1}_{B\p{0,2^r}\setminus B\p{0,2^{-r}}}}\eta_\ell=0.
\end{multline}
For the right hand side of~\eqref{eeee} we use first Proposition~\ref{diagonalization of normal stress to normal dirichlet} and then item two of Theorem~\ref{finer asymptotic development of the multiplier of the normal stress to normal Dirichlet pseudodifferential operator} to bound
\begin{multline}\label{ffff}
    0=\ssum{\ell=1}{m}\int_{\R^{n-1}}\mathscr{F}\tsb{(\mathfrak{g}\jump{\uprho}_\ell+\sig_\ell\Delta_\|)\eta_\ell}\cdot\mathscr{F}\tsb{\m{Tr}_{\Sigma_\ell}v\cdot e_n}=\int_{\R^{n-1}}\m{Re}\ssb{\mathscr{F}\tsb{((\mathfrak{g}\tjump{\uprho}_\ell+\sig_\ell\Delta_{\|})\eta_\ell)_{\ell=1}^m}\cdot\bf{n}_\gam\mathscr{F}\tsb{(\psi_\ell)_{\ell=1}^m}}\\\gtrsim\int_{B\p{0,2^r}\setminus B\p{0,2^{-r}}}\min\{\abs{\xi}^2,\abs{\xi}^{-1}\}\ssum{\ell=1}{m}\tabs{\mathscr{F}\tsb{(\mathfrak{g}\jump{\uprho}_\ell+\sig_\ell\Delta_\|)\eta_\ell}\p{\xi}}^2\;\m{d}\xi.
\end{multline}
Hence the right hand side above is zero for all $r\in\N^+$. This proves that for all $\ell\in\cb{1,\dots,m}$ we have $\eta_\ell=0$. Thus, we have the inclusion $\p{p,u}\in H^{1+s}\p{\Omega}\times {_0}H^{2+s}\p{\Omega;\R^n}$.  The space on the right is the domain for $\Upphi_\gam$.  Since $(\eta_\ell)_{\ell=1}^m=0$ and $\Upupsilon_{\gam,\upsigma}(p,u,(\eta_\ell)_{\ell=1}^m)=0$ we have that $\Upphi_\gam\p{p,u}=0$.  In Theorem~\ref{isomorphism associated to multilayer traveling stokes with stress boundary conditions} we showed that $\Upphi_\gam$ is an isomorphism, so $u=0$ and $p=0$.  Hence, $\Upupsilon_{\gam,\upsigma}$ is an injection.
\end{proof}

\subsection{Isomorphism in the case with surface tension}\label{section on isomorphism in the case with surface tension}

In  this subsection we characterize the solvability of~\eqref{multilayer traveling Stokes with gravity-capillary boundary and jump conditions} for data belonging to the space $\mathcal{Y}^s$ and positive surface tensions, i.e. $\cb{\sigma_\ell}_{\ell=1}^n\subset\R^+$. Before we state and prove the relevant isomorphism theorem, we show how the data determine the free surface functions.

\begin{lem}[Determination of free surface functions: surface tension case]\label{data determines free surface - free interface functions}
If $\gam\in\R\setminus\cb{0}$, $\cb{\sig_\ell}_{\ell=1}^m\subset\R^+$, and $\p{g,f,\tp{k_\ell}_{\ell=1}^m,\tp{h_\ell}_{\ell=1}^m}\in\mathcal{Y}^s$ for some $s\in\R^+\cup\cb{0}$ then there exists $\tp{\eta_\ell}_{\ell=1}^m\in\tp{\mathcal{H}^{5/2+s}\tp{\R^{n-1}}}^m$ such that the modified data tuple
\begin{multline}\label{modified data tuple 1}
    \textstyle\tp{g,f+\mathfrak{g}\sum_{\ell=1}^m\jump{\uprho}_\ell\grad\eta_\ell\mathbbm{1}_{\p{0,a_\ell}},\tp{k_\ell+\sig_\ell\Delta_\|\eta_\ell e_n}_{\ell=1}^m,\tp{h_\ell-\gam\pd_1\eta_\ell}_{\ell=1}^m}\\\textstyle\in H^{1+s}\p{\Omega}\times H^{s}\p{\Omega;\R^n}\times\prod_{\ell=1}^mH^{1/2+s}\p{\Sigma_\ell;\R^n}\times\prod_{\ell=1}^mH^{3/2+s}\p{\Sigma_\ell}    
\end{multline}
belongs to the range of $\Uppsi_\gam$, where this latter operator is from Theorem~\ref{isomorphism associated to overdetermined multilayer traveling stokes}. Moreover, we have the universal estimate:
\begin{equation}\label{across the universe}
    \tnorm{\tp{\eta_\ell}_{\ell=1}^m}_{\mathcal{H}^{5/2+s}}\lesssim\norm{\p{g,f,\tp{k_\ell}_{\ell=1}^m,\tp{h_\ell}_{\ell=1}^m}}_{\mathcal{Y}^s}.
\end{equation}
\end{lem}
\begin{proof}
We divide the proof into three steps.

\textbf{Step 1: Establishing invertibility of a matrix field.} Let 
\begin{equation}
\bf{o}_0=-\mathfrak{g}\m{diag}\p{\jump{\uprho}_1,\dots,\jump{\uprho}_m},\;\bf{o}_1=\m{diag}\p{\sig_1,\dots,\sig_m} \in \R^{m \times m},
\end{equation}
and for $\xi\in\R^{n-1}$ we set $\bf{o}\p{\xi}=\bf{o}_0+4\pi^2\abs{\xi}^2\bf{o}_1 \in \R^{m \times m}$ and
\begin{equation}
    \bf{p}_\gam\p{\xi}=\bf{n}_\gam\p{\xi}^\ast\bf{o}\p{\xi}-2\pi\ii\gam\xi_1I_{m\times m}=\bf{n}_{-\gam}\p{\xi}\bf{o}\p{\xi}-2\pi\ii\gam\xi_1I_{m\times m} \in \C^{m \times m},
\end{equation}
where $\textbf{n}_\gam(\xi)$ is as defined in Proposition~\ref{diagonalization of normal stress to normal dirichlet} and satisfies $\bf{n}_{\gam}(\xi)^\ast=\bf{n}_{-\gam}(\xi)$ by Proposition~\ref{adjoint of the normal stress to normal Dirichlet multiplier}. We claim that there exists a constant $C\in\R^+$, depending only on the physical parameters and $\gamma$, such that for all $b\in\C^m$ and a.e. $\xi\in\R^{n-1}$ we have that
\begin{equation}\label{alphaalpha}
    \textstyle C^{-1}\abs{\bf{p}_\gam\p{\xi}b}^2\le\tsb{\tp{{\xi_1}^2+\abs{\xi}^4}\mathbbm{1}_{B\p{0,1}}\p{\xi}+\abs{\xi}^2\mathbbm{1}_{\R^{n-1}\setminus B\p{0,1}}\p{\xi}}\abs{b}^2\le C\abs{\bf{p}_\gam\p{\xi}b}^2.
\end{equation}

We begin the proof of the claim by recalling that there is a full measure set $E\subseteq\R^{n-1}$ such that if $\xi\in E$ then the estimates  from Theorem~\ref{finer asymptotic development of the multiplier of the normal stress to normal Dirichlet pseudodifferential operator} hold for the matrix $\bf{n}_\gam\p{\xi}$.   Let $\xi\in E\setminus B\p{0,1}$ and $b\in\C^m$. Then from Theorem~\ref{finer asymptotic development of the multiplier of the normal stress to normal Dirichlet pseudodifferential operator} we deduce that
\begin{equation}\label{alphabeta}
    \textstyle\abs{\bf{p}_\gam\p{\xi}b}^2\lesssim\ssb{{\xi_1}^2+\abs{\xi}^2}\abs{b}^2\lesssim\abs{\xi}^2\abs{b}^2.
\end{equation}
Similarly, since $\bf{o}\p{\xi}$ is self-adjoint, we have that $2\pi i \gamma \xi_1  b\cdot \bf{o}(\xi) b$ is purely imaginary, so again Theorem~\ref{finer asymptotic development of the multiplier of the normal stress to normal Dirichlet pseudodifferential operator} allows us to bound
\begin{multline}\label{betaalpha}
    \textstyle\abs{\xi}^2\abs{\bf{p}_\gam\p{\xi}b}\abs{b}\gtrsim\m{Re}\sb{\bf{p}_\gam\p{\xi}b\cdot\bf{o}\p{\xi}b}=\m{Re}\sb{\bf{n}_\gam\p{\xi}^\ast\bf{o}\p{\xi}b\cdot\bf{o}\p{\xi}b}\\\gtrsim\abs{\xi}^{-1}\abs{\bf{o}\p{\xi}b}^2=\abs{\xi}^{-1}\ssum{\ell=1}{m}(-\mathfrak{g}\tjump{\uprho}_\ell+4\pi^2\abs{\xi}^2\sig_\ell)^2|b\cdot e_\ell|^2\gtrsim\abs{\xi}^3\abs{b}^2.
\end{multline}
Combining~\eqref{alphabeta} and~\eqref{betaalpha} gives estimate~\eqref{alphaalpha} for $\xi\in E\setminus B\p{0,1}$. 

On the other hand, if $\xi\in B\p{0,1}\cap E$ and $b\in\C^m$, we once again appeal to Theorem~\ref{finer asymptotic development of the multiplier of the normal stress to normal Dirichlet pseudodifferential operator} to arrive at the upper bound 
\begin{equation}\label{betabeta}
    \textstyle\abs{\bf{p}_\gam\p{\xi}b}^2\lesssim\sp{{\xi_1}^2+\abs{\xi}^4}\abs{b}^2.
\end{equation}
For the matching lower bound we combine the following estimates wherein we tacitly use: 1) for each $\ell\in\cb{1,\dots,m}$ $-\mathfrak{g}\jump{\uprho}_\ell>0$ and $\abs{\gam}>0$, and 2) $\bf{n}_\gam\bf{o}{\bf{n}_\gam}^\ast$ is a self-adjoint matrix field. First:
\begin{multline}\label{deltadelta}
    \tabs{\xi}^2\abs{\bf{p}_\gam(\xi)b}\abs{b}\gtrsim\tabs{\bf{o}(\xi)\bf{p}_\gam(\xi)b\cdot{\bf{n}_\gam(\xi)}^\ast\bf{o}(\xi)b}=\tabs{\bf{n}_\gam(\xi)\bf{o}(\xi){\bf{n}_\gam(\xi)}^{\ast}\bf{o}(\xi)b\cdot\bf{o}(\xi)b-2\pi\ii\gam\xi_1\bf{o}(\xi)b\cdot{\bf{n}_\gam(\xi)}^{\ast}\bf{o}(\xi)b}\\
    \ge\tabs{\m{Im}\tsb{\bf{n}_\gam(\xi)\bf{o}(\xi){\bf{n}_\gam(\xi)}^{\ast}\bf{o}(\xi)b\cdot\bf{o}(\xi)b-2\pi\ii\gam\xi_1\bf{o}(\xi)b\cdot{\bf{n}_\gam(\xi)}^{\ast}\bf{o}(\xi)b}}\\=2\pi\tabs{\gam\xi_1}\tabs{\m{Re}\tsb{\bf{n}_\gam(\xi)\bf{o}(\xi)b\cdot\bf{o}(\xi)b}}
    \gtrsim2\pi\abs{\gam\xi_1}\abs{\xi}^2\abs{\bf{o}(\xi)b}^2\gtrsim \abs{\gam\xi_1}\abs{\xi}^2\abs{b}^2
\end{multline}
Second:
\begin{equation}\label{deltadeltadelta}
    \abs{\bf{p}_\gam\p{\xi}b}\abs{b}\gtrsim\m{Re}\sb{\bf{p}_\gam\p{\xi}b\cdot\bf{o}\p{\xi}b}=\m{Re}\tsb{\bf{n}_\gam\p{\xi}^\ast\bf{o}\p{\xi}b\cdot\bf{o}\p{\xi}b}\gtrsim\abs{\xi}^2\abs{\bf{o}\p{\xi}b}^2\gtrsim\abs{\xi}^2\abs{b}^2.
\end{equation}
Estimates~\eqref{betabeta}, \eqref{deltadelta}, and~\eqref{deltadeltadelta} give~\eqref{alphaalpha} in the remaining cases.

\textbf{Step 2: Construction of the free surface functions.} Given $\p{g,f,\tp{k_\ell}_{\ell=1}^m,\tp{h_\ell}_{\ell=1}^m}\in\mathcal{Y}^s$ we propose to define, via item three of Proposition~\ref{linear topological properties of container for free surface and free interface functions}, $\tp{\eta_{\ell}}_{\ell=1}^m\in\tp{\mathcal{H}^{5/2+s}\tp{\R^{n-1}}}^m$ through
\begin{equation}
    \textstyle\mathscr{F}\sb{\tp{\eta_\ell}_{\ell=1}^m}={\bf{p}_\gam}^{-1}\mathscr{F}\sb{\mathscr{K}^\gam\p{g,f,\tp{k_\ell}_{\ell=1}^m,\tp{h_\ell}_{\ell=1}^m}}.
\end{equation}
    Recall that $\mathscr{K}^\gam$ is the operator from Theorem~\ref{theorem on measurement of compatibility}. It is clear that since $\bf{p}_\gam\p{-\xi}=\Bar{\bf{p}_\gam\p{\xi}}$ (this realness assertion follows from that of $\bf{n}_\gam$ - see Propositions~\ref{diagonalization of normal stress to normal dirichlet} and~\ref{characterizations of real-valued tempered distruibutions} ) for a.e. $\xi$ then the above assignment will define a real-valued tempered distribution provided it defines a tempered distribution in the first place. For the latter to hold we need only observe that (note the use of inequality~\eqref{alphaalpha} and continuity of $\mathscr{K}^\gam$)
\begin{multline}
    \norm{(\eta_\ell)_{\ell=1}^m}_{\mathcal{H}^{5/2+s}}^2=\int_{\R^{n-1}}\tsb{\abs{\xi}^{-2}\sp{{\xi_1}^2+\abs{\xi}^4}\mathbbm{1}_{B\p{0,1}}\p{\xi}+\abs{\xi}^{5+2s}\mathbbm{1}_{\R^{n-1}\setminus B\p{0,1}}\p{\xi}}\abs{\mathscr{F}\sb{\tp{\eta_\ell}_{\ell=1}^m}\p{\xi}}^2\;\m{d}\xi\\
    \lesssim\int_{\R^{n-1}}\max\{\abs{\xi}^{-2},\abs{\xi}^{3+2s}\}\tabs{\mathscr{F}\tsb{\mathscr{K}^\gam\tp{g,f,\tp{k_\ell}_{\ell=1}^m,\tp{h_\ell}_{\ell=1}^m}}\p{\xi}}^2\;\m{d}\xi\lesssim\tnorm{\p{g,f,\tp{k_\ell}_{\ell=1}^m,\tp{h_\ell}_{\ell=1}^m}}_{\mathcal{Y}^s}.
\end{multline}
This gives~\eqref{across the universe}.

\textbf{Step 3: Modification of the data.} To show that the modified data tuple in equation~\eqref{modified data tuple 1} belongs to the range of $\Uppsi_\gam$ we use the second item in the equivalence of Theorem~\ref{isomorphism associated to overdetermined multilayer traveling stokes}; so let $F\in({_0}H^1\p{\Omega;\R^n})^\ast$ be defined as $F=\mathscr{P}(f,(k_\ell)_{\ell=1}^m)$, for $\mathscr{P}$ as in Definition~\ref{strong to weak conversion notation}. Using the definitions of $\tp{\eta_\ell}_{\ell=1}^m$ and the mappings $\mathscr{H}^\gam$,$\mathscr{I}^\gam$, and $\mathscr{K}^\gamma$ (Definition~\ref{compatibility form and overdetermined data spaces}, Proposition~\ref{measurement of compatibility}, and Theorem~\ref{theorem on measurement of compatibility}, respectively), for any $\tp{\psi_\ell}_{\ell=1}^m\in\prod_{\ell=1}^mH^{-1/2}\p{\Sigma_\ell}$ we may compute
\begin{multline}\label{a long computation}
    \mathscr{H}^\gam\tsb{\tp{g,F,\tp{h_\ell}_{\ell=1}^m},\tp{\psi_\ell}_{\ell=1}^m}=\tbr{\tp{\psi_\ell}_{\ell=1}^m,\mathscr{I}^\gam\tp{g,F,\tp{h_\ell}_{\ell=1}^m}}_{H^{-1/2},H^{1/2}}\\=\tbr{\tp{\psi_\ell}_{\ell=1}^m,\mathscr{K}^\gam\tp{g,f,(k_\ell)_{\ell=1}^m,\tp{h_\ell}_{\ell=1}^m}}_{H^{-1/2},H^{1/2}}=\int_{\R^{n-1}}\mathscr{F}\tsb{\tp{\psi_\ell}_{\ell=1}^m}\cdot\bf{p}_\gam\mathscr{F}\tsb{\tp{\eta_\ell}_{\ell=1}^m}\\=\int_{\R^{n-1}}\bf{n}_\gam\p{\xi}\mathscr{F}\tsb{\tp{\psi_\ell}_{\ell=1}^m}\p{\xi}\cdot4\pi^2\tabs{\xi}^2\bf{o}_1\mathscr{F}\tsb{\tp{\eta_\ell}_{\ell=1}^m}\p{\xi}\;\m{d}\xi
    -\int_{\R^{n-1}}\mathscr{F}\tsb{\tp{\psi_\ell}_{\ell=1}^m}\p{\xi}\cdot2\pi\ii\gam\xi_1\mathscr{F}\tsb{\tp{\eta_\ell}_{\ell=1}^m}\tp{\xi}\;\m{d}\xi\\+\int_{\R^{n-1}}\bf{n}_\gam\mathscr{F}\tsb{\tp{\psi_\ell}_{\ell=1}^m}\cdot\bf{o}_0\mathscr{F}\tsb{\tp{\eta_\ell}_{\ell=1}^m}.
\end{multline}
For the first two terms after the last equality above we may use item two of Proposition~\ref{linear topological properties of container for free surface and free interface functions} to justify the application of $\mathscr{F}$'s unitary properties; we also recall Proposition~\ref{diagonalization of normal stress to normal dirichlet} which states that $\bf{n}_\gam$ is a spectral representation of $\upnu_\gam$. Hence,
\begin{equation}\label{bird 1}
        \int_{\R^{n-1}}\bf{n}_\gam\p{\xi}\mathscr{F}\tsb{\tp{\psi_\ell}_{\ell=1}^m}\p{\xi}\cdot4\pi^2\abs{\xi}^2\bf{o}_1\mathscr{F}\tsb{\tp{\eta_\ell}_{\ell=1}^m}\p{\xi}\;\m{d}\xi=-\ssum{\ell=1}{m}\int_{\Sigma_\ell}v\cdot\sig_\ell\Delta_\|\eta_\ell e_n,
\end{equation}
for $(q,v)={\upchi_{-\gam}}^{-1}(0,\mathscr{O}(\psi_\ell e_n)_{\ell=1}^m) \in L^2\p{\Omega}\times{_0}H^1\p{\Omega;\R^n}$ the solution to the applied stress PDE in~\eqref{normal_stress pde} with data $\tp{\psi_\ell}_{\ell=1}^m$.  We also have the equality
\begin{equation}\label{bird 2}
    \int_{\R^{n-1}}\mathscr{F}\tsb{\tp{\psi_\ell}_{\ell=1}^m}\p{\xi}\cdot2\pi\ii\gam\xi_1\mathscr{F}\tsb{\tp{\eta_\ell}_{\ell=1}^m}\p{\xi}\;\m{d}\xi=\tbr{\tp{\psi_\ell}_{\ell=1}^m,\tp{\gam\pd_1\eta_\ell}_{\ell=1}^m}_{H^{-1/2},H^{1/2}}.
\end{equation}
For the final term in~\eqref{a long computation} we cannot, in general, apply that $\mathscr{F}$ is unitary directly since $\tp{\eta_\ell}_{\ell=1}^m$ need not belong to $\prod_{\ell=1}^mL^2\p{\Sigma_\ell}$. Instead we utilize the fact that $v$ is solenoidal and vanishes on $\Sigma_0$, which provides us identity~\eqref{dig a pony}.  Hence,
\begin{multline}\label{bird 3}
    \int_{\R^{n-1}}\bf{n}_\gam\mathscr{F}\tsb{\tp{\psi_\ell}_{\ell=1}^m}\cdot\bf{o}_0\mathscr{F}\tsb{\tp{\eta_\ell}_{\ell=1}^m}=-\mathfrak{g}\ssum{\ell=1}{m}\jump{\uprho}_\ell\int_{\R^{n-1}}\mathscr{F}\tsb{\m{Tr}_{\Sigma_\ell}v\cdot e_n}\cdot\mathscr{F}\tsb{\eta_\ell}\\
    =-\mathfrak{g}\ssum{\ell=1}{m}\jump{\uprho}_\ell\int_{\R^{n-1}}\int_{\p{0,a_\ell}}\mathscr{F}\sb{v\p{\cdot,y}}\p{\xi}\;\m{d}y\cdot 2\pi\ii\p{\xi,0}\mathscr{F}\sb{\eta_\ell}\p{\xi}\;\m{d}\xi=-\mathfrak{g}\ssum{\ell=1}{m}\jump{\uprho}_\ell\int_{\Omega}\grad\eta_\ell\mathbbm{1}_{\p{0,a_\ell}}\cdot v.
\end{multline}
The last equality, which is an application of Plancherel's and Fubini's theorems, is justified by the third item of Proposition~\ref{linear topological properties of container for pressure}.  Define $G\in({_0}H^1\p{\Omega;\R^n})^\ast$ through the assignment
\begin{multline}
\br{G,w}_{({_0}H^1)^\ast,{_0}H^1}=\br{F,w}_{({_0}H^1)^\ast,{_0}H^1}+\int_{\Omega}\bsb{\mathfrak{g}\ssum{\ell=1}{m}\jump{\uprho}_\ell\grad\eta_\ell\mathbbm{1}_{\p{0,a_\ell}}}\cdot w+\ssum{\ell=1}{m}\int_{\Sigma_\ell}\sig_\ell\Delta_\|\eta_\ell\cdot w\\
=\int_{\Omega}f\cdot w+\ssum{\ell=1}{m}\int_{\Sigma_\ell}k_\ell\cdot w+\int_{\Omega}\bsb{\mathfrak{g}\ssum{\ell=1}{m}\jump{\uprho}_\ell\grad\eta_\ell\mathbbm{1}_{\p{0,a_\ell}}}\cdot w+\ssum{\ell=1}{m}\int_{\Sigma_\ell}\sig_\ell\Delta_\|\eta_\ell\cdot w,\;w\in{_0}H^1\p{\Omega;\R^n}.
\end{multline}
We now synthesize identities~\eqref{a long computation}, \eqref{bird 1}, \eqref{bird 2}, and~\eqref{bird 3}:
\begin{multline}\label{up around the bend}
    \mathscr{H}^\gam\tsb{\tp{g,F,\tp{h_\ell}_{\ell=1}^m},\tp{\psi_\ell}_{\ell=1}^m}=-\ssum{\ell=1}{m}\int_{\Sigma_\ell}v\cdot\sig_\ell\Delta_\|\eta_\ell e_n-\gam\tbr{\tp{\psi_\ell}_{\ell=1}^m,\tp{\pd_1\eta_\ell}_{\ell=1}^m}_{H^{-1/2},H^{1/2}}\\-\mathfrak{g}\ssum{\ell=1}{m}\jump{\uprho}_\ell\int_{\Omega}\grad\eta_\ell\mathbbm{1}_{\p{0,a_\ell}}\cdot v=\mathscr{H}^\gam\tsb{(0,F-G,(\gam\pd_1\eta_\ell)_{\ell=1}^m),(\psi_\ell)_{\ell=1}^m}.
\end{multline}
Rearranging~\eqref{up around the bend} and using that $\mathscr{H}^\gam$ is bilinear shows that
\begin{equation}
    \mathscr{H}^\gam\tsb{\p{g,G,\tp{h_\ell-\gam\pd_1\eta_\ell}_{\ell=1}^m},\tp{\psi_\ell}_{\ell=1}^m}=0.
\end{equation}
As the above expression vanishes for all $\tp{\psi_\ell}_{\ell=1}^m\in\prod_{\ell=1}^mH^{-1/2}\p{\Sigma_\ell}$, we conclude that the modified data tuple of equation~\eqref{modified data tuple 1} belongs to $\overset{\leftarrow}{\m{ker}}\mathscr{H}^\gam$, which Theorem~\ref{isomorphism associated to overdetermined multilayer traveling stokes} establishes is the range of $\Uppsi_\gam$.  
\end{proof}

At last we are ready to state and prove an isomorphism of Banach spaces induced by the PDE~\eqref{multilayer traveling Stokes with gravity-capillary boundary and jump conditions}.

\begin{thm}[Existence and uniqueness of solutions to~\eqref{multilayer traveling Stokes with gravity-capillary boundary and jump conditions}: surface tension case]\label{isomorphism associated to multilayer traveling Stokes with gravity-capillary boundary and jump conditions in the surface tension case}
For $\gam\in\R\setminus\cb{0}$, $\upsigma=\cb{\sig_\ell}_{\ell=1}^m\subset\R^+$, and $s\in\R^+\cup\cb{0}$ the bounded linear mapping $\Upupsilon_{\gam,\upsigma}:\mathcal{X}^s\to\mathcal{Y}^s$, with action given by~\eqref{ready set action}, is an isomorphism. 
\end{thm}
\begin{proof}
Proposition~\ref{well-definedness and injectivity} ensures that this mapping is well-defined and injective, so it remains only to prove surjectivity.   Let $\p{g,f,\tp{k_\ell}_{\ell=1}^m,\tp{h_\ell}_{\ell=1}^m}\in\mathcal{Y}^s$, and define the associated tuple of free surface functions $\tp{\eta_\ell}_{\ell=1}^m\in\tp{\mathcal{H}^{5/2+s}\tp{\R^{n-1}}}^m$ via Lemma~\ref{data determines free surface - free interface functions}. Then the modified data tuple in~\eqref{modified data tuple 1} belongs to the range of $\Uppsi_\gam$. Consequentially, there exists $\p{q,u}\in H^{1+s}\p{\Omega}\times {_0}H^{2+s}\p{\Omega;\R^n}$ such that
\begin{equation}\label{use}
    \textstyle\Uppsi_\gam\p{q,u}=\tp{g,f+\mathfrak{g}\sum_{\ell=1}^m\jump{\uprho}_\ell\grad\eta_\ell\mathbbm{1}_{\p{0,a_\ell}},\tp{k_\ell+\sig_\ell\Delta_\|\eta_\ell e_n}_{\ell=1}^m,\tp{h_\ell-\gam\pd_1\eta_\ell}_{\ell=1}^m}.
\end{equation}
Set $p\in\mathcal{P}^{1+s}\p{\Omega}$ via $p=q-\mathfrak{g}\sum_{\ell=1}^m\jump{\uprho}_\ell\eta_\ell\mathbbm{1}_{\p{0,a_\ell}}$. As $p+\mathfrak{g}\sum_{\ell=1}^m\jump{\uprho}_\ell\eta_\ell\mathbbm{1}_{\p{0,a_\ell}}=q\in H^{1+s}\p{\Omega}$, we have that $\p{p,u,\tp{\eta_\ell}_{\ell=1}^m}\in\mathcal{X}^s$. We then observe that
\begin{equation}
    \tjump{S^{\upmu}(p,u)e_n}_{\ell}=\tjump{S^{\upmu}(q,u)e_n}_\ell-\mathfrak{g}\bjump{\ssum{k=1}{m}\jump{\uprho}_k\eta_k\mathbbm{1}_{(0,a_k)}}_\ell e_n=k_\ell+(\mathfrak{g}\jump{\uprho}_\ell+\sig_\ell\Delta_\|)\eta_\ell e_n.
\end{equation}
It is now straightforward to check $\Upupsilon_{\gam,\upsigma}\p{p,u,\tp{\eta_\ell}_{\ell=1}^m}=\p{g,f,\tp{k_\ell}_{\ell=1}^m,\tp{h_\ell}_{\ell=1}^m}$, which completes the proof that $\Upupsilon_{\gam,\upsigma}$ is a surjection.
\end{proof}

\subsection{Isomorphism in the case without surface tension}\label{section on isomorphism in the case without surface tension}

In this subsection we study~\eqref{multilayer traveling Stokes with gravity-capillary boundary and jump conditions} in the case of a two dimensional fluid ($n=2$) and vanishing surface tension ($\upsigma=\cb{\sig_\ell}_{\ell=1}^m=0$). Again, we first present how the free surface functions are determined from the data. In this instance the proof is simpler because item four of Proposition~\ref{linear topological properties of container for free surface and free interface functions} tells us that the function spaces holding the tuple of free surface functions are familiar Sobolev spaces.

\begin{lem}[Determination of free surface functions: case without surface tension]\label{free surface and free interface function determination in no surface tension case}
If $\gam\in\R\setminus\cb{0}$ and $\p{g,f,\tp{k_\ell}_{\ell=1}^m,\tp{h_\ell}_{\ell=1}^m}\in\mathcal{Y}^s$ for some $s\in\R^+\cup\cb{0}$, then there exists $\tp{\eta_\ell}_{\ell=1}^m\in\tp{\mathcal{H}^{5/2+s}\tp{\R^{n-1}}}^m=\tp{H^{5/2+s}\tp{\R^{n-1}}}^m$ such that the modified data tuple
\begin{multline}\label{modified data tuple 2}
    \textstyle\tp{g,f,\tp{k_\ell+\mathfrak{g}\jump{\uprho}_\ell\eta_\ell e_n}_{\ell=1}^m,\tp{h_\ell-\gam\pd_1\eta_\ell}_{\ell=1}^m}\\\textstyle\in H^{1+s}\p{\Omega}\times H^{s}\p{\Omega;\R^n}\times\prod_{\ell=1}^mH^{1/2+s}\p{\Sigma_\ell;\R^n}\times\prod_{\ell=1}^mH^{3/2+s}\p{\Sigma_\ell}    
\end{multline}
belongs to the range of $\Uppsi_\gam$, where this latter operator is from Theorem~\ref{isomorphism associated to overdetermined multilayer traveling stokes}. Moreover, we have the universal estimate
\begin{equation}\label{silence}
    \norm{\tp{\eta_\ell}_{\ell=1}^m}_{\mathcal{H}^{5/2+s}}\lesssim\norm{\p{g,f,\tp{k_\ell}_{\ell=1}^m,\tp{h_\ell}_{\ell=1}^m}}_{\mathcal{Y}^s}.
\end{equation}
\end{lem}
\begin{proof}
We again proceed in three steps as in the proof of Lemma~\ref{data determines free surface - free interface functions}.

\textbf{Step 1: Estimates and invertibility.} Again let $\bf{o}_0=\m{diag}\p{\sig_1,\dots,\sig_m} \in \R^{m \times m}$. We claim that the matrix field
\begin{equation}
    \bf{p}_\gam\p{\xi}=\bf{n}_\gam\p{\xi}^\ast\bf{o}_0-2\pi\ii\gam\xi I_{m\times m}=\bf{n}_{-\gam}\p{\xi}\bf{o}_0-2\pi\ii\gam\xi I_{m\times m} \in \C^{m \times m},
\end{equation}
where $\textbf{n}_\gam(\xi)$ is as defined in Proposition~\ref{diagonalization of normal stress to normal dirichlet} and , satisfies the estimate
\begin{equation}\label{zero}
    C^{-1}\abs{\bf{p}_\gam\p{\xi} b}^2\le\abs{\xi}^2\abs{b}^2\le C\abs{\bf{p}_\gam\p{\xi}b}^2
\end{equation}
for a.e. $\xi\in\R$ and all $b\in\C^m$, for a constant $C\in\R^+$ depending only on the physical parameters. Recall that since $n=2$, $\xi=\xi_1$.

By the first item of Theorem~\ref{finer asymptotic development of the multiplier of the normal stress to normal Dirichlet pseudodifferential operator}, there is a universal constant $C_0\in\R^+$ and a full measure set $E\subset\R$ such that if $\xi\in E$, then $\abs{\bf{n}_\gam\p{\xi}^\ast\bf{o}_0}\le C_0\min\{\abs{\xi}^2,\abs{\xi}^{-1}\}$. Thus the left inequality in~\eqref{zero} follows from the triangle inequality. Also, as a consequence of this estimate on $\bf{n}_\gam^\ast\bf{o}_0$, we learn that there are radii (depending only on $C_0$ and $\abs{\gam}$) $0<R_0<1<R_1$ such that if $\xi\in\R\setminus\sb{(-R_1,-R_0)\cup(R_0,R_1)}$, then
\begin{equation}\label{one}
    2\pi\abs{\gam}\abs{\xi}-C_0\min\{\abs{\xi}^2,\abs{\xi}^{-1}\}\ge\pi\abs{\gam}\abs{\xi}.
\end{equation}
Estimate~\eqref{one} gives the right inequality in~\eqref{zero} for $\xi\in E$ with $\abs{\xi}\le R_0$ or $\abs{\xi}\ge R_1$, by the reverse triangle inequality. For $\xi\in E\cap (-R_1,R_1)\setminus\p{-R_0,R_0}$ we use the second item of Theorem~\ref{finer asymptotic development of the multiplier of the normal stress to normal Dirichlet pseudodifferential operator}. Let $b\in\C^m$. As $\bf{o}_0$ is self-adjoint, we may estimate
\begin{equation}
    \abs{\bf{p}_\gam\p{\xi}b}\abs{b}\gtrsim|\m{Re}\sb{\bf{p}_\gam\p{\xi} b\cdot\bf{o}_0 b}|=|\m{Re}\sb{\bf{o}_0b\cdot\bf{n}_\gam\p{\xi}\bf{o}_0b}|\gtrsim\min\{\abs{\xi}^2,\abs{\xi}^{-1}\}\abs{b}^2\ge\min\{{R_0}^2,{R_1}^{-1}\}\abs{b}^2.
\end{equation}
Therefore~\eqref{zero} is shown.

\textbf{Step 2: Construction of the free surface functions.} Again using item three of Proposition~\ref{linear topological properties of container for free surface and free interface functions} we define, given $\p{g,f,\tp{k_\ell}_{\ell=1}^m,\tp{h_\ell}_{\ell=1}^m}\in\mathcal{Y}^s$, a corresponding tuple of functions $\tp{\eta_\ell}_{\ell=1}^m\in\tp{H^{5/2+s}\tp{\R^{n-1}}}^m$ via $\mathscr{F}\sb{\tp{\eta_\ell}_{\ell=1}^m}={\bf{p}_\gam}^{-1}\mathscr{F}\sb{\mathscr{K}^\gam\p{g,f,\tp{k_\ell}_{\ell=1}^m,\tp{h_\ell}_{\ell=1}^m}}$. This is well-defined thanks to the estimate
\begin{multline}
    \tnorm{(\eta_\ell)_{\ell=1}^m}^2_{\mathcal{H}^{5/2+s}}\asymp\int_{\R^{n-1}}(1+\abs{\xi}^2)^{5/2+s}\abs{\mathscr{F}\sb{\tp{\eta_\ell}_{\ell=1}^m}\p{\xi}}^2\;\m{d}\xi\\\lesssim\int_{\R^{n-1}}(1+\abs{\xi}^2)^{5/2+s}\abs{\xi}^{-2}\abs{\mathscr{F}\sb{\mathscr{K}^\gam\p{g,f,\tp{k_\ell}_{\ell=1}^m,\tp{h_\ell}_{\ell=1}^m}}\p{\xi}}^2\;\m{d}\xi\lesssim\norm{\p{g,f,\tp{k_\ell}_{\ell=1}^m,\tp{h_\ell}_{\ell=1}^m}}_{\mathcal{Y}^s}.
\end{multline}
Hence the estimate of~\eqref{silence} holds.

\textbf{Step 3: Data correction.} To show that the modified data tuple in~\eqref{modified data tuple 2} belongs to the range of $\Uppsi_\gam$ we appeal to the equivalence presented in Theorem~\ref{isomorphism associated to overdetermined multilayer traveling stokes}. Again we define $F\in({_0}H^1\p{\Omega;\R^n})^\ast$ as $F=\mathscr{P}(f,(k_\ell)_{\ell=1}^m)$, for $\mathscr{P}$ in Definition~\ref{strong to weak conversion notation}. Using the definition of $\mathscr{K}^\gam$, we compute
\begin{multline}\label{number nine dream}
    \mathscr{H}^\gam\tsb{\tp{g,F,\tp{h_\ell}_{\ell=1}^m},\tp{\psi_\ell}_{\ell=1}^m} 
    = \tbr{\tp{\psi_{\ell}}_{\ell=1}^m,\mathscr{I}^\gam\tp{g,F,\tp{h_\ell}_{\ell=1}^m}}_{H^{-1/2},H^{1/2}} \\
    =\tbr{\tp{\psi_{\ell}}_{\ell=1}^m,\mathscr{K}^\gam\tp{g,f,(k_\ell)_{\ell=1}^m,\tp{h_\ell}_{\ell=1}^m}}_{H^{-1/2},H^{1/2}}
    =\int_{\R}\mathscr{F}\tsb{\tp{\psi_\ell}_{\ell=1}^m}\cdot\bf{p}_\gam\mathscr{F}\tsb{\tp{\eta_\ell}_{\ell=1}^m} \\
    =\int_{\R}\bf{n}_\gam\mathscr{F}\tsb{\tp{\psi_\ell}_{\ell=1}^m}\cdot\bf{o}_0\mathscr{F}\tsb{\tp{\eta_\ell}_{\ell=1}^m} 
    -\int_{\R}\mathscr{F}\tsb{\tp{\psi_\ell}_{\ell=1}^m}\p{\xi}\cdot2\pi\ii\xi\mathscr{F}\tsb{\tp{\eta_\ell}_{\ell=1}^m}\p{\xi}\;\m{d}\xi,
\end{multline}
for $\tp{\psi_\ell}_{\ell=1}^m\in\prod_{\ell=1}^mH^{-1/2}\p{\Sigma_\ell}$. Denote the solution to the normal stress PDE in equation~\eqref{normal_stress pde} with data $\tp{\psi_\ell}_{\ell=1}^m$ as $    \p{q,v}={\upchi_{-\gam}}^{-1}(0,\mathscr{O}(\psi_\ell e_n)_{\ell=1}^m)\in L^2\p{\Omega}\times{_0}H^1\p{\Omega;\R^n}$ (recall that $\upchi_{-\gam}$ and $\mathscr{O}$ are defined in Proposition~\ref{weak solutions to the traveling stokes with stress boundary problem} and Definition~\ref{strong to weak conversion notation}). We next use the fact that $\mathscr{F}$ is unitary on $L^2$ to rewrite~\eqref{number nine dream} as
\begin{equation}\label{i will}
    \mathscr{H}^\gam\tsb{\p{g,F,\tp{h_\ell}_{\ell=1}^m},\tp{\psi_\ell}_{\ell=1}^m}=-\mathfrak{g}\ssum{\ell=1}{m}\jump{\uprho}_\ell\int_{\Sigma_\ell}\eta_\ell v\cdot e_n-\tbr{\tp{\psi_\ell}_{\ell=1}^m,\tp{\gam\pd_1\eta_\ell}_{\ell=1}^m}_{H^{-1/2},H^{1/2}}.
\end{equation}
Set $G\in({_0}H^1\p{\Omega;\R^n})^\ast$ via
\begin{multline}
    \br{G,w}_{({_0}H^1)^\ast,{_0}H^1}=\br{F,w}_{({_0}H^1)^\ast,{_0}H^1}+\mathfrak{g}\ssum{\ell=1}{m}\jump{\uprho}_\ell\int_{\Sigma_\ell}\eta_\ell w\cdot e_n\\
    =\int_{\Omega}f\cdot w+\ssum{\ell=1}{m}\int_{\Sigma_\ell}k_\ell\cdot w+\mathfrak{g}\ssum{\ell=1}{m}\jump{\uprho}_\ell\int_{\Sigma_\ell}\eta_\ell w\cdot e_n \text{ for }w\in{_0}H^1(\Omega;\R^n).
\end{multline}
Then~\eqref{i will} implies that $\mathscr{H}^\gam\sb{\p{g,G,\tp{h_\ell-\gam\pd_1\eta_\ell}_{\ell=1}^m,\tp{\psi_\ell}_{\ell=1}^m}}=0$ for all $\tp{\psi_\ell}_{\ell=1}^m$,  so we conclude, using Theorem~\ref{isomorphism associated to overdetermined multilayer traveling stokes}, that the modified data tuple belongs to the range of $\Uppsi_\gam$.
\end{proof}

Finally, we state and prove the analogue to Theorem~\ref{isomorphism associated to multilayer traveling Stokes with gravity-capillary boundary and jump conditions in the surface tension case}.

\begin{thm}[Existence and uniqueness of solutions to~\ref{multilayer traveling Stokes with gravity-capillary boundary and jump conditions}: case without surface tension]\label{isomorphism associated to multilayer traveling Stokes with gravity-capillary boundary and jump conditions in the zero surface tension case}
For $\gam\in\R\setminus\cb{0}$, $n=2$, and $s\in\R^+\cup\cb{0}$ the bounded linear mapping $\Upupsilon_{\gam,0}:\mathcal{X}^s\to\mathcal{Y}^s$, with action given by~\eqref{ready set action}, is an isomorphism.
\end{thm}
\begin{proof}
Proposition~\ref{well-definedness and injectivity} ensures that this mapping is well-defined and injective, so only surjectivity remains.  Let $\p{g,f,\tp{k_\ell}_{\ell=1}^m,\tp{h_\ell}_{\ell=1}^m}\in\mathcal{Y}^s$ and define the associated tuple of free surface functions $\tp{\eta_\ell}_{\ell=1}^m\in \tp{H^{5/2+s}\tp{\R^{n-1}}}^m$ via Lemma~\ref{free surface and free interface function determination in no surface tension case}. Then the modified data tuple in equation~\eqref{modified data tuple 2} belongs to the range of $\Uppsi_\gam$. Consequentially, there exists $\p{p,u}\in H^{1+s}\p{\Omega}\times{_0}H^{2+s}\p{\Omega;\R^n}$ such that
\begin{equation}
\Uppsi_\gam\p{p,u}=\tp{g,f,\tp{k_\ell+\mathfrak{g}\jump{\uprho}_\ell\eta_\ell}_{\ell=1}^m,\tp{h_\ell-\gam\pd_1\eta_\ell}_{\ell=1}^m}.
\end{equation}
By item four of both Propositions~\ref{linear topological properties of container for free surface and free interface functions} and~\ref{linear topological properties of container for pressure}, we have that $\p{p,u,\tp{\eta_{\ell=1}^m}}\in\mathcal{X}^s$. It's also clear that $\Upupsilon_{\gam,0}\p{p,u,\tp{\eta_{\ell})_{\ell=1}^m}}=\p{g,f,\tp{k_\ell}_{\ell=1}^m,\tp{h_\ell}_{\ell=1}^m}$. Hence, $\Upupsilon_{\gam,0}$ is a surjection.
\end{proof}

\section{Nonlinear analysis}\label{section on nonlinear analysis}
We now use the Banach isomorphisms constructed in the previous section to solve the fully nonlinear problems~\eqref{flattened problem at the nonlinear level} and~\eqref{equations of motion written in traveling Eulerian coordinates} for small data by way of the implicit function theorem.  The proofs of most of the results in this section essentially mirror those used in the one layer analysis of~\cite{leoni2019traveling} (except that we use our new isomorphisms), so for the sake of brevity we will mostly sketch the details.  For full details we refer to Section 8 of \cite{leoni2019traveling}.

\subsection{Preliminaries}

This subsection is dedicated to showing that the nonlinear mapping associated to the flattened PDE~\eqref{flattened problem at the nonlinear level} is both well-defined and smooth. We begin by examining the smoothness of the nonlinearities present. First we have a simple product estimate.
\begin{prop}\label{product between a standard sobolev space and a specialized sobolev space}
Let $\ell\in\cb{1,\dots,m}$ and $s\in\R^+$ with $(n-1)/2<s$. If $f\in\mathcal{H}^s\p{\R^{n-1}}$ and $g\in H^s\p{\Sigma_\ell}$, then the pointwise product satisfies the inclusion $fg\in H^s\p{\Sigma_\ell}$. Moreover the bilinear mapping $\mathcal{H}^s\times H^s\ni\p{f,g}\mapsto fg\in H^s$ is continuous and hence smooth.
\end{prop}
\begin{proof}
This is Theorem 5.8 in~\cite{leoni2019traveling}.
\end{proof}

The more complicated nonlinearities present in system~\eqref{flattened problem at the nonlinear level} are also smooth, as a consequence of the following result.
\begin{prop}\label{smoothness of the more complicated nonlinearity}
Let $s\in\R^+$ with $s>n/2$, and $m=1$ There exists a positive radius $\del\p{s}\in\R^+$ such that the following hold.
\begin{enumerate}
    \item If $\eta\in\mathcal{H}^s\tp{\R^{n-1}}$ satisfies $\norm{\eta}_{\mathcal{H}^s}<\del\p{s}$ then $\norm{\eta}_{C^0_b}<\f12$.
    \item By the first item for $\eta\in B_{\mathcal{H}^s}\tp{0,\del\p{s}}$, $w\in H^s\tp{\R^{n-1}}$, and $v\in H^{s}\p{\Omega}$ we are free to define pointwise $\Upgamma_0\p{\eta,w}=\f{w}{1+\eta}$ and $\Upgamma_1\p{\eta,v}=\f{v}{1+\eta}$. Then $\Upgamma_0\p{\eta,w}\in H^s\tp{\R^{n-1}}$ and $\Upgamma_1\p{\eta,v}=\f{v}{1+\eta}\in H^s\tp{\Omega}$, and the mappings $\Upgamma_0:B_{\mathcal{H}^s}\p{0,\del\p{s}}\times H^s\tp{\R^{n-1}}\to H^{s}\tp{\R^{n-1}}$ and $\Upgamma_1:B_{\mathcal{H}^s}\p{0,\del\p{s}}\times H^s\tp{\Omega}\to H^s\tp{\Omega}$ are smooth.
\end{enumerate}
\end{prop}
\begin{proof}
The existence of a $\del_1\p{s} \in \R^+$ for which the first item holds follows from the supercritical embedding within item 2 of Proposition~\ref{linear topological properties of container for free surface and free interface functions}.

Theorem 5.15 in~\cite{leoni2019traveling} states that for some $\ep_0\in\R^+$ the mapping $\Upgamma_2:B_{\mathcal{P}^s}(0,\ep_0)\times H^s(\Omega)\to H^s(\Omega)$, for $B_{\mathcal{P}^s}(0,\ep_0)$ the open $\ep_0$-ball of the space $\mathcal{P}^s(\Omega)$ (Definition~\ref{container space for the pressure}, for $m=1$), defined by $\Upgamma_2(\zeta,u)=\f{u}{1+\zeta}$ is smooth and well-defined. Denote the continuous (and hence smooth) inclusion mapping: $\iota:\mathcal{H}^s(\R^{n-1})\to\mathcal{P}^s(\Omega)$. There then exists $\del_2(s)\in\R^+$ such that $\iota(B_{\mathcal{H}^s}(0,\del_2(s)))\subseteq B_{\mathcal{P}^s}(0,\ep_0)$. Let $\del(s)=\min\tcb{\del_1(s),\del_2(s)}$. Then $\Upgamma_1$ is smooth as $\Upgamma_1=\Upgamma_2\circ(\iota,\m{id}_{H^s(\Omega)})$. By letting $\mathfrak{P}:H^s(\Omega)\to H^s(\R^{n-1})$ denote the smooth projection onto this closed subspace and $\tilde{\iota}:H^s(\R^{n-1})\to H^s(\Omega)$ denote this smooth inclusion we deduce that $\Upgamma_0=\mathfrak{P}\circ\Upgamma_2\circ(\iota,\tilde{\iota})$ is also smooth.
\end{proof}

The data spaces $\mathcal{Y}^s$ for which we solve the linearized flattened problem enforce the divergence compatibility condition from Proposition~\ref{divergence_divergence}. To ensure that the nonlinear mapping associated to the flattened problem has a target enforcing this condition, we require the following result.

\begin{prop}\label{nonlinear divergence compatibility}
Suppose that $s\in\R^+$ satisfies $s>n/2$, $u\in{_0}H^{2+s}\p{\Omega[0];\R^n}$, and $\tp{\eta_{\ell}}_{\ell=1}^m\subset B_{\mathcal{H}^{5/2+s}}\p{0,\del}$ for $\del=\f12\min\{a_1,a_2-a_1,\dots,a_{m}-a_{m-1}\}\del\p{5/2+s}\in\R^+$, where  $\del\p{5/2+s} \in \R^+$ is as in Proposition~\ref{smoothness of the more complicated nonlinearity}. Then for each $\ell\in\cb{1,\dots,m}$ we have the identity 
\begin{equation}
    \int_{\p{0,a_\ell}}J\mathcal{A}\grad\cdot u=u\cdot\mathcal{N}_\ell\p{\cdot,a_\ell}+(\grad_\|,0)\cdot\int_{\p{0,a_\ell}}J\mathcal{A}^{\m{t}}u,
\end{equation}
where $J$, $\mathcal{A}$, and $\mathcal{N}_\ell$ are functions of $\tp{\eta_\ell}_{\ell=1}^m$ as defined in Section~\ref{section on reformulation in a fixed domain}.
\end{prop}
\begin{proof}
Let $k\in\cb{1,\dots,\ell}$. Arguing as in Proposition 8.2 in~\cite{leoni2019traveling} we arrive at
\begin{equation}
    \int_{\p{a_{k-1},a_\ell}}J_k\mathcal{A}_k\grad\cdot u=u\cdot\mathcal{N}_\ell\p{\cdot,a_k}-u\cdot\mathcal{N}_{k-1}\p{\cdot,a_{k-1}}+\p{\grad_\|,0}\cdot\int_{\p{a_{k-1},a_k}}J_k{\mathcal{A}_k}^{\m{t}}u,
\end{equation}
where we take $\mathcal{N}_{0}=e_n$. Summing over $k\le\ell$ and using that $u$ vanishes on $\Sigma_0$ gives the result.
\end{proof}

We now arrive at our final preliminary result, which states that the nonlinear mapping associated to the flattened problem~\eqref{flattened problem at the nonlinear level} is well-defined and smooth.
\begin{thm}\label{smoothness of nonlinear operator of the flattened problem}
Let $s\in\R^+$ with $s>n/2$, $\upsigma=\cb{\sig_\ell}_{\ell=1}^m\subset\R^{+}\cup\cb{0}$, and $\kappa\in\R^+$. Define the open set 
\begin{equation}
    U^s_\kappa=\cb{\p{p,u,\tp{\eta_\ell}_{\ell=1}^m}\in\mathcal{X}^s\;:\;  \eta_\ell\in B_{\mathcal{H}^{5/2+s}}\p{0,\kappa} \text{ for } \ell\in\cb{1,\dots,m}}
\end{equation}
and the mapping $\textstyle\Upxi_{\upsigma}:\R\times\prod_{\ell=1}^mH^{1/2+s}\p{\Sigma_\ell[0];\R^{n\times n}_{\m{sym}}}\times U^s_\kappa\to\mathcal{Y}^s$ with action given via
\begin{multline}
    \Upxi_{\upsigma}\p{\gam,\tp{\mathcal{T}_\ell}_{\ell=1}^m,p,u,\tp{\eta_\ell}_{\ell=1}^m}=\Big(J\mathcal{A}\grad\cdot u,\ssum{\ell=1}{m}\mathbbm{1}_{\Omega_\ell[0]}\sb{[\rho_\ell\p{u-\gam e_1}\cdot\mathcal{A}\grad] u+(\mathcal{A}\grad)\cdot S^{\upmu}_{\mathcal{A}}\p{p,u}},\\\sp{\jump{S^{\upmu}\p{p,u}}_\ell\mathcal{N}_\ell-\p{\mathfrak{g}\jump{\uprho}_\ell\eta+\sig_\ell\m{H}\p{\eta_\ell}}\mathcal{N}_\ell-\mathcal{T}_\ell\mathcal{N}_\ell}_{\ell=1}^m,\tp{\gam\pd_1\eta_\ell+u\cdot\mathcal{N}_\ell}_{\ell=1}^m\Big),
\end{multline}
for $J$, $\mathcal{A}$, and $\mathcal{N}$ defined as functions of $\tp{\eta_\ell}_{\ell=1}^m$ as in Sections~\ref{section on Eulerian coordinate formulation} and~\ref{section on reformulation in a fixed domain}. There exists $\kappa_0\in\R^+$ such that for all $0<\kappa\le\kappa_0$ the mapping $\Upxi_{\upsigma}$ is well-defined, i.e. maps into $\mathcal{Y}^s$, and is smooth.
\end{thm}
\begin{proof}
Theorem A.10 of~\cite{leoni2019traveling} asserts that there is $\del_1\in\R^+$ for which the mean curvature operator $\m{H}:B_{\mathcal{H}^{5/2}+s}(0,\del_1)\to H^{1/2+s}(\R^{n-1})$ is well-defined and smooth. Thus we set $\kappa_0$ to be the minimum of $\del_1$ and $\del$, for $\del$ the radius from Proposition~\ref{nonlinear divergence compatibility}. By combining the analysis of the nonlinearities from Propositions~\ref{product between a standard sobolev space and a specialized sobolev space} and~\ref{smoothness of the more complicated nonlinearity} with the nonlinear divergence compatibility of Proposition~\ref{nonlinear divergence compatibility}, we may argue as in Theorem 8.3 in~\cite{leoni2019traveling} to deduce well-definedness and smoothness.
\end{proof}
\subsection{Solvability of \eqref{flattened problem at the nonlinear level} and \eqref{equations of motion written in traveling Eulerian coordinates}}

To solve~\eqref{flattened problem at the nonlinear level} we combine the smoothness result from Theorem~\ref{smoothness of nonlinear operator of the flattened problem} with the linear isomorphisms of Theorems~\ref{isomorphism associated to multilayer traveling Stokes with gravity-capillary boundary and jump conditions in the surface tension case} and~\ref{isomorphism associated to multilayer traveling Stokes with gravity-capillary boundary and jump conditions in the zero surface tension case}.
\begin{thm}\label{theorem on the solvability of the flattened problem}
Suppose that $\upsigma=\cb{\sigma_\ell}_{\ell=1}^m\subset\R^+$ and $n\ge 2$ or $\upsigma=0$ and $n=2$. Assume that $\R^+\ni s>n/2$. Then there exists open sets $\mathcal{V}_s\subset \mathcal{X}^s$ and $\mathcal{U}_s\subset\R\setminus\cb{0}\times\prod_{\ell=1}^mH^{1/2+s}\tp{\Sigma_\ell[0];\R^{n\times n}_{\m{sym}}}\times H^s\p{\Omega[0];\R^n}$ such that the following hold.
\begin{enumerate}
    \item $\p{0,0,\tp{0}_{\ell=1}^m}\in\mathcal{V}_s$ and $(\R\setminus\cb{0})\times\cb{\tp{0}_{\ell=1}^m}\times\cb{0}\subset\mathcal{U}_s$.
    \item For each $\p{\gam,\tp{\mathcal{T}_\ell}_{\ell=1}^m,f}\in\mathcal{U}_s$ there exists a unique $\p{p,u,\tp{\eta_\ell}_{\ell=1}^m}\in\mathcal{V}_s$ solving~\eqref{flattened problem at the nonlinear level} classically.
    \item The mapping $\mathcal{U}_s\ni\p{\gam,\tp{\mathcal{T}_\ell}_{\ell=1}^m,f}\mapsto\p{p,u,\tp{\eta_\ell}_{\ell=1}^m}\in\mathcal{V}_s$ is smooth.
\end{enumerate}
\end{thm}
\begin{proof}
We apply the implicit function theorem to $\Upxi_{\upsigma}$ (see, for instance, Theorem 2.5.7 in~\cite{MR960687}). Denote the Hilbert space $\mathcal{E}^s=\R\times\prod_{\ell=1}^mH^{1/2+s}\tp{\Sigma_\ell[0];\R^{n\times n}_{\m{sym}}}$. Viewing the domain of $\Upxi_{\upsigma}$ as the product $\mathcal{E}^s\times U^s_{\kappa_0}\subset\mathcal{E}^s\times\mathcal{X}^s$ we define the partial derivatives with respect to the first and second factors via
\begin{equation}
    \textstyle D_1\Upxi_{\upsigma}:\mathcal{E}^s\times U^s_\del\to\mathcal{L}\p{\mathcal{E}^s;\mathcal{Y}^s}\text{ and }D_2\Upxi_{\upsigma}:\mathcal{E}^s\times U^s_\del\to\mathcal{L}\p{\mathcal{X}^s;\mathcal{Y}^s}.
\end{equation}
For any $\gam\in\R$ we have $\Upxi_{\upsigma}\p{\gam,\tp{0}_{\ell=1}^m,0,0,\tp{0}_{\ell=1}^m}=0$ and $D_2\Upxi_{\upsigma}\p{\gam,\tp{0}_{\ell=1}^m,0,0,\tp{0}_{\ell=1}^m}=\Upupsilon_{\gam,\upsigma}$, for the latter operator as in Proposition~\ref{well-definedness and injectivity}. Theorems~\ref{smoothness of nonlinear operator of the flattened problem}, \ref{isomorphism associated to multilayer traveling Stokes with gravity-capillary boundary and jump conditions in the surface tension case}, and~\ref{isomorphism associated to multilayer traveling Stokes with gravity-capillary boundary and jump conditions in the zero surface tension case} witness the satisfaction of the implicit function theorem's hypotheses whenever $\gam\in\R\setminus\cb{0}$.

Therefore for each $\gam_{\star}\in\R\setminus\cb{0}$ there exists open sets $\mathfrak{A}\p{\gam_\star}\subset\mathcal{E}^s$, $\mathfrak{B}\p{\gam_\star}\subset U^s_{\kappa_0}$, and $\mathfrak{C}\p{\gam_\star}\subset\mathcal{Y}^s$ such that $\p{\gam_\star,\tp{0}_{\ell=1}^m}\in\mathfrak{A}\p{\gam_\star}$, $\p{0,0,\tp{0}_{\ell=1}^m}\in\mathfrak{B}\p{\gam_\star}$, and $\p{0,0,\tp{0}_{\ell=1}^m,\tp{0}_{\ell=1}^m}\in\mathfrak{C}\p{\gam_\star}$ and a smooth mapping $\upvarpi_{\gam_\star}:\mathfrak{A}\p{\gam_\star}\times\mathfrak{C}\p{\gam_\star}\to\mathfrak{B}\p{\gam_\star}$ such that
\begin{equation}\label{lithium}
    \textstyle\Upxi_{\upsigma}\p{\gam,\tp{\mathcal{T}_\ell}_{\ell=1}^m,\upvarpi_{\gam_\star}\p{\gam,\tp{\mathcal{T}_\ell}_{\ell=1}^m,g,f,\tp{k_\ell}_{\ell=1}^m,\tp{h_\ell}_{\ell=1}^m}}=\p{g,f,\tp{k_\ell}_{\ell=1}^m,\tp{h_\ell}_{\ell=1}^m}
\end{equation}
for all $\p{g,f,\tp{k_\ell}_{\ell=1}^m,\tp{h_\ell}_{\ell=1}^m}\in\mathfrak{C}\p{\gam_\star}$ and all $\p{\gam,\tp{\mathcal{T}_\ell}_{\ell=1}^m}\in\mathfrak{A}\p{\gam_\star}$. Moreover the tuple $\p{p,u,\tp{\eta_\ell}_{\ell=1}^m}=\upvarpi_{\gam_\star}\p{\gam,\tp{\mathcal{T}_\ell}_{\ell=1}^m,g,f,\tp{k_\ell}_{\ell=1}^m,\tp{h_\ell}_{\ell=1}^m}\in\mathfrak{B}\p{\gam_\star}$ is the unique solution to~\eqref{lithium} in $\mathfrak{B}\p{\gam_\star}$.

Define the open sets
\begin{multline}
    \mathfrak{C}_1(\gam_\star)=\tcb{f\;:\;(0,f,(0)_{\ell=1}^m,(0)_{\ell=1}^m)\in\mathfrak{C}(\gam_\star)}\subseteq H^s(\Omega[0];\R^n),\\\textstyle
    \mathcal{U}_s =\bigcup_{\gam\in\R\setminus\cb{0}}\mathfrak{A}\p{\gam_\star}\times\mathfrak{C}_1\p{\gam_\star}\subset\mathcal{E}^s\times H^s\p{\Omega[0];\R^n},  \text{ and }
\mathcal{V}_s=\bigcup_{\gam\in\R\setminus\cb{0}}\mathfrak{B}\p{\gam_\star}\subset U_{\kappa_0}^s.
\end{multline}
Observe that the first item is satisfied with these open sets. Define $\upvarphi:\mathcal{U}_s\to\mathcal{V}_s$ via $\upvarphi\p{\gam,\tp{\mathcal{T}_\ell}_{\ell=1}^m,f}=\upvarpi_{\gam_\star}\p{\gam,\tp{\mathcal{T}_\ell}_{\ell=1}^m,f}$ when $\p{\gam,\tp{\mathcal{T}_\ell}_{\ell=1}^m}\in\mathfrak{A}\p{\gam_\star}$ for some $\gam_\star\in\R\setminus\cb{0}$. The map $\upvarphi$ is well-defined and smooth by the previous analysis.

Taking $\p{p,u,\tp{\eta_\ell}_{\ell=1}^m}=\upvarphi\p{\gam,\tp{\mathcal{T}_\ell}_{\ell=1}^m,f}$ and noting the embeddings of the specialized Sobolev spaces (see Propositions~\ref{linear topological properties of container for free surface and free interface functions} and~\ref{linear topological properties of container for pressure}) completes the justification of the second and third items.
\end{proof}

Next we examine system~\eqref{equations of motion written in traveling Eulerian coordinates}. Our first result gives some of the mapping properties of the flattening map $\mathfrak{F}$ and its inverse from Section~\ref{section on reformulation in a fixed domain}.

\begin{prop}\label{properties of the flatteneing map and its inverse}
Let $n,k\in\N$ with $1\le n/2<k$, let $\upeta=\tp{\eta_\ell}_{\ell=1}^m\in\tp{\mathcal{H}^{5/2+k}\tp{\R^{n-1}}}^m$ be such that
\begin{equation}\label{hydrogen}
    \textstyle\max\{\norm{\eta_1}_{C^0_b},\dots,\norm{\eta_m}_{C^0_b}\}\le\f14\min\cb{a_1,a_2-a_1,\dots,a_m-a_{m-1}},
\end{equation}
and define $\mathfrak{G}:\Omega[\upeta]\to\Omega[0]$ via $\mathfrak{G}={\mathfrak{F}_\ell}^{-1}$ in the set $\Omega_\ell[\upeta]$ for each $\ell\in\cb{1,\dots,m}$ as in~\eqref{sodium}. Then the following hold.
\begin{enumerate}
    \item $\mathfrak{G}\in C^{0,1}\p{\Omega[\upeta],\Omega[0]}$ is a bi-Lipschitz homeomorphism with inverse given by  $\mathfrak{F}\in C^{0,1}\p{\Omega[0];\Omega[\upeta]}$, as defined in~\eqref{calcium}.
    \item Set $\mathfrak{G}_\ell=\mathfrak{G}\res\Bar{\Omega_\ell[\upeta]}$ for $\ell\in\cb{1,\dots,m}$. Then $\mathfrak{G}_\ell  \in C^r\p{\Omega_\ell[\upeta],\Omega_\ell[0]}$ is a diffeomorphism with inverse given by $\mathfrak{F}_\ell = \mathfrak{F} \res \Bar{\Omega_\ell[0]} \in C^r\p{\Omega_\ell[0],\Omega_\ell[\upeta]}$,  where $\N\ni r<3+k-n/2$.
    \item If $g\in{_0}H^1\p{\Omega[0]}$ then $g\circ\mathfrak{G}\in{_0}H^1\p{\Omega[\upeta]}$. Moreover there is $c\in\R^+$, independent of $g$, such that $\norm{g\circ\mathfrak{G}}_{{_0}H^1}\le c\norm{g}_{{_0}H^1}$.
    \item For $\R^+\cup\cb{0}\ni s\le k+2$, if $f\in H^s\p{\Omega[0]}$ then $f\circ\mathfrak{G}\in H^s\p{\Omega[\upeta]}$. Moreover there is $\tilde{c}\in\R^+$, independent of $f$, such that $\norm{f\circ\mathfrak{G}}_{H^s\p{\Omega[\upeta]}}\le\tilde{c}\norm{f}_{H^s\p{\Omega[0]}}$. 
\end{enumerate}
\end{prop}
\begin{proof}
By inspection, $\mathfrak{G}$ is a homeomorphism with weak derivative in $\Omega[\upeta]$ given by $\grad\mathfrak{G}\p{x,y}=\mathcal{A}^{\m{t}}\circ\mathfrak{G}$, for $\mathcal{A}$ the geometry matrix field from Section~\ref{section on reformulation in a fixed domain}. By the embedding of item two in Proposition~\ref{linear topological properties of container for free surface and free interface functions}, this weak gradient is essentially bounded. Hence $\mathfrak{G}$ is Lipschitz. A similar argument shows that $\mathfrak{F}=\mathfrak{G}^{-1}$ is also Lipschitz. Hence the first and third items are now shown.  The second and fourth items are now shown by applying the arguments of Theorem 8.4 in~\cite{leoni2019traveling} to the restrictions $\mathfrak{G}\res\Omega_\ell[\upeta]$ for $\ell\in\cb{1,\dots,m}$.
\end{proof}

Finally, we prove the solvability of the free boundary problem~\eqref{equations of motion written in traveling Eulerian coordinates}.

\begin{thm}\label{solvability of the free boundary and free interface problem}
Let $n,k\in\N$ with $1\le n/2<k$. Suppose that $\upsigma=\cb{\sigma_\ell}_{\ell=1}^m\subset\R^+$ and $n\ge 2$ or $\upsigma=0$ and $n=2$.  For all $\gam\in\R\setminus\cb{0}$, there exists $\ep\in\R^+$ such that if $\tp{\mathcal{T}_\ell}_{\ell=1}^m\in\prod_{\ell=1}^mH^{1/2+k}\p{\Sigma_\ell[0];\R^{n\times n}_{\m{sym}}}$, $f\in H^{k}\p{\Omega[0];\R^n}$, and $\sum_{\ell=1}^m\sb{\norm{\mathcal{T}_\ell}_{H^{1/2+k}}+\norm{f}_{H^s\p{\Omega_\ell[0]}}}<\ep$ then there exists a tuple of free surface functions $\upeta=\tp{\eta_\ell}_{\ell=1}^m\in\tp{\mathcal{H}^{5/2+k}\tp{\R^{n-1}}}^m$ satisfying~\eqref{hydrogen} such that the following hold. 
\begin{enumerate}
    \item If $\mathfrak{G}$ is the diffeomorphism from Proposition~\ref{properties of the flatteneing map and its inverse}, then we have the inclusion $\mathcal{F}:=f\circ\mathfrak{G}\in H^{k}\p{\Omega[\upeta];\R^n}$.
    \item There exists $\p{q,v}\in\mathcal{P}^{1+k}\p{\Omega[\upeta]}\times{_0}H^{2+k}\p{\Omega[\upeta];\R^n}$ such that $\p{q,v,\upeta}$ is a classical solution to system~\eqref{equations of motion written in traveling Eulerian coordinates} with forcing $\mathcal{F}$ and applied surface stresses $\tp{\mathcal{T}_\ell}_{\ell=1}^m$.
\end{enumerate}
\end{thm}
\begin{proof}
We argue as in the proof of Theorem 1.3 in~\cite{leoni2019traveling}. For small data we may solve the flattened problem via Theorem~\ref{theorem on the solvability of the flattened problem}. Then we obtain the associated flattening mapping via Proposition~\ref{properties of the flatteneing map and its inverse}. Finally, we pre-compose the solution to the flattened problem with the inverse of the flattening map to obtain the desired solution to the free boundary problem.
\end{proof}

\appendix
\section{tools from analysis}\label{section on tools from analysis}

This appendix records various tools and results used throughout the paper.

\subsection{Real-valued tempered distributions}\label{appendix on real valued tempered distributions}
Recall the notion of a real valued tempered distribution.

\begin{defn}[Real valued tempered distributions]\label{defn of real valued tempered distributions} We say that  $F\in(\mathscr{S}(\R^d;\C))^\ast$ is $\R$-valued if $F$ equals its complex conjugate $\Bar{F}$, where we define $\Bar{F}\in(\mathscr{S}^\ast(\R^d;\C))^\ast$ with action $\tbr{\Bar{F},\varphi}=\Bar{\tbr{F,\Bar{\varphi}}}$ for $\varphi\in\mathscr{S}(\R^d;\mathbb{C})$.  
\end{defn}

The following are useful characterizations of the $\R$-valued tempered distributions. Here we recall that the reflection operator $\updelta_{-1}$ acts on functions $f: \R^d \to \C^k$ via $\updelta_{-1}f(x) = f(-x)$ and acts on $F\in(\mathscr{S}(\R^d;\C))^\ast$ via $\br{\updelta_{-1} F,\varphi} = \br{ F,\updelta_{-1}\varphi}.$

\begin{prop}[Characterizations of real-valued tempered distributions]\label{characterizations of real-valued tempered distruibutions}
For $F\in(\mathscr{S}(
\R^d;\C))^\ast$ the following are equivalent.
\begin{enumerate}
    \item $F$ is $\R$-valued.
    \item $\Bar{\mathscr{F}[F]}=\updelta_{-1}\mathscr{F}[F]$.
    \item $F\in(\mathscr{S}(\R^d;\R))^\ast$ in the sense that $\tbr{F,\varphi}\in\R$ for all $\varphi\in\mathscr{S}(\R^d;\R)$.
\end{enumerate}
\end{prop}
\begin{proof}
The equivalence of the first and second items is standard; see, for instance, Lemma A.1 of~\cite{leoni2019traveling} for a proof. We prove that the first and third items are equivalent. Suppose first that $(3)$ holds. If $\varphi\in\mathscr{S}(\R^d;\C)$ then $\m{Re}[\varphi],\m{Im}[\varphi]\in\mathscr{S}(\R^d;\R)$; hence we are free to equate
\begin{equation}
    \tbr{\Bar{F},\varphi}=\tbr{\Bar{F},\m{Re}[\varphi]}+\ii\tbr{\Bar{F},\m{Im}[\varphi]}=\tbr{F,\m{Re}[\varphi]}+\ii\tbr{F,\m{Im}[\varphi]}=\tbr{F,\varphi}.
\end{equation}
Therefore $(3)\imp(1)$. Next suppose that $(1)$ holds. If $\varphi\in\mathscr{S}(\R^d;\mathbb{R})$ then $\varphi=\Bar{\varphi}$. Hence,
\begin{equation}
    \tbr{F,\varphi}=\tbr{\Bar{F},\varphi}=\Bar{\tbr{F,\Bar{\varphi}}}=\Bar{\tbr{F,\varphi}}\imp\tbr{F,\varphi}\in\R.
\end{equation}
Thus $(1)\imp(3)$.
\end{proof}
\begin{rmk}\label{notation for real valued tempered distributions}
By the previous proposition it is not an abuse of notation to denote the space of real valued tempered distributions with $(\mathscr{S}(\R^d;\R))^\ast$.
\end{rmk}
\begin{rmk}\label{real valued functions and real valued tempered distributions}
If $f\in(\mathscr{S}(\R^d;\C))^\ast\cap L^1_{\loc}(\R^d;\C)$ then by the third item of Proposition~\ref{characterizations of real-valued tempered distruibutions} $f$ is a $\R$-valued tempered distribution if and only if $f(x)\in\R$ for a.e. $x\in\R^d$.
\end{rmk}

\subsection{(Anti-)duality and Lax-Milgram}\label{appendix on (anti-)duality and Lax-Milgram}

Recall notions of sesquilinearity and anti-duality as defined in Section~\ref{section on conventions of notation} of the introduction.  The following variant of the Lax-Milgram lemma is adapted to anti-duality. For the well-known $\R$-valued version of this result we refer, for instance, to Corollary 5.8 of~\cite{MR2759829}.

\begin{prop}[Lax-Milgram]\label{lax-milgram lemma}
Suppose that $H$ is $\C$-Hilbert and $B:H\times H\to\C$ is a continuous and sesquilinear mapping for which there exists $c\in\R^+$ such that for all $u\in H$ one has the coercive estimate $\norm{u}^2\le c\m{Re}[B(u,u)]$. Then, there exists a $\C$-linear continuous isomorphism $\beta:H^{\Bar{\ast}}\to H$ satisfying
\begin{equation}
    B(\beta F,v)=\tbr{F,v}_{H^{\Bar{\ast}},H}\text{ for all }F\in H^{\Bar{\ast}}\text{ and }v\in H.
\end{equation}
\end{prop}
\begin{proof}
Let $K$ be the $\R$-Hilbert space with underlying vector space equal to that of $H$ and equipped with inner product $(\cdot,\cdot)_K=\m{Re}[(\cdot,\cdot)_H]$. The map $\m{Re}[B(\cdot,\cdot)]:K\times K\to\R$ is then a bilinear form satisfying the hypotheses of the $\R$-valued Lax-Milgram lemma; in other words, $\m{Re}[B(\cdot,\cdot)]$ is bounded and coercive. Thus there exists an $\R$-isomorphism $\al_0:K^\ast\to K$ such that for all $v\in K$ and all $G\in K^\ast$ we have $\m{Re}[B(\al_0G,v)]=\tbr{G,v}_{K^\ast,K}$. Let $\al_1:H^{\bar{\ast}}\to K^\ast$ be the $\R$-linear mapping defined via $\tbr{\al_1F,v}_{K^\ast,K}=\m{Re}[\tbr{F,v}_{H^{\Bar{\ast}},H}]$. Set $\be:H^{\Bar{\ast}}\to H$ via $\be=\al_0\al_1$.

By definition of $\be$ for all $F\in H^{{\Bar{\ast}}}$ and $v\in H$ we have $\m{Re}[B(\be F,v)-\tbr{F,v}_{H^{\bar{\ast}},H}]=0$. By antilinearity:
\begin{equation}
    B(\be F,v)-\tbr{F,v}_{H^{\Bar{\ast}},H}=\m{Re}[B(\be F,v)-\tbr{F,v}_{H^{\Bar{\ast}},H}]+\ii\m{Re}[B(\be F,\ii v)-\tbr{F,\ii v}_{H^{\Bar{\ast}},H}]=0.
\end{equation}
Finally $\beta$ is a $\C$-linear isomorphism as it is the inverse of the following $\C$-linear mapping: $\al_2:H\to H^{\Bar{\ast}}$ with action. $\br{\al_2v,w}_{H^{\Bar{\ast}},H}=B(v,w)$.
\end{proof}
To conclude this subsection, we review some representation formulae of (anti)-dual spaces. 
\begin{prop}[(Anti)-dual representation of Sobolev spaces]\label{(anti-)dual representation of Sobolev spaces}
Let $s\in\R$ and $\mathbb{K}\in\tcb{\R,\C}$. Recall that the $\mathbb{K}^k$-valued $L^2$-based Sobolev space on $\R^d$ of order $s$ is defined as
\begin{multline}
    H^s(\R^d;\mathbb{K}^k)=\Big\{f\in\sp{(\mathscr{S}(\R^d;\mathbb{K}))^\ast}^k\;:\;\mathscr{F}[f]\in L^1_{\loc}(\R^d;\C^k)\\\text{and }\tnorm{f}_{H^s}^2=\int_{\R^d}(1+\tabs{\xi}^2)^s\tabs{\mathscr{F}[f](\xi)}^2\;\m{d}\xi<\infty\Big\}.
\end{multline}
We have the representation formula $(H^s(\R^d;\mathbb{K}^k))^{\Bar{\ast}}=H^{-s}(\R^n;\mathbb{K}^k)$ where one may view the (anti-)dual pairing as the sesquilinear (bilinear when $\mathbb{K}=\C$) $L^2$-pairing of Fourier transforms. In other words, for all $G\in(H^s(\R^d;\mathbb{K}^k))^{\Bar{\ast}}$ there exists a unique $g\in H^{-s}(\R^d;\mathbb{K}^k)$ such that for all $f\in H^s(\R^d;\mathbb{K}^k)$ one has the equality
\begin{equation}
    \tbr{G,f}_{(H^s)^{\Bar{\ast}},H^s}=\int_{\R^d}\mathscr{F}[g]\cdot\mathscr{F}[f]=:\tbr{g,f}_{H^{-s},H^s}.
\end{equation}
Conversely if $g\in H^{-s}(\R^d;\mathbb{K}^k)$ then $f\mapsto\br{g,f}_{H^{- s},H^s}$ defines a member of $(H^s(\R^d;\mathbb{K}^k))^{\Bar{\ast}}$.
\end{prop}
\begin{proof}
    The assertions for the case $\mathbb{K}=\C$ are a consequence of the discussion after Theorem 6.3 in~\cite{MR1357411}. Suppose that $\mathbb{K}=\R$ and that $G\in(H^s(\R^d;\R^k))^{\ast}$. We may define $G_0\in(H^s(\R^d;\C^k))^{\Bar{\ast}}$ via
    \begin{equation}
        \tbr{G_0,f}_{(H^s)^{\Bar{\ast}},H^s}=\tbr{G,\m{Re}[f]}_{(H^s)^\ast,H^s}-\ii\tbr{G,\m{Im}[f]}_{(H^s)^\ast,H^s},\;f\in H^s(\R^d;\C^k).
    \end{equation}
    Applying the result for the $\C$-valued case gives us $g\in H^{-s}(\R^d;\C^k)$ such that $\tbr{G_0,f}_{(H^{-s})^{\Bar{\ast}},H^s}=\tbr{g,f}_{H^{-s},H^s}$ for all $f\in H^s(\R^d;\C^k)$. We next note that $g\in\tp{(\mathscr{S}(\R^d;\R))^\ast}^k$ by Proposition~\ref{characterizations of real-valued tempered distruibutions}: that is, if $f\in\mathscr{S}(\R^d;\R^k)\subset H^s(\R^d;\C^k)$, then
    \begin{equation}
        \tbr{g,f}_{\mathscr{S}^\ast,\mathscr{S}}=\tbr{g,\Bar{f}}_{\mathscr{S}^\ast,\mathscr{S}}=\tbr{g,f}_{H^{-s},H^s}=\tbr{G,\m{Re}[f]}_{(H^s)^\ast,H^s}-\ii\tbr{G,\m{Im}[f]}_{(H^s)^\ast,H^s}=\tbr{G,f}_{(H^s)^\ast,H^s}\in\R.
    \end{equation}
    Finally, $g$ is uniquely determined by the following argument. Suppose that $g_0\in H^{-s}(\R^d;\R^k)$ also satisfies $\tbr{G,f}_{(H^s)^\ast,H^s}=\tbr{g_0,f}_{H^{-s},H^s}$ for all $f\in H^s(\R^d;\R^k)$. Define $f_0\in H^s(\R^d;\C^k)$ via $\mathscr{F}[f_0](\xi)=(1+\abs{\xi}^2)^{-s}\mathscr{F}[g-g_0](\xi)$, $\xi\in\R^d$. Again, Proposition~\ref{characterizations of real-valued tempered distruibutions} assures us that, in fact, $f_0\in H^s(\R^d;\R^k)$. Therefore,
    \begin{equation}
        0=\tbr{g-g_0,f_0}_{H^{-s},H^s}=\int_{\R^d}(1+\abs{\xi}^2)^{-s}\tabs{\mathscr{F}\tsb{g-g_0}(\xi)}^2\;\m{d}\xi=\norm{g-g_0}_{H^{-s}}^2.
    \end{equation}
\end{proof}
\begin{rmk}
In this paper we choose identify the functional $G$ with the tempered distribution $g$ and the (anti-)duality pairing $\br{\cdot,\cdot}_{(H^s)^{\Bar{\ast}},H^s}$ with $\br{\cdot,\cdot}_{H^{-s},H^s}$.
\end{rmk}
\subsection{Fourier multipliers}\label{appendix on (tangential) multipliers}
We begin this subsection recalling the characterization of essentially bounded Fourier multipliers as $L^2$-bounded translation invariant linear mappings.
\begin{prop}\label{tilo-fourier multiplier L2 characterization}
Let $\mathbb{K}\in\tcb{\R,\C}$. The following are equivalent for a continuous linear mapping $T\in\mathcal{L}(L^2(\R^d;\mathbb{K}))$.
\begin{enumerate}
    \item $T$ commutes with all translation operators.
    \item There exists $m\in L^\infty(\R^d;\C)$ such that $Tf=\mathscr{F}^{-1}[m\mathscr{F}[f]]$ for all $f\in L^2(\R^d;\mathbb{K})$.  
\end{enumerate}
In either case $\norm{T}_{\mathcal{L}(L^2)}\le\norm{m}_{L^\infty}\le2\norm{T}_{\mathcal{L}(L^2)}$, and if $\mathbb{K}=\R$ then $\Bar{m(\xi)}=m(-\xi)$ for a.e. $\xi\in\R^d$.
\end{prop}
\begin{proof}
    The case $\mathbb{K}=\C$ is handled in Theorem 2.5.10 of~\cite{MR3243734}. It remains to handle the case $\mathbb{K}=\R$. The implication $(2)\imp(1)$ is clear. On the other hand, given a translation invariant $T\in\mathcal{L}(L^2(\R^d;\R))$ define $T_0\in\mathcal{L}(L^2(\R^d;\C))$ via $T_0f=T\m{Re}[f]+\ii T\m{Im}[f]$ for $f\in L^2(\R^d;\C)$. $T_0$ is also translation invariant. Hence by the $\C$-valued case there is $m\in L^\infty(\R^n;\C)$ such that the action of $T_0$ is given by multiplication of $m$ in frequency space. Using Proposition~\ref{characterizations of real-valued tempered distruibutions} and Remark~\ref{real valued functions and real valued tempered distributions} we compute:
    \begin{equation}
        \Bar{m\mathscr{F}[f]}=\mathscr{F}[T\m{Re}[f]]-\ii\mathscr{F}[T\m{Im}[f]]=\updelta_{-1}\mathscr{F}[T\m{Re}[\Bar{f}]]+\ii\updelta_{-1}\mathscr{F}[T\m{Im}[\Bar{f}]]=\updelta_{-1}m\updelta_{-1}\mathscr{F}[\Bar{f}]=\updelta_{-1}m\Bar{\mathscr{F}[f]}.
    \end{equation}
    The above equality holds for all $f\in L^2(\R^n;\C)$ and so we deduce that $\Bar{m}=\updelta_{-1}m$ almost everywhere.
    \end{proof}
    We may also generalize the previous theorem to a characterization of continuous, linear, and translation invariant mappings between the $L^2$-based fractional Sobolev spaces. First we recall the Bessel potential.

    \begin{defn}[Bessel potential]\label{definition of Bessel potential}
    For $s\in\R$ we define the Bessel Potential of order $s$ as the operator $\J^s\in\mathcal{L}((\mathscr{S}(\R^d;\C))^\ast)$ defined via $\J^sF=\mathscr{F}^{-1}[(1+\tabs{\cdot}^2)^{s/2}\mathscr{F}[F]]$ for $F\in(\mathscr{S}(\R^d;\C))^\ast$.  Thanks to Proposition~\ref{characterizations of real-valued tempered distruibutions}, $\J^sF$ is $\R$-valued whenever this true of $F$. We also recall that for any $t\in\R$ and $\mathbb{K}\in\tcb{\R;\C}$ $\J^s\in\mathcal{L}(H^{t+s}(\R^d;\mathbb{K});H^t(\R^d;\mathbb{K}))$ is an isometric isomorphism.
    \end{defn}
    \begin{prop}\label{tilo-fourier multiplier Sobolev space characterization}
    Let $\mathbb{K}\in\tcb{\R,\C}$, $s,t\in\R$. The following are equivalent for a continuous linear mapping $T\in\mathcal{L}\tp{H^s(\R^d;\mathbb{K});H^t(\R^d;\mathbb{K})}$.
    \begin{enumerate}
        \item $T$ commutes with all translation operators.
        \item There exists a measurable function $\mu:\R^d\to\C$ such that
        \begin{equation}\label{equation central station}
            \m{m}_\mu[r]:=\m{esssup}\tcb{(1+\abs{\xi}^2)^{r/2}\abs{\mu(\xi)}\;:\;\xi\in\R^d}\in[0,\infty]
        \end{equation}
        is finite for $r=t-s$ and for all $f\in H^s(\R^d;\mathbb{K})$ one has $Tf=\mathscr{F}^{-1}[\mu\mathscr{F}[f]]$.
    \end{enumerate}
    In either case $\norm{T}_{\mathcal{L}(H^s;H^t)}\le\m{m}_\mu[t-s]\le 2\norm{T}_{\mathcal{L}(H^s;H^t)}$, and if $\mathbb{K}=\R$ then $\updelta_{-1}\mu=\Bar{\mu}$,
    \end{prop}
    \begin{proof}
        Suppose that the first item holds. Using the Bessel potentials from the previous definition, we obtain the bounded and translation-invariant $L^2$-operator $T_0 := \J^{t}T\J^{-s}$.  Applying Proposition~\ref{tilo-fourier multiplier L2 characterization} grants us $\omega\in L^\infty\tp{\R^d;\mathbb{K}}$ such that if $\mathbb{K}=\R$ then $\updelta_{-1}\omega=\Bar{\omega}$ and $T_0F=\mathscr{F}^{-1}[\omega\mathscr{F}[F]]$ for $F\in L^2(\R^d;\mathbb{K})$. Set $\mu\p{\xi}=(1+\abs{\xi}^2)^{(s-t)/2}\omega(\xi)$. We check that $\mu$ is the desired spectral representation of $T$:
        \begin{equation}
            TF=\J^{-t}T_0\J^sF=\J^{-t}T_0\mathscr{F}^{-1}[(1+\abs{\cdot})^{s/2}\mathscr{F}[F]]=\mathscr{F}^{-1}[(1+\abs{\cdot}^2)^{-t/2}\omega(\xi)(1+\abs{\cdot}^2)^{s/2}\mathscr{F}[F]]=\mathscr{F}^{-1}[\mu\mathscr{F}[F]],
        \end{equation}
        for $F\in H^s(\R^d;\mathbb{K})$. Using once more Proposition~\ref{tilo-fourier multiplier L2 characterization}, we arrive at the bounds
        \begin{equation}
            \m{m}_\mu[t-s]=\norm{\omega}_{L^\infty}\asymp\norm{T_0}_{\mathcal{L}(L^2)}=\norm{T}_{\mathcal{L}(H^s;H^t)}.
        \end{equation}
        Thus the forward direction is shown. The reverse implication is proved in a similar manner.
    \end{proof}
    
    We now set notation for $L^2$-bounded translation invariant mappings, emphasizing that this space is parameterized by essentially bounded functions.
    
    \begin{defn}[Tangential Fourier multipliers I]\label{definition of tangential multipliers}
    Let $\mathbb{K}\in\tcb{\R,\C}$ and let $\omega\in L^\infty(\R^d;\C^{k\times k})$ be a multiplier such that if $\mathbb{K}=\R$ then $\Bar{\omega}=\updelta_{-1}\omega$. We make the following definitions.
    \begin{enumerate}
        \item $M_\omega\in\mathcal{L}(L^2(\R^d;\mathbb{K}^k))$ is defined via $M_{\omega}f=\mathscr{F}^{-1}[\omega\mathscr{F}[f]]$ for $f\in L^2(\R^d;\mathbb{K}^k)$.
        \item If $a<b$ are real numbers we set $U=\R^d\times(a,b)$ and extend $M_\omega$ to be a member of $L^2(U;\mathbb{K}^k)$ via $M_\omega f(\cdot,y)=\mathscr{F}^{-1}[\omega\mathscr{F}[f(\cdot,y)]]$ for $y\in(a,b)$ and $f\in L^2(U;\mathbb{K}^k)$.
    \end{enumerate}
    \end{defn}
    
    We would like to further extend the definitions of $M_\omega$ to the spaces $H^s(\R^d\times(a,b);\mathbb{K}^k)$ for $s\in\R^+$ and $({_0}H^1(\R^d\times(a,b);\mathbb{K}^k))^{\Bar{\ast}}$ and study their boundedness properties. To do this we need the following preliminary estimates.

    \begin{lem}\label{lemma on fractional sobolev boundedness of tangential multipliers in the stip}
    Let $s\in\R^+\cup\cb{0}$, $\mathbb{K}$, and $\omega$ be as in Definition~\ref{definition of tangential multipliers}, $a<b$ be real, and $U=\R^d\times(a,b)$. Then the following hold.
    \begin{enumerate}
        \item If $f\in H^s(U;\mathbb{K}^k)$ then $M_\omega f\in H^s(U;\mathbb{K}^k)$ and $\tnorm{M_\omega f}_{H^s}\le c_0\tnorm{\omega}_{L^\infty}\tnorm{f}_{H^s}$ for a constant $c_0\in\R^+$ depending only on $a$, $b$, $d$, $s$. Moreover if $s>1/2$ then for $z\in[a,b]$ we have
        \begin{equation}
            \m{Tr}_{\R^d\times\tcb{z}}M_\omega f=M_\omega\m{Tr}_{\R^d\times\tcb{z}}f.
        \end{equation}
        \item If $f\in H^s(U;\mathbb{K}^k)$, then setting $\J^sf(\cdot,y):=\mathscr{F}^{-1}[(1+\abs{\cdot}^2)^{s/2}\mathscr{F}[f(\cdot,y)]]$ for $y\in(a,b)$ defines an $L^2$-function $\J^sf\in L^2(U;\mathbb{K}^k)$, and there is a constant $c_1\in\R^+$, dependent only upon $a$, $b$, $d$, and $s$, for which $\tnorm{\J^sf}_{L^2}\le c_1\tnorm{f}_{H^s}$.
    \end{enumerate}
    \end{lem}
    \begin{proof}
        The first item follows from interpolation, the fact that $M_\omega$ commutes with distributional derivatives, and Proposition~\ref{tilo-fourier multiplier L2 characterization}. The second item follows from Corollary A.6 of~\cite{leoni2019traveling}.
    \end{proof}
    
    By the first item of the previous lemma, we may extend the definition of tangential Fourier multipliers in the following way.
    
    \begin{defn}[Tangential Fourier multipliers II]\label{definition of tangential fourier multipliers 2}
    Let $\mathbb{K}$ and $\omega$ be as in Definition~\ref{definition of tangential multipliers}, $a<b$ be real, and $U=\R^d\times (a,b)$. If $F\in(_0H^1(U;\mathbb{K}^k))^{\Bar{\ast}}$ we define $M_\omega F\in({_0}H^1(U;\mathbb{K}^k))^{\Bar{\ast}}$ to be the (anti-)linear functional with action on $\varphi\in{_0}H^1(U;\mathbb{K}^k)$ given by
    \begin{equation}
        \tbr{M_{\omega}F,\varphi}_{({_0}H^1)^{\Bar{\ast}},{_0}H^1}=\tbr{F,M_{\Bar{\omega}}\varphi}_{({_0}H^1)^{\Bar{\ast}},{_0}H^1}.
    \end{equation}
    Thanks to the first item of Lemma~\ref{lemma on fractional sobolev boundedness of tangential multipliers in the stip}, $M_\omega F$ is well-defined and $\tnorm{M_{\omega}F}_{({_0}H^1)^{\Bar{\ast}}}\lesssim\norm{\omega}_{L^\infty}\norm{F}_{({_0}H^1)^{\Bar{\ast}}}$.
    \end{defn}
    
    Finally, we arrive at the principal result of this subsection.

    \begin{prop}\label{culmination of the multiplier appendix}
    Let $s\in\R^+\cup\cb{0}$, $\mathbb{K}\in\tcb{\R,\C}$, and $\omega\in L^\infty\tp{\R^d;\C^{k\times k}}$ be such that if $\mathbb{K}=\R$ then $\Bar{\omega}=\updelta_{-1}\omega$.  Let $m\in\N^+$ and $\cb{a_\ell}_{\ell=1}^m\subset\R^+$ be such that $a_0=0<a_1<\cdots<a_m$, and set $U_\ell=\R^d\times(a_{\ell-1},a_{\ell})$, $\ell\in\tcb{1,\dots,m}$, $U=\bigcup_{\ell=1}^mU_\ell$, and $W=(\Bar{U})\oo$.     Then there exists a constant $c_2\in\R^+$, independent of $\omega$, for which the following estimates hold, where  $\m{m}_{\omega}[\cdot]$ is as in~\eqref{equation central station}. 
    \begin{enumerate}
        \item If $g\in H^{1+s}(U;\mathbb{K}^k)$, then $M_{\omega}g\in L^2(U;\mathbb{K}^k)$ and 
        \begin{equation}
            \tnorm{M_{\omega}g}_{L^2}\le c_2\m{m}_{\omega}[1+s]\ssum{\ell=1}{m}\tnorm{g}_{H^{1+s}(U_\ell)}.
        \end{equation}
        \item If $f\in H^s(U;\mathbb{K}^k)$, $\tp{k_\ell}_{\ell=1}^m\in\prod_{\ell=1}^mH^{1/2+s}(\R^d\times\tcb{a_\ell};\mathbb{K}^k)$, and we define $F\in({_0}H^1(W;\mathbb{K}^k))^{\Bar{\ast}}$ via $\tbr{F,\varphi}_{({_0}H^1)^{\Bar{\ast}},{_0}H^1}=\int_{U}f\cdot\varphi+\sum_{\ell=1}^m\int_{\R^d\times\tcb{a_\ell}}k_\ell\cdot \varphi$, then $M_{\omega}F\in({_0}H^1(W;\mathbb{K}^k))^{\Bar{\ast}}$ with
        \begin{equation}
            \tnorm{M_{\omega}F}_{({_0}H^1)^{\Bar{\ast}}}\le c_2\m{m}_{\omega}[1+s]\ssum{\ell=1}{m}\tsb{\tnorm{f}_{H^s(U_\ell)}+\norm{k_\ell}_{H^{1/2+s}}}.
        \end{equation}
    \end{enumerate}
    \end{prop}
    \begin{proof}
        For the first item we use the second assertion of Lemma~\ref{lemma on fractional sobolev boundedness of tangential multipliers in the stip} to bound
        \begin{equation}
            \tnorm{M_{\omega}g}_{L^2}^2\le {\m{m}_{\omega}[1+s]}^2\ssum{\ell=1}{m}\tnorm{\J^{1+s}g}^2_{L^2(\Omega_\ell)}\le {c_1}^2{\m{m}_{\omega}[1+s]}^2\ssum{\ell=1}{m}\norm{g}_{H^{1+s}(U_\ell)}^2.
        \end{equation}
        We next prove the second item. Suppose that $\varphi\in{_0}H^1(W;\mathbb{K}^k)$. If $\ell\in\cb{1,\dots,m}$ then by trace theory and the first assertion of Lemma~\ref{lemma on fractional sobolev boundedness of tangential multipliers in the stip},
        \begin{multline}\label{eqquat1}
            \babs{\int_{\R^d\times\tcb{a_\ell}}k_\ell M_{\Bar{\omega}}\varphi}=\tabs{\tbr{\J^{1+s}k_\ell,\J^{-1-s}M_{\Bar{\omega}}\varphi(\cdot,a_\ell)}_{H^{-1/2},H^{1/2}}}\le\tnorm{k_\ell}_{H^{1/2+s}}\norm{M_{\varpi}\varphi(\cdot,a_\ell)}_{H^{1/2}}\\
            \le\tilde{c}\tnorm{k_\ell}_{1/2+s}\norm{M_{\varpi}\varphi}_{{_0}H^1}\le\tilde{c}c_0\m{m}_{\omega}[1+s]\tnorm{k_\ell}_{H^{1+s}}\tnorm{\varphi}_{{_0}H^{1}},
        \end{multline}
        for $\tilde{c}\in\R^+$ a constant from trace theory and the auxiliary multiplier $\varpi(\xi)=(1+\abs{\xi}^2)^{-(s+1)/2}\Bar{\omega(\xi)}$, $\xi\in\R^d$. By again Lemma~\ref{lemma on fractional sobolev boundedness of tangential multipliers in the stip} item one, we finally estimate
        \begin{equation}\label{eqquat2}
            \babs{\int_{U}f\cdot M_{\Bar{\omega}}\varphi}=\babs{\int_{U}\J^sf\cdot\J^{-s}M_{\Bar{\omega}}\varphi}\le c_1\ssum{\ell=1}{m}\tnorm{f}_{H^s(U_\ell)}\tnorm{M_{\varpi}\J^1\varphi}_{L^2}\lesssim \m{m}_{\omega}[1+s]\norm{\varphi}_{{_0}H^1}\ssum{\ell=1}{m}\tnorm{f}_{H^s(U_\ell)}.
        \end{equation}
        Combining~\eqref{eqquat1} and~\eqref{eqquat2} gives the second item.
    \end{proof}
\subsection{Korn's inequality}\label{appendix on Korn's inequality}
We record a version of Korn's inequality stating that the $L^2$-norm of the symmetrized gradient controls the $H^1$ norm on the closed subspace of functions vanishing on the lower boundary.
\begin{prop}\label{Korn's inequality}
Let $a,b \in \R$ with $a<b$. Then there exists a constant $c\in\R^+$, depending only on $b-a$ and $n$, such that for all $f\in H^1(\R^{n-1}\times(a,b);\R^n)$ such that $\m{Tr}_{\R^{n-1}\times\cb{a}}f=0$ we have the inequality:
\begin{equation}
    \norm{f}_{L^2}+\norm{\grad f}_{L^2}\le c\norm{\mathbb{D}f}_{L^2}.
\end{equation}
\end{prop}
\begin{proof}
We refer the reader to the proof of Lemma 2.7 of~\cite{MR611750}.
\end{proof}

\bibliographystyle{alpha}
\bibliography{abbreviated_layers.bib}
\end{document}